\documentclass[graybox]{svmult}

\usepackage{mathptmx}       
\usepackage{helvet}         
\usepackage{courier}        
\usepackage{type1cm}        
\usepackage{makeidx}         
\usepackage{graphicx}        
\usepackage{multicol}        
\usepackage[bottom]{footmisc}

\usepackage{amsmath,amscd,amssymb} 
\usepackage[all]{xy}
\usepackage{color}

\DeclareSymbolFont{cmlargesymbols}{OMX}{cmex}{m}{n}
\DeclareMathSymbol{\mycoprod}{\mathop}{cmlargesymbols}{"60}
\let\coprod\mycoprod

\def\subrel#1#2{\mathrel{\mathop{#2}\limits_{#1}}}

\def\ssp{\Sigma{\rm Sp}}
\def\isp{{\rm Sp}}

\def\bmod{{\rm BMod}}
\def\comod{{\rm Comod}}
\def\lcomod{{\rm LComod}}
\def\rcomod{{\rm RComod}}
\def\lmod{{\rm LMod}}
\def\rmod{{\rm RMod}}
\def\Ho{{\rm Ho}}
\def\sp{{\rm Sp}}
\def\ab{{\rm Ab}_*}

\makeindex             

\begin{document}

\title*{On quasi-categories of comodules and Landweber exactness}
\author{Takeshi Torii} 
\institute{Takeshi Torii \at 
Department of Mathematics, 
Okayama University,
Okayama 700--8530, Japan,\\
\email{torii@math.okayama-u.ac.jp}}
%
%
\maketitle

\abstract*{
In this paper we study quasi-categories of
comodules over coalgebras in a stable homotopy theory.
We show that the quasi-category of comodules
over the coalgebra associated to a Landweber exact
$\mathbb{S}$-algebra depends only on the height of
the associated formal group.
We also show that the quasi-category of $E(n)$-local spectra
is equivalent to the quasi-category of comodules
over the coalgebra $A\otimes A$
for any Landweber exact $\mathbb{S}_{(p)}$-algebra
of height $n$ at a prime $p$.
Furthermore,
we show that
the category of module objects over a discrete
model of the Morava $E$-theory spectrum
in the $K(n)$-local discrete symmetric $\mathbb{G}_n$-spectra
is a model of the $K(n)$-local category, 
where $\mathbb{G}_n$ is the extended Morava stabilizer group.
}

\abstract{
In this paper we study quasi-categories of
comodules over coalgebras in a stable homotopy theory.
We show that the quasi-category of comodules
over the coalgebra associated to a Landweber exact
$\mathbb{S}$-algebra depends only on the height of
the associated formal group.
We also show that the quasi-category of $E(n)$-local spectra
is equivalent to the quasi-category of comodules
over the coalgebra $A\otimes A$
for any Landweber exact $\mathbb{S}_{(p)}$-algebra
of height $n$ at a prime $p$.
Furthermore,
we show that
the category of module objects over a discrete
model of the Morava $E$-theory spectrum
in the $K(n)$-local discrete symmetric $\mathbb{G}_n$-spectra
is a model of the $K(n)$-local category, 
where $\mathbb{G}_n$ is the extended Morava stabilizer group.
}

\section{Introduction}

It is known that the stable homotopy category
of spectra is intimately related to
the theory of formal groups through
complex cobordism and 
Adams-Novikov spectral sequence
by the works of Morava~\cite{Morava}, 
Miller-Ravenel-Wilson~\cite{MRW},
Devinatz-Hopkins-Smith~\cite{DHS}, 
Hopkins-Smith~\cite{Hopkins-Smith},
Hovey-Strickland~\cite{Hovey-Strickland0} 
and many others.
The $E_2$-page of the Adams-Novikov spectral sequence 
is described as the derived functor of
taking primitives 
in the abelian category of
graded comodules over the co-operation
Hopf algebroid
associated to the complex cobordism spectrum.

We also have a localized version of Adams-Novikov
spectral sequence. 
For example,
for a Landweber exact spectrum $E$
of height $n$ at a prime $p$,
we have an $E$-based Adams-Novikov spectral sequence
abutting to the homotopy groups of $E$-local
spectra.
In this case the $E$-localization 
and the $E_2$-page
of the $E$-based Adams-Novikov spectral sequence
depends only on the height $n$ of the associated
formal group at $p$.
There are many results that
the derived functor describing the
$E_2$-page of the $E$-based Adams-Novikov
spectral sequence depends only on
the substack of the moduli stack of formal groups
\cite{Hovey2},
\cite{Hovey-Sadofsky},
\cite{Hovey-Strickland},
\cite{Naumann}.

These results suggest that
there may be an intimate relationship between
localized quasi-categories of spectra
and quasi-categories of comodules over co-operation
coalgebras.
In this paper we investigate this relationship.
We show that
the quasi-category of comodules over a coalgebra associated to
a Landweber exact $\mathbb{S}$-algebra depends only on the height
of the associated formal group and
that the quasi-category of comodules over a coalgebra
associated to a Landweber exact $\mathbb{S}_{(p)}$-algebra of height $n$
at a prime $p$ is equivalent to the quasi-category
of $E(n)$-local spectra,
where $E(n)$ is the $n$th Johnson-Wilson spectrum
at $p$.

First, we introduce a quasi-category of comodules
over a coalgebra associated to an algebra object
of a stable homotopy theory $\mathcal{C}$.
In this paper we regard coalgebra objects
as algebra objects of the opposite
monoidal quasi-category of $A$-$A$-bimodule objects
for an algebra object $A$ of $\mathcal{C}$.
We regard comodule objects over a coalgebra $\Gamma$ 
as module objects over $\Gamma$
in the opposite quasi-category of $A$-module objects
in $\mathcal{C}$.
In particular, we show that $A\otimes A$ is a coalgebra
object for an algebra object $A$ of $\mathcal{C}$
and we can consider the quasi-category
\[ \lcomod_{\Gamma(A)}(\mathcal{C})\]
of left comodules
over $A\otimes A$ in $\mathcal{C}$,
where $\Gamma(A)$ represents the pair $(A,A\otimes A)$.
For a map $A\to B$ of algebra objects of $\mathcal{C}$,
we have the extension of scalars functor
$B\otimes_A(-): \lmod_A(\mathcal{C})\to\lmod_B(-)$,
where $\lmod_A(\mathcal{C})$ and
$\lmod_B(\mathcal{C})$ are
the quasi-categories of left $A$-modules and $B$-modules,
respectively.
We show that the extension of scalars functor
extends to a functor
\[ B\otimes_A(-):
    \lcomod_{\Gamma(A)}(\mathcal{C})\longrightarrow
    \lcomod_{\Gamma(B)}(\mathcal{C}).\]  
of quasi-categories of comodules.

Next, we consider Landweber exact $\mathbb{S}$-algebras
in the quasi-category of spectra $\sp$,
where $\mathbb{S}$ is the sphere spectrum.    
We show that, if $A$ is a Landweber exact
$\mathbb{S}$-algebra,
then the quasi-category of comodules over the coalgebra
$A\otimes A$ depends only on the height of 
the associated formal group.

\begin{theorem}
[{cf. Theorem~\ref{thm:equiv-comodule-at-p}}]
If $A$ and $B$ are Landweber exact $\mathbb{S}$-algebras
with the same height at all primes $p$,
then there is an equivalence of quasi-categories
\[ \lcomod_{\Gamma(A)}(\sp)\simeq
   \lcomod_{\Gamma(B)}(\sp).\]
\end{theorem}

We also show that the quasi-category of comodules over $A\otimes A$
is equivalent to the quasi-category $L_n\sp$
of $E(n)$-local spectra
if $A$ is a Landweber exact $\mathbb{S}_{(p)}$-algebra of height $n$
at a primes $p$.

\begin{theorem}
[{cf.~Theorem~\ref{thm:Ln-equivalent-comodules}}]
If $A$ is a Landweber exact $\mathbb{S}_{(p)}$-algebra
of height $n$ at a prime $p$,
then there is an equivalence of quasi-categories
\[ L_n\sp\simeq\lcomod_{\Gamma(A)}(\sp).\]
\end{theorem}

As an application of the results in this paper
we show that the model category constructed in \cite{Torii1}
is a model of the $K(n)$-local category,
where $K(n)$ is the $n$th Morava $K$-theory spectrum at a prime $p$.
We denote by $\Sigma\sp$ the model category of symmetric spectra
and by $\Sigma\sp_{K(n)}$ its left Bousfield localization of
$\Sigma\sp$ with respect to $K(n)$.
The $n$th extended Morava stabilizer group $\mathbb{G}_n$
is a profinite group and we can consider the model category
$\Sigma\sp(\mathbb{G}_n)$ of discrete symmetric
$\mathbb{G}_n$-spectra and its Bousfield localization
$\Sigma\sp(\mathbb{G}_n)_{K(n)}$ with respect to $K(n)$.
We have a commutative monoid object
$F_n$ in $\Sigma\sp(\mathbb{G}_n)_{K(n)}$ constructed by
Davis~\cite{Davis} and Behrens-Davis~\cite{Behrens-Davis},
which is a discrete model of the $n$th
Morava $E$-theory spectrum $E_n$.
In \cite{Torii1} we showed that the extension of scalars
functor 
\[ L_{K(n)}(F_n\otimes(-)):
    \Sigma\sp(\mathbb{G}_n)_{K(n)}
    \longrightarrow \lmod_{F_n}(\Sigma\sp(\mathbb{G}_n)_{K(n)}),  \]
which is a left Quillen functor,
is homotopically fully faithful,
that is, it induces a weak homotopy equivalence
between mapping spaces for any two objects
in $\Sigma\sp(\mathbb{G}_n)_{K(n)}$.
In this paper we show that this functor is actually
a left Quillen equivalence and hence we can consider
the category $\lmod_{F_n}(\Sigma\sp(\mathbb{G}_n)_{K(n)})$
is a model of the $K(n)$-local category.

\begin{theorem}
[{cf.~Theorem~\ref{thm:model-K(n)-local-category}}]
The extension of scalars
functor
\[ L_{K(n)}(F_n\otimes(-)):
    \Sigma\sp(\mathbb{G}_n)_{K(n)}
    \longrightarrow \lmod_{F_n}(\Sigma\sp(\mathbb{G}_n)_{K(n)})  \]
is a left Quillen equivalence.
\end{theorem}

The organization of this paper is as follows:
In \S\ref{sec:Notation}
we fix some notation we use throughout this paper.
In \S\ref{sec:opposite_monoidal}
we study opposite coCartesian fibrations,
opposite monoidal quasi-categories, 
and opposite tensored quasi-categories.
In particular,
we show that a lax monoidal right adjoint functor
between monoidal quasi-categories
induces a lax monoidal right adjoint functor
between the opposite monoidal quasi-categories.
In\S\ref{sec:quasi-categories-comodules}
we introduce a quasi-category
of comodules over a coalgebra in a stable homotopy theory.
We define a cotensor product
of a right comodule and a left comodule
over a coalgebra as a limit of
the cobar construction.
We study the relationship between 
localizations of a stable homotopy theory
and quasi-categories of comodules.
In \S\ref{sec:comodule-spectra}
we study comodules in spectra over a coalgebra
associated to a Landweber exact $\mathbb{S}$-algebra.
First, we study the Bousfield-Kan spectral sequence
associated to the two-sided cobar construction.
Next,
we show that the quasi-category of comodules
over the coalgebra associated to a Landweber
exact $\mathbb{S}$-algebra depends only on the height
of the associated formal group.
Finally,
we show that the model category
of modules over $F_n$
in the $K(n)$-local discrete symmetric $\mathbb{G}_n$-spectra 
is a model of the $K(n)$-local
category.
In \S\ref{sec:proof-theorem-construction-RF}
we give a proof of 
Proposition~\ref{prop:existence-final-object-sections}
stated in \S\ref{sec:opposite_monoidal},
which is technical but 
important for constructing a canonical map between
opposite coCartesian fibrations.

\section{Notation}
\label{sec:Notation}

For a category $\mathcal{C}$,
we denote by 
${\rm Hom}_{\mathcal{C}}(x,y)$
the set of all morphisms 
from $x$ to $y$ in $\mathcal{C}$
for $x,y\in\mathcal{C}$.

We denote by ${\rm sSet}$
the category of simplicial sets.
For a simplicial set $K$,
we denote by $K^{\rm op}$
the opposite simplicial set
(see \cite[\S1.2.1]{Lurie1}).  
If $K$ is a quasi-category,
then $K^{\rm op}$ is also a quasi-category.
For simplicial sets $X,Y$,
we denote by ${\rm Fun}(X,Y)$
the simplicial mapping space from $X$ to $Y$.
For a simplicial set $X$ equipped with a map
$\pi: X\to S$ of simplicial sets,
we denote by $X_s$ the fiber of $\pi$ over $s\in S$.
If $X$ and $Y$ are simplicial sets over a simplicial set $S$,
then we denote by
${\rm Fun}_S(X,Y)$ the simplicial set of maps
from $X$ to $Y$ over $S$.

For a small (simplicial) category $\mathcal{C}$,
we denote by $N(\mathcal{C})$
the simplicial set obtained by
applying the (simplicial) nerve functor $N(-)$ to $\mathcal{C}$
(see \cite[\S1.1.5]{Lurie1}).
We denote by ${\rm Cat}_{\infty}$
the quasi-category of (small) quasi-categories
(see \cite[\S3]{Lurie1}).

We denote by $\Sigma\sp$ 
the category of symmetric spectra
equipped with the stable model structure
(see \cite{HSS}).
We denote by ${\rm Sp}$
the quasi-category of spectra,
which is the underlying quasi-category
of the simplicial model category $\Sigma\sp$. 
We denote by ${\rm Ho}(\sp)$
the stable homotopy category of spectra.
We denote by $\mathbb{S}$ the sphere spectrum.
For a spectrum $X\in {\rm Sp}$,
we write $X_*$ for the homotopy groups $\pi_*X$.
For spectra $X,Y\in {\rm Sp}$,
we write $X\otimes Y$ the smash product
of $X$ and $Y$.


\section{Opposite monoidal quasi-categories and 
opposite tensored quasi-categories
over monoidal quasi-categories}
\label{sec:opposite_monoidal}

In this section we study 
the opposite quasi-categories of monoidal quasi-categories
and the opposite quasi-categories
of tensored quasi-categories over monoidal quasi-categories.
The author thinks that the results in this section
are well-known to experts but
he decided to include this section because
he does not find out appropriate references.

In \S\ref{subsec:opposite-coCartesian-fibrations}
we recall a model of opposite coCartesian fibrations
by Barwick-Glasman-Nardin~\cite{Barwick-Glasman-Nardin}
and study maps between opposite coCartesian fibrations.  
In \S\ref{subsec:opposite-monoidal-quasi-categories}
we study the opposite quasi-category
of a monoidal quasi-category and 
show that a lax monoidal right adjoint functor
between monoidal quasi-categories
induces a lax monoidal right adjoint functor
between the opposite monoidal quasi-categories.
In \S\ref{subsec:opposite-module-quasi-categories}
we study the opposite 
of a tensored quasi-category over a monoidal quasi-category.
We show that a lax tensored right adjoint 
functor between tensored quasi-categories 
induces a lax tensored right adjoint functor
between the opposites of the tensored quasi-categories.

\subsection{Opposite coCartesian fibrations}
\label{subsec:opposite-coCartesian-fibrations}

For a coCartesian fibration
we have the opposite coCartesian fibration whose fibers
are the opposite quasi-categories of the fibers
of the original coCartesian fibration.
In this subsection we recall the explicit model of 
opposite coCartesian fibrations
due to Barwick-Glasman-Nardin~\cite{Barwick-Glasman-Nardin}. 
We show that a map between coCartesian fibrations
whose restriction to every fiber admits a left adjoint
induces a map between the opposite coCartesian fibrations.

First, we recall the explicit model of 
opposite coCartesian fibrations
by Barwick-Glasman-Nardin~\cite{Barwick-Glasman-Nardin}. 

Let $S$ be a simplicial set and 
let $p: X\to S$ be a coCartesian fibration with small fibers.
We denote by $X_s$ the quasi-category that is 
the fiber of $p$ over $s\in S$.
Let ${\rm Cat}_{\infty}$ be
the quasi-category of small quasi-categories.
By \cite[\S3.3.2]{Lurie1},
the coCartesian fibration $p$
is classified by a functor
$\mathbf{X}:S\to {\rm Cat}_{\infty}$.
There is an involution
\[ R: {\rm Cat}_{\infty}\longrightarrow
      {\rm Cat}_{\infty}\]
carrying a quasi-category to its opposite.
The composite functor $R\mathbf{X}$ classifies a coCartesian
fibration $Rp:RX\to S$
in which the fiber $(RX)_s$ of $Rp$ over $s\in S$
is equivalent to the opposite quasi-category
$(X_s)^{\rm op}$ for all $s\in S$.
We call $Rp: RX\to S$ 
the opposite coCartesian fibration
of $p: X\to S$. 
In the following of this subsection
we assume that the base simplicial set
$S$ is a quasi-category.


To describe the model of opposite coCartesian fibrations,
we recall the twisted arrow quasi-category. 
The twisted arrow quasi-category
$\widetilde{\mathcal{O}}(K)$ 
for a quasi-category $K$ 
is the simplicial set 
in which the set of $n$-simplexes is given by
\[ \widetilde{\mathcal{O}}(K)_n =
   {\rm Hom}_{{\rm Set}_{\Delta}}
   ((\Delta^{n})^{\rm op}\star\Delta^n,K)\]
with obvious structure maps.   
The simplicial set
$\widetilde{\mathcal{O}}(K)$ is actually a quasi-category
(see \cite[Prop.~4.2.3]{Lurie3}).
Note that the inclusions
$\Delta^n\hookrightarrow
(\Delta^n)^{\rm op}\star\Delta^n$
and 
$(\Delta^n)^{\rm op}\hookrightarrow
(\Delta^n)^{\rm op}\star\Delta^n$
induce maps of simplicial sets
$\widetilde{\mathcal{O}}(K)\to K$
and 
$\widetilde{\mathcal{O}}(K)\to K^{\rm op}$,
respectively.

We use the twisted arrow category
$\widetilde{\mathcal{O}}(\Delta^n)$
for the $n$-simple $\Delta^n$ for $n\ge 0$
to describe the model of opposite coCartesian fibrations.
The twisted arrow quasi-category
$\widetilde{\mathcal{O}}(\Delta^n)$
for $\Delta^n$
is the nerve of $\widetilde{[n]}$,
where $\widetilde{[n]}$
is the ordered set of all pairs $(i,j)$
of integers with $0\le i\le j\le n$
equipped with order relation 
$(i,j)\le (i',j')$
if and only if 
$i\ge i'$ and $j\le j'$.  
The ordered set $\widetilde{[n]}$
is depicted as follows
\begin{equation}\label{eq:twisted-arroes-diagram} 
\begin{array}{ccccccccccc}
   00 & \rightarrow & 01 & 
   \rightarrow & 02 & \rightarrow & \cdots &
   \rightarrow & 0\overline{1} &
   \rightarrow & 0\overline{0} \\
      &            & \uparrow &
      & \uparrow & & & & \uparrow & & \uparrow \\ 
      &            & 11 &
   \rightarrow & 12 & \rightarrow & \cdots &
   \rightarrow & 1\overline{1} & 
   \rightarrow & 1\overline{0} \\
      & & & & \uparrow & & & & \uparrow & & \uparrow \\    
      & & & & 22 & \rightarrow & \cdots
      & \rightarrow & 2\overline{1}& \rightarrow
      & 2\overline{0} \\              
      & & & & & & & & \uparrow & & \uparrow \\
      & & & & & & \ddots & & \vdots & & \vdots \\
      & & & & & & & & \uparrow & & \uparrow \\
      & & & & & & & & \overline{1}\overline{1}
      & \rightarrow & \overline{1}\overline{0}\\
      & & & & & & & & & & \uparrow \\
      & & & & & & & & & & \phantom{,}\overline{0}\overline{0},\\
\end{array}
\end{equation}
where $\overline{k}=n-k$.

By functoriality of $\widetilde{\mathcal{O}}(-)$,
we have a cosimplicial simplicial set
$\widetilde{\mathcal{O}}(\Delta^{\bullet})$.  
For a simplicial set $K$ over $S$,
we define a simplicial set $H(K)$ over $S$ as follows.
The simplicial set $H(K)$
is a simplicial subset 
of ${\rm Hom}_{\rm sSet}
    (\widetilde{\mathcal{O}}(\Delta^{\bullet}),K)$.
A map $\varphi: \widetilde{\mathcal{O}}(\Delta^n)\to K$
is an $n$-simplex of $H(K)$ for $n\ge 0$
if the $j$-simplex $\varphi(jj)\to \cdots 
\to \varphi(1j)\to \varphi(0j)$ covers 
a totally degenerate $j$-simplex of $S$,
that is, a $j$-simplex in the image of the map
$S_0\to S_j$, for all $0\le j\le n$.
Assigning to an $n$-simplex $\varphi$ of $H(K)$
the $n$-simplex $p\varphi(00)\to p\varphi(01)\to 
\cdots\to p\varphi(0n)$ of $S$,
we obtain a map $H(K)\to S$.

Let $p:X\to S$ be a coCartesian fibration,
where $S$ is a quasi-category.
We define a simplicial set $RX$
as follows. 
The simplicial set $RX$ is
a simplicial subset of $H(X)$. 
A map $\varphi: \widetilde{\mathcal{O}}(\Delta^n)\to X$
is an $n$-simplex of $RX$ for $n\ge 0$
if the following two conditions are satisfied:
\begin{enumerate}
\item
The $j$-simplex $\varphi(jj)\to \cdots 
\to \varphi(1j)\to \varphi(0j)$ covers 
a totally degenerate $j$-simplex of $S$
for all $0\le j\le n$.
\item
The $1$-simplex
$\varphi(ij)\to \varphi(ik)$ is a $p$-coCartesian edge for all 
$0\le i\le j\le k\le n$.
\end{enumerate}

As in $H(X)$, we have a map
\[ Rp: RX\to S, \]
which is a coCartesian fibration.
The fiber $(RX)_s$ over $s\in S$ is equivalent
to the opposite quasi-category $(X_s)^{\rm op}$ of 
the fiber $X_s$ for all $s\in S$.
An edge $\varphi\in {\rm Hom}_{\rm sSet}(\Delta^1,RX)$
is $Rp$-coCartesian if and only if
the edge $\varphi(11)\to \varphi(01)$ 
is an equivalence in the fiber $X_s$,
where $s=p\varphi(11)$.
The coCartesian fibration
$Rp: RX\to S$ is a model of the opposite
coCartesian fibration corresponding
to the composite
\[ R\mathbf{X}: S\stackrel{\mathbf{X}}{\longrightarrow} 
   {\rm Cat}_{\infty}\stackrel{R}{\longrightarrow}{\rm Cat}_{\infty},\]
where $\mathbf{X}:S\to {\rm Cat}_{\infty}$
is the map corresponding to the coCartesian
fibration $p: X\to S$,
and $R: {\rm Cat}_{\infty}\to{\rm Cat}_{\infty}$
is the functor which assigns to a quasi-category
its opposite quasi-category.

Next, we consider a map
between coCartesian fibrations which
admits a left adjoint for each fibers.
We show that the map induces a canonical
map in the opposite direction
between the opposite coCartesian fibrations.
   
Let $p: X\to S$ and $q: Y\to S$ be coCartesian fibrations over 
a quasi-category $S$.
Suppose we have a map $G: Y\to X$ over $S$.
Note that we do not assume that $G$ preserves coCartesian edges.
The map $G:Y\to X$ over $S$ induces a functor $G_s: Y_s\to X_s$ 
between the quasi-categories of fibers
for each $s\in S$.

We shall define a simplicial set 
$\mathcal{R}$ over $S$ equipped with maps
$\pi_X:\mathcal{R}\to RX$ and
$\pi_Y:\mathcal{R}\to RY$ over $S$.
For a simplicial set $K$ and $X$,
we denote by ${\rm Fun}(K,X)$
the mapping simplicial set from $K$ to $X$.
The map $p: X\to S$
induces a map $p_*: {\rm Fun}(\Delta^1,X)\to 
{\rm Fun}(\Delta^1,S)$.
We regard $S$ as a simplicial subset of
${\rm Fun}(\Delta^1,S)$
via constant maps.
We denote by ${\rm Fun}^S(\Delta^1,X)$
the pullback of $p_*$ along 
the inclusion $S\hookrightarrow {\rm Fun}(\Delta^1,S)$.
The inclusion $\Delta^{\{i\}}\hookrightarrow \Delta^1$
induces a map
${\rm Fun}^S(\Delta^1,X)\to X$ over $S$ for $i=0,1$. 

We have the inclusion
$RX\hookrightarrow H(X)\cong 
 H({\rm Fun}(\Delta^{\{0\}},X))$.
The map $G: Y\to X$ over $S$
induces a map
$RY\hookrightarrow H(Y)
\stackrel{G_*}{\longrightarrow}H(X)\cong
H({\rm Fun}(\Delta^{\{1\}},X))$.
The inclusion $\Delta^{\{i\}}\hookrightarrow
\Delta^1$ 
induces a map
$H({\rm Fun}^S(\Delta^1,X))\to
H({\rm Fun}(\Delta^{\{i\}},X))$
for $i=0,1$.
Using these maps,
we define a simplicial set $\mathcal{R}$ by
\[ \mathcal{R}=RX\times_{H({\rm Fun}(\Delta^{\{0\}},X))}
               H({\rm Fun}^S(\Delta^1,X))
   \times_{H({\rm Fun}(\Delta^{\{1\}},X))}
   RY. \]
We have a map $\mathcal{R}\to S$ and 
projections
$\pi_X:\mathcal{R}\to RX$ and $\pi_Y: \mathcal{R}\to RY$
over $S$.

Now we assume that
the functor
$G_s: Y_s\to X_s$ admits a left adjoint $F_s$ for all $s\in S$.
Then an object $x$ of $X$ with $s=p(x)$
determines an object $(x,u_x,F_s(x))$
of $\mathcal{R}$,
where $u_x: x\to G_sF_s(x)$
is the unit map of the adjunction
$(F_s,G_s)$ at $x$.
We define $\mathcal{R}^0$ to be
the full subcategory of $\mathcal{R}$
spanned by $\{(x,u_x,F_s(x))\}$
for all $x\in X$, where $s=p(x)$.
Let 
\[ \pi_{X}^{0}: \mathcal{R}^0\longrightarrow RX.\]
be the restriction of $\pi_X$ to $\mathcal{R}^0$.

\begin{proposition}
\label{prop:existence-final-object-sections}
The map $\pi_X^0:\mathcal{R}^0\to RX$
is a trivial Kan fibration.
\end{proposition}

We defer the proof of 
Proposition~\ref{prop:existence-final-object-sections}
to \S\ref{sec:proof-theorem-construction-RF}.

We take a section $T_0$ of $\pi_X^0$,
which is unique up to contractible space of choices.
Let $\pi_Y^0:\mathcal{R}^0\to RY$ 
be the restriction of $\pi_Y$ to $\mathcal{R}^0$.
We define a functor 
\[ RF: RX \longrightarrow RY \]
to be $\pi_Y^0\,T_0$.

We would like to describe some properties
of the section $T_0$.
Let $s\in S$.
We consider the restriction of $T_0$
to $(RX)_s$.
The fiber $\mathcal{R}_s$ is described as
\[ \mathcal{R}_s=(RX)_s\times_{H({\rm Fun}(\Delta^{\{0\}},X_s))}
   H({\rm Fun}(\Delta^1,X_s))
   \times_{H({\rm Fun}(\Delta^{\{1\}},X_s))}(RY)_s,\] 
and the fiber $\mathcal{R}^0_s$ is a full subcategory
of $\mathcal{R}_s$.
The composition
$(RX)_s\hookrightarrow H(X_s)\stackrel{(F_s)_*}
{\longrightarrow}H(Y_s)$
factors through $(RY)_s$.
We denote by $R(F_s)$ the induced functor $(RX)_s\to (RY)_s$.
The unit map
$u_s: 1_{X_s}\to G_sF_s$ in ${\rm Fun}(X_s,X_s)$
can be identified with a map
$u_s: X_s\to {\rm Fun}(\Delta^1,X_s)$.
We obtain a map
$Hu_s: (RX)_s\to H({\rm Fun}(\Delta^1,X_s))$
by the composition
$(RX)_s\hookrightarrow H(X_s)\stackrel{(u_s)_*}{\longrightarrow}
H({\rm Fun}(\Delta^1,X_s))$.
Note that $Hu_s$ followed by
$H({\rm Fun}(\Delta^1,X_s))\to
H({\rm Fun}(\Delta^{\{0\}},X_s))$
is the inclusion $(RX)_s\hookrightarrow H(X_s)$,
and the map $Hu_s$ followed by
$H({\rm Fun}(\Delta^1,X_s))\to
H({\rm Fun}(\Delta^{\{1\}},X_s))$
is the composition of $R(F_s): (RX)_x\to (RY)_s$
followed by the inclusion
$(RY)_s\hookrightarrow H(X_s)$.
Hence we obtain a section
of $\mathcal{R}^0_s$ over $(RX)_s$:
\[ (1_{(RX)_s},Hu_s,R(F_s)):
   (RX)_s\longrightarrow \mathcal{R}^0_s.\]

\begin{proposition}
\label{prop:final-object-property-1}
We have 
\[ T_0|_{(RX)_s}\simeq (1_{(RX)_s},Hu_s,R(F_s)).\]
for any $s\in S$.
\end{proposition}

\begin{proof}
Restricting $\pi_X^0$ to the fibers over $s\in S$,
we obtain a trivial Kan fibration
$(\pi_X^0)_s: \mathcal{R}^0_s\to (RX)_s$.
The restriction of the section $T_0$
to $(RX)_s$ is a section of $(\pi_X^0)_s$.
The map $(1_{(RX)_s},Hu_s,R(F_s))$
is also a section of $(\pi_X^0)_s$.
Hence we have
$T_0|_{(RX)_s}\simeq (1_{(RX)_s},Hu_s,R(F_s))$.
\qed\end{proof}

Next, we consider
the image of edges of $RX$
under the section $T_0$.
Let $\varphi$ be an edge of $RX$
over $e:s\to s'$ in $S$
represented by a $p$-coCartesian edge 
$\varphi(00)\to\varphi(01)$ in $X$ and
an edge $\varphi(11)\to\varphi(01)$
in the fiber $X_{s'}$.
We take a $q$-coCartesian edge
$\psi: F_s\varphi(00)\to y'$ in $Y$ over $e$.
Since $\varphi(00)\to\varphi(01)$
is $p$-coCartesian,
we obtain an edge $\varphi(01)\to G_{s'}y'$ in $X_{s'}$,
which makes the following diagram commute
\begin{equation}\label{diagram:coCartesian-induce-map} 
  \xymatrix{
    \varphi(00) \ar[r] \ar[d]_u &  \varphi(01) \ar@{.>}[d]\\
    G_sF_s\varphi(00) \ar[r]^{G\psi}  &  G_{s'}y',\\
   }
\end{equation}  
where $u$ is the unit map of the adjunction
$(F_s,G_s)$ at $\varphi(00)$.
Let $w: F_{s'}\varphi(11)\to y'$ be
the map in $Y_{s'}$
obtained from $\varphi(11)\to \varphi(01)$
by applying $F_{s'}$
followed by the adjoint map of $\varphi(01)\to G_{s'}y'$.
We denote by $R\varphi$
the edge of $RY$ over $e$ represented by
\[ F_s\varphi(00)\stackrel{\psi}{\longrightarrow}y'
   \stackrel{w}{\longleftarrow}F_{s'}\varphi(11).\]
Since the composite
$\varphi(11)\to\varphi(01)\to G_{s'}y'$
is adjoint to 
$w: F_{s'}\varphi(11)\to y'$,
we have an edge 
$H\varphi: \varphi\to G(R\varphi)$
of $H({\rm Fun}^S(\Delta^1,X))$
represented by the following commutative diagram
\[\xymatrix{
    \varphi(00) \ar[r] \ar[d]_u &  \varphi(01) \ar[d] 
     & \ar[l] \varphi(11) \ar[d]^u\\
    G_sF_s\varphi(00) \ar[r]^{G\psi}  &  G_{s'}y'
     & \ar[l]_{G_{s'}w} G_{s'}F_{s'}\varphi(11).\\ }
\]

\begin{proposition}
\label{prop:final-section-properties-2}
For any edge $\varphi$ of $RX$,
we have 
\[ T_0(\varphi)\simeq (\varphi,H\varphi,R\varphi). \]
\end{proposition}

\begin{proof}
Let $\pi_{\varphi}^0: \mathcal{R}_{\varphi}^0\to \Delta^1$
be the trivial Kan fibration obtained by
the pullback of $\pi_X^0$ along the map
$\varphi: \Delta^1\to RX$.
The triple $(\varphi, H\varphi,R\varphi)$ determines
a section of $\pi_{\varphi}^0$.
Hence $T_0(\varphi)\simeq (\varphi, H\varphi,R\varphi)$.
\qed\end{proof}


The main result in this subsection is 
the following theorem.

\begin{theorem}
\label{thm:construction-RF}
Let $p: X\to S$ and $q: Y\to S$ be coCartesian fibrations over 
a quasi-category $S$.
Suppose we have a map $G: Y\to X$ 
over $S$.
If $G_s$ admits a left adjoint $F_s$
for all $s\in S$,
then 
there exists a canonical map $RF:RX\to RY$
over $S$
up to contractible space of choices.
We have
$(RF)_s\simeq F_s^{\,\rm op}$ for all $s\in S$
and $RF(\varphi)\simeq R\varphi$
for any edge $\varphi$ of $RX$.
\end{theorem}

\begin{proof}
The theorem follows from
Propositions~\ref{prop:existence-final-object-sections},
\ref{prop:final-object-property-1}, and 
\ref{prop:final-section-properties-2}.
\qed
\end{proof}

Now we consider which coCartesian edge of $RX$
is preserved by the functor $RF$.
Let $e: s\to s'$ be a $1$-simplex of $S$.
Since $q:Y\to S$ and $p:X\to S$ are coCartesian fibrations,
we have functors 
$e^Y_!: Y_s\to Y_{s'}$ and 
$e^X_!: X_s\to X_{s'}$
associated to $e$.
The map $G:Y\to X$ over $S$ induces  
a diagram 
$\partial(\Delta^1\times\Delta^1)\to{\rm Cat}_{\infty}$
depicted as 
\begin{equation}
\label{eq:Delta1Delta1diagram}
\xymatrix{
     Y_s \ar[d]_{e^Y_!} \ar[r]^{G_s}& X_s \ar[d]^{e^X_!} \\
     Y_{s'} \ar[r]^{G_{s'}} & X_{s'}  \\   }
\end{equation}
and a natural transformation
\begin{equation}
\label{eq:natural-transformation-associated-to-G} 
e^X_!G_s\longrightarrow G_{s'}e^Y_!.
\end{equation}
If natural 
transformation~(\ref{eq:natural-transformation-associated-to-G})
is an equivalence,
then
$G: Y\to X$ preserves coCartesian edges over $e$.

We recall the definition of left adjointable
diagram (see \cite[Def.~4.7.5.13]{Lurie2}).
Suppose we are given a diagram of quasi-categories
\[ \xymatrix{
    \mathcal{C} \ar[r]^{G} \ar[d]_U& 
    \mathcal{D} \ar[d]^V\\
    \mathcal{C}' \ar[r]^{G'} & \mathcal{D}' }\]
which commutes up to a specified equivalence
$\alpha: VG\simeq G'U$.
We say that this diagram is left adjointable
if the functors $G$ and $G'$ admit left adjoints 
$F$ and $F'$, respectively,
and if the composite transformation
\[ F'V\to F'VGF\stackrel{\alpha}{\simeq}
   F'G'UF\to UF \]
is an equivalence,
where the first map is induced by the unit map
of the adjunction $(F,G)$, 
and the third map is induced by the counit map
of the adjunction $(F',G')$.

\begin{proposition}
\label{prop:preserving-coCartesian-edges}
Let $e$ be a $1$-simplex of $S$.
If natural 
transformation~{\rm (\ref{eq:natural-transformation-associated-to-G})}
is an equivalence and
diagram~{\rm (\ref{eq:Delta1Delta1diagram})}
equipped with this equivalence
is left adjointable,
then $RF: RX\to RY$ preserves
coCartesian edges over $e$.
\end{proposition}

\begin{proof}
Let $\varphi$ be an $Rp$-coCartesian edge  of $RX$
over $e$
represented by $\varphi(00)\to \varphi(01)
\leftarrow \varphi(11)$,
where $\varphi(00)\to\varphi(01)$
is a $p$-coCartesian edge of $X$ over $e$
and $\varphi(11)\to \varphi(01)$
is an equivalence in $X_{s'}$.
We can regard $\varphi(11)$
as $e^{X}_!\varphi(00)$.

Suppose that the edge $R\varphi$ of $RY$ over $e$
is represented by
\[ F_s\varphi(00)\stackrel{\psi}{\longrightarrow} 
   y'\stackrel{w}{\longleftarrow} F_{s'}\varphi(11),\]
where $\psi$ is a $q$-coCartesian edge of $Y$ over $e$
and $w$ is an edge of $Y_{s'}$.
We have to show that $w$ is an equivalence
of $Y_{s'}$.

We can regard $y'$ as $e^{Y}_!F_s\varphi(00)$
and $F_{s'}\varphi(11)$
as $F_{s'}e^X_!\varphi(00)$.
The morphism $w$ is the adjoint of the
morphism
$\varphi(11)\to\varphi(01)\to G_{s'}y'$,
which can be identified with
$e^X_!\varphi(00)\to G_{s'}e^Y_!F_s\varphi(00)$.
By the assumption that
diagram~(\ref{eq:Delta1Delta1diagram}) 
is left adjointable,
we see that $w$ is an equivalence.
\qed
\end{proof}

\subsection{Opposite monoidal quasi-categories}
\label{subsec:opposite-monoidal-quasi-categories}

In this subsection
we study the opposite monoidal quasi-category
of a monoidal quasi-category.
We show that a lax right adjoint functor
between monoidal quasi-categories
induces a lax right adjoint functor
between the opposite monoidal quasi-categories.

First, we recall the definition
of monoidal quasi-categories.
Let $p: M\to N(\Delta)^{\rm op}$
be a coCartesian fibration 
of simplicial sets.
For any $n\ge 0$,
the inclusion
$[1]\cong\{i-1,i\}\hookrightarrow [n]$ induces 
a functor 
$p_i: X_{[n]}\to X_{[1]}$ 
of quasi-categories for $i=1,\ldots,n$.
We say that $p$ is a monoidal quasi-category
if the functor
\[ p_1\times\cdots\times p_n:
   X_{[n]}\longrightarrow
   \overbrace{X_{[1]}\times\cdots\times X_{[1]}}^n \]
is a categorical equivalence 
for all $n\ge 0$.
The fiber $M_{[1]}$ of $p$
over $[1]\in\Delta$ is said to be
the underlying quasi-category of 
the monoidal category $p$.

Let $p:M\to N(\Delta)^{\rm op}$
be a monoidal quasi-category.
Since $p$ is a coCartesian fibration by definition,
we have a functor $\mathbf{X}:
N(\Delta)^{\rm op}\to {\rm Cat}_{\infty}$
classifying $p$.
We have the opposite coCartesian fibration
$Rp: RM\to N(\Delta)^{\rm op}$
that is classified by the functor
$R\mathbf{X}$.
We easily see that 
$Rp: RM\to N(\Delta)^{\rm op}$ 
is a monoidal quasi-category.
Note that the fiber $(RM)_{[n]}$
is equivalent to $(M_{[n]})^{\rm op}\simeq
(M_{[1]}^{\rm op})^n$
for any $n\ge 0$.
We say that $RM$
is the opposite monoidal quasi-category of $M$. 

A map $[m]\to [n]$ in $\Delta$
is said to be convex 
if it is injective and
the image is $\{i,i+1,\ldots,i+m\}$
for some $i$.
Let $p:M\to N(\Delta)^{\rm op}$
and $q:N\to N(\Delta)^{\rm op}$
be monoidal quasi-categories.
A lax monoidal functor
$G: N\to M$ between the monoidal quasi-categories
is a map of simplicial sets
over $N(\Delta)^{\rm op}$
which carries $p$-coCartesian edges over convex
morphisms in $N(\Delta)^{\rm op}$ to $q$-coCartesian edges.

\begin{lemma}
If $G_{[1]}: N_{[1]}\to M_{[1]}$ admits a left adjoint $F_{[1]}$,
then there is a canonical functor $RF: RM\to RN$ 
over $N(\Delta)^{\rm op}$ up to contractible space of choices.
We have $(RF)_{[n]}\simeq (F_{[1]}^{\,\rm op})^n$
for all $n\ge 0$.
\end{lemma}

\begin{proof}
For any $n\ge 0$,
we have equivalences $M_{[n]}\simeq (M_{[1]})^n$
and $N_{[n]}\simeq (N_{[1]})^n$.
Since $G$ is a lax monoidal functor,
we see that $G_{[n]}$ is equivalent
to $(G_{[1]})^n$ under the above equivalences.
Hence $G_{[n]}$ admits a left adjoint for all $n\ge 0$. 
The lemma follows from 
Theorem~\ref{thm:construction-RF}.
\qed\end{proof}

\begin{proposition}
\label{prop:opposite-lax-monoidal-functor}
If $G: N\to M$ is a lax monoidal functor 
between monoidal quasi-categories
such that  
$G_{[1]}: N_{[1]}\to M_{[1]}$ admits a left adjoint,
then the functor
$RF: RM\to RN$ is also a lax monoidal functor
between the opposite monoidal quasi-categories.
\end{proposition}

\begin{proof}
We have to show that $RF$ preserves
coCartesian edges over convex morphisms.
Let $\alpha: [m]\to [n]$ be a convex morphism
in $\Delta$.
Since $G:N\to M$ is a lax monoidal functor,
we have a commutative diagram
\[ \xymatrix{
     N_{[n]} \ar[d]_{\alpha^N_!} \ar[r]^{G_{[n]}}& 
     M_{[n]} \ar[d]^{\alpha^M_!}\\
     N_{[m]} \ar[r]^{G_{[m]}}&   M_{[m]}\\  
   }\]
in ${\rm Cat}_{\infty}$.
Since 
$\alpha^M_!: M_{[n]}\to M_{[m]}$ and 
$\alpha^N_!:N_{[n]}\to N_{[m]}$ are equivalent to
projections $M_{[1]}^n\to M_{[1]}^m$
and $N_{[1]}^n\to N_{[1]}^m$, respectively, 
we see that the diagram is left adjointable.
The proposition follows from 
Proposition~\ref{prop:preserving-coCartesian-edges}.
\qed\end{proof}

\subsection{Opposites of tensored quasi-categories
over monoidal quasi-categories}
\label{subsec:opposite-module-quasi-categories}

In this subsection we study the opposite
of a tensored quasi-category over a monoidal quasi-category.
We show that the opposite of a lax tensored right adjoint 
functor between tensored quasi-categories 
induces a lax tensored right adjoint functor
between the opposites of the tensored quasi-categories.

First,
we recall the definition 
of left tensored quasi-category
over a monoidal quasi-category.
Let $p:X\to N(\Delta)^{\rm op}\times\Delta^1$
be a coCartesian fibration of simplicial sets.
For any $n\ge 0$,
the identity ${\rm id}_{[n]}:[n]\to [n]$ in $\Delta$
and the edge
$\{0\}\to\{1\}$ in $\Delta^1$
induces a morphism
$([n],0)\to([n],1)$ in $N(\Delta)^{\rm op}\times \Delta^1$,
and hence we obtain a functor of quasi-categories
$\alpha_n:X_{([n],0)}\to X_{([n],1)}$.
For any $n\ge 0$,
the inclusion
$[0]\cong\{n\}\hookrightarrow [n]$ in $\Delta$
and the identity ${\rm id}_{\{0\}}:\{0\}\to\{0\}$ in $\Delta^1$
induces a morphism
$([0],0)\to ([n],0)$ in $N(\Delta)^{\rm op}\times \Delta^1$,
and hence we obtain 
a functor of quasi-categories
$\beta_n:X_{([n],0)}\to X_{([0],0)}$.
If the base change of $p$ 
along the inclusion
$N(\Delta)^{\rm op}\times\{1\}\hookrightarrow
N(\Delta)^{\rm op}\times\Delta^1$
is a monoidal quasi-category,
and the functor
\[ \alpha_n\times \beta_n: X_{([n],0)}\longrightarrow
   X_{([n],1)}\times X_{([0],0)}\]
is a categorical equivalence
for all $n\ge 0$,
then 
we say that $p$ is a left tensored quasi-category.

Now suppose $p: X\to N(\Delta)^{\rm op}\times\Delta^1$
is a left tensored quasi-category.
We set $\mathcal{M}=X_{([1],1)}$ and 
$\mathcal{C}=X_{([0],0)}$.
Note that $\mathcal{M}$ is the underlying
quasi-category of a monoidal quasi-category
$p|_{N(\Delta)^{\rm op}\times\{1\}}$.
We say that $\mathcal{C}$
is left tensored over the monoidal quasi-category $\mathcal{M}$.

Let $\Sigma$ be a set of edges of 
$N(\Delta)^{\rm op}\times \Delta^1$
consisting of edges of
the forms
\[ ([n],0)\longrightarrow ([m],0), \]
where $[m]\to[n]$ is a convex morphism in $\Delta$
that carries $m$ to $n$,
and 
\[ ([n],i)\longrightarrow ([m],1) \]
for $i=0,1$,
where $[m]\to[n]$ is convex.

Suppose that $p:X\to N(\Delta)^{\rm op}\times \Delta^1$
and $q: Y\to N(\Delta)^{\rm op}\times \Delta^1$
are left tensored quasi-categories.
We say that a functor 
$G: Y\to X$ over $N(\Delta)^{\rm op}\times \Delta^1$
is a lax left tensored functor
if $G$ carries $p$-coCartesian edges
over $\Sigma$ to $q$-coCartesian edges.

\begin{lemma}
Let $p:X\to N(\Delta)^{\rm op}\times \Delta^1$
and $q: Y\to N(\Delta)^{\rm op}\times \Delta^1$
be left tensored quasi-categories.
If $G: Y\to X$ is a lax left tensored functor
such that $G_{([0],0)}$ and $G_{([1],1)}$
admit left adjoints $F_{([0],0)}$ and $F_{([1],1)}$,
respectively,
then there is a canonical functor
$RF: RX\to RY$ over $N(\Delta)^{\rm op}\times \Delta^1$
up to contractible space of choices.
We have
$(RF)_{([n],0)}\simeq (F_{([1],1)}^{\rm op})^n
\times F_{([0],0)}^{\rm op}$
and
$(RF)_{([n],1)}\simeq (F_{[1],1}^{\rm op})^n$
for all $n\ge 0$.
\end{lemma}

\begin{proof}
For any $n\ge 0$,
we have equivalences
$X_{([n],0)}\simeq (X_{([1],1)})^n\times X_{([0],0)}$
and 
$Y_{([n],0)}\simeq (Y_{([1],1)})^n\times Y_{([0],0)}$.
Since $G$ is a lax left tensored functor,
we see that $G_{([n],0)}$ 
is equivalent to 
$(G_{([1],1)})^n\times G_{([0],0)}$.
In the same way,
we see that $G_{([n],1)}$
is equivalent to
$(G_{([1],1)})^n$ for any $n\ge 0$.
Hence $G_s$ admits a left adjoint
for all $s\in N(\Delta)^{\rm op}\times \Delta^1$.
The lemma follows from 
Theorem~\ref{thm:construction-RF}.
\qed\end{proof}

\begin{proposition}
\label{prop:opposite-lax-tensored-module-functor}
Let $p:X\to N(\Delta)^{\rm op}\times \Delta^1$
and $q: Y\to N(\Delta)^{\rm op}\times \Delta^1$
be left tensored quasi-categories.
If $G: Y\to X$ is a lax left tensored functor
such that $G_{([0],0)}$ and $G_{([1],1)}$
admit left adjoints,
then the functor $RF: RX\to RY$ is also a lax left tensored functor.     
\end{proposition}

\begin{proof}
We can prove the proposition
in the same say as 
Proposition~\ref{prop:opposite-lax-monoidal-functor}.
We have to show that $RF$ preserves coCartesian
edges over $\Sigma$.
Let $\alpha: s\to s'$ be an edge in $\Sigma$.
Since $G: Y\to X$ is a lax left tensored functor,
we have a commutative diagram
\[ \xymatrix{
     Y_s \ar[d]_{\alpha^Y_!} \ar[r]^{G_s}&
     X_s \ar[d]^{\alpha^X_!}\\
     Y_{s'} \ar[r]^{G_{s'}} &  X_{s'}\\  
   }\]
in ${\rm Cat}_{\infty}$.
Since 
$\alpha^Y_!: Y_{s}\to Y_{s'}$ and 
$\alpha^X_!: X_{s}\to X_{s'}$ are equivalent to
the projections,
we see that the diagram is left adjointable.
The proposition follows from 
Proposition~\ref{prop:preserving-coCartesian-edges}.
\qed\end{proof}

\section{Quasi-categories of comodules}
\label{sec:quasi-categories-comodules}

In this section we introduce a quasi-category
of comodules over a coalgebra in a stable homotopy theory $\mathcal{C}$.
We regard a coalgebra as an algebra object
of the opposite monoidal quasi-category of 
$A$-$A$-bimodule objects, where $A$ is an algebra object 
of $\mathcal{C}$.  
We regard a comodule object over a coalgebra $\Gamma$
as a module object over $\Gamma$
in the opposite quasi-category of $A$-module objects.
We define a cotensor product
of a right comodule and a left comodule
over a coalgebra as a limit of
the cobar construction.
Using these formulations,
we study the functor
from the localization of $\mathcal{C}$
with respect to $A$ to
the quasi-category of comodules
over the coalgebra $A\otimes A$.

\subsection{Monoidal structure on 
${}_A{\rm BMod}_{A}(\mathcal{C})^{\rm op}$}

In this subsection
we introduce a quasi-category of coalgebras
and a quasi-category of comodules over a coalgebra
in a stable homotopy theory.

Let $\mathcal{M}^{\otimes}$ be a monoidal quasi-category.
We denote by $\mathcal{M}$
the underlying quasi-category of 
the monoidal quasi-category $\mathcal{M}^{\otimes}$.
For algebra objects $A$ and $B$ of $\mathcal{M}$,
we denote by ${_A}\bmod_B(\mathcal{M})$
the quasi-category of $A$-$B$-bimodule objects in $\mathcal{M}$.
If $B$ is the monoidal unit $\mathbf{1}$ in $\mathcal{M}$,
we abbreviate the quasi-category
${_A}\bmod_{\mathbf{1}}(\mathcal{M})$
of $A$-$\mathbf{1}$-bimodule objects in $\mathcal{M}$
as ${_A}\bmod(\mathcal{M})$.
Let $\mathcal{N}$ be a quasi-category
left tensored over $\mathcal{M}^{\otimes}$.
For an algebra object $A$ of $\mathcal{M}$,
we denote by $\lmod_A(\mathcal{N})$
the quasi-category of left $A$-module objects in $\mathcal{N}$.
Note that there is a natural equivalence
$\lmod_A(\mathcal{M})\simeq
{_A}\bmod(\mathcal{M})$ of quasi-categories.

Let $(\mathcal{C},\otimes,\mathbf{1})$ be a stable homotopy theory
in the sense of \cite[Def.~2.1]{Mathew},
that is,
$\mathcal{C}$ is a presentable stable quasi-category
which is the underlying quasi-category
of a symmetric monoidal quasi-category $\mathcal{C}^{\otimes}$,
where the tensor product commutes
with all colimits separately in each variable.
For an algebra object
$A$ of $\mathcal{C}$,
we denote by
${}_A\bmod_A(\mathcal{C})$
the quasi-category of $A$-$A$-bimodules in $\mathcal{C}$,
which is the underlying quasi-category
of the monoidal quasi-category
${}_A\bmod_A(\mathcal{C})^{\otimes}$,
where the tensor product is given by 
the relative tensor product $\otimes_A$
and the unit is the $A$-$A$-bimodule $A$
(see \cite[4.3 and 4.4]{Lurie2}).
Note that the relative tensor product $\otimes_A$
commutes with all colimits separately 
in each variable by \cite[Cor.~4.4.2.15]{Lurie2}.    
For algebra objects $A$ and $B$ of $\mathcal{C}$,
we denote by ${}_A\bmod_A(\mathcal{C})$
the quasi-category of $A$-$B$-bimodules,
which is presentable
by \cite[Cor.~4.3.3.10]{Lurie2}.

If $\mathcal{M}^{\otimes}$ is a monoidal quasi-category,
then the opposite quasi-category
$(\mathcal{M}^{\otimes})^{\rm op}$ 
also carries a monoidal structure.
Since ${}_A\bmod_A(\mathcal{C})$
is the underlying quasi-category 
of the monoidal quasi-category
${}_A\bmod_A(\mathcal{C})^{\otimes}$
for an algebra object $A$ of $\mathcal{C}$,
the opposite quasi-category
${}_A\bmod_A(\mathcal{C})^{\rm op}$
is the underlying quasi-category
of the opposite monoidal quasi-category
$({}_A\bmod_A(\mathcal{C})^{\otimes})^{\rm op}$.
We regard an algebra object
$\Gamma$ of ${}_A\bmod_A(\mathcal{C})^{\rm op}$
as a coalgebra object of ${}_A\bmod_A(\mathcal{C})$.
We define the quasi-category ${}_A{\rm CoAlg}_A(\mathcal{C})$
of coalgebra objects of ${}_A\bmod_A(\mathcal{C})$
to be the opposite of the quasi-category
of algebra objects of ${}_A\bmod_A(\mathcal{C})^{\rm op}$:
\[ {}_A{\rm CoAlg}_A(\mathcal{C})=
   {\rm Alg}({}_A\bmod_A(\mathcal{C})^{\rm op})^{\rm op}.\] 

For a quasi-category $\mathcal{Y}$ left tendered 
over a monoidal category $\mathcal{M}^{\otimes}$,
the opposite quasi-category $\mathcal{Y}^{\rm op}$
carries the structure of left tensored 
quasi-category over the opposite
monoidal quasi-category $(\mathcal{M}^{\otimes})^{\rm op}$.

The quasi-category ${_A}\bmod(\mathcal{C})\simeq \lmod_A(\mathcal{C})$
is left tensored over ${_A}\bmod_A(\mathcal{\mathcal{C}})^{\otimes}$
by the relative tensor product $\otimes_A$
for an algebra object $A$ of $\mathcal{C}$.
Hence the opposite 
quasi-category ${}_A\bmod(\mathcal{C})^{\rm op}$
is left tensored over 
the opposite monoidal quasi-category
$({}_A\bmod_A(\mathcal{C})^{\otimes})^{\rm op}$.
Let $\Gamma$ be a coalgebra object
of ${}_A\bmod_A(\mathcal{C})$,
that is, an algebra object
of ${}_A\bmod_A(\mathcal{C})^{\rm op}$.
We regard a left $\Gamma$-module
in ${}_A\bmod(\mathcal{C})^{\rm op}$
as a left $\Gamma$-comodule 
in ${}_A\bmod(\mathcal{C})$.
We define the quasi-category
of left $\Gamma$-comodules
$\lcomod_{(A,\Gamma)}(\mathcal{C})$
to be the opposite of the quasi-category
of left $\Gamma$-module objects
in ${}_A\bmod_A(\mathcal{C})^{\rm op}$:
\[ \lcomod_{(A,\Gamma)}(\mathcal{C})=
   (\lmod_{\Gamma}({}_A\bmod(\mathcal{C})^{\rm op}))^{\rm op}.\]
Note that $\lcomod_{(A,\Gamma)}(\mathcal{C})$
is right tensored over $\mathcal{C}$ and 
there is a forgetful functor
\[ \lcomod_{(A,\Gamma)}(\mathcal{C})\longrightarrow
   {}_A{\rm BMod}(\mathcal{C})\simeq
   {\rm LMod}_A(\mathcal{C}),\] 
which is a map of quasi-categories right tensored over $\mathcal{C}$.

In the same way as $\lcomod_{(A,\Gamma)}(\mathcal{C})$,
we can define the quasi-category
of right $\Gamma$-comodules 
$\rcomod_{(A,\Gamma)}(\mathcal{C})$
in $\mathcal{C}$
for a coalgebra $\Gamma$ of 
${}_A\bmod_A(\mathcal{C})$.

\subsection{Comparison maps}


In this subsection
we construct a functor between quasi-categories of comodules
for a map of algebra objects.
We also show that the definition of a quasi-category
of comodules in this paper is consistent with
the definition in \cite{Torii1}.

Suppose $(\mathcal{C},\otimes,\mathbf{1})$
be a stable homotopy theory.
Let
$f: A\to B$ be a map of algebra objects of $\mathcal{C}$.
We have the functor
\[ (f,f)^*:{}_B\bmod_B(\mathcal{C})\longrightarrow 
   {}_A\bmod_A(\mathcal{C}) \]
which is obtained 
by restriction of scalars through $f$.
The functor $(f,f)^*$ extends to
a lax monoidal functor
\[ ((f,f)^{*})^{\otimes}:{}_B\bmod_B(\mathcal{C})^{\otimes}
   \longrightarrow 
   {}_A\bmod_A(\mathcal{C})^{\otimes}. \]
Furthermore,
the functor $(f,f)^*$ admits
a left adjoint 
\[ (f,f)_!:{}_A\bmod_A(\mathcal{C})\longrightarrow 
   {}_B\bmod_B(\mathcal{C}), \]
which assigns to an $A$-$A$-bimodule $X$
the $B$-$B$-bimodule $B\otimes_A X\otimes_AB$.
By Proposition~\ref{prop:opposite-lax-monoidal-functor},
we obtain the following lemma.

\begin{lemma}
If $f: A\to B$ is a map of algebra objects of $\mathcal{C}$,
then the induced functor
\[ (f,f)_!^{\rm op}:{}_A\bmod_A(\mathcal{C})^{\rm op}
   \longrightarrow 
   {_B}\bmod_B(\mathcal{C})^{\rm op} \]
can be extended to a lax monoidal functor.
\end{lemma}

In the following of this paper,
for simplicity,
we say that the underlying quasi-category
$\mathcal{M}$ of a monoidal category $\mathcal{M}^{\otimes}$
is a monoidal category and that 
the underlying functor $F: \mathcal{M}\to\mathcal{N}$
of a (lax) monoidal functor 
$F^{\otimes}: \mathcal{M}^{\otimes}\to \mathcal{N}^{\otimes}$
is a (lax) monoidal functor.

For a map of algebra objects $f: A\to B$
in $\mathcal{C}$,
the lax monoidal functor
$(f,f)_!^{\rm op}$
induces a map of quasi-categories
of algebra objects
\[ (f,f)_!^{\rm op}:
    {\rm Alg}({_A}\bmod_A(\mathcal{C}))^{\rm op})\longrightarrow
   {\rm Alg}({_B}\bmod_B(\mathcal{C})^{\rm op})\]
and hence we obtain a map
of quasi-categories of coalgebra objects
\[ (f,f)_!:
    {}_A{\rm CoAlg}_A(\mathcal{C})\longrightarrow
   {}_B{\rm CoAlg}_B(\mathcal{C}).\]
Therefore, for a coalgebra object $\Gamma$ of 
${_A}\bmod_A(\mathcal{C})$,
we obtain a coalgebra objects
\[ B\otimes_A\Gamma\otimes_A B=(f,f)_!(\Gamma) \]
of ${_B}\bmod_B(\mathcal{C})$.
In particular, since the monoidal unit $A$ in 
${_A}\bmod_A(\mathcal{C})$
is a coalgebra object,
we see that
\[ B\otimes_A B=(f,f)_!(A) \]
is a coalgebra object of ${_B}\bmod_B(\mathcal{C})$. 

In particular,
since the monoidal unit $\mathbf{1}$
is a coalgebra object of 
${}_{\mathbf{1}}\bmod_{\mathbf{1}}(\mathcal{C})
\simeq\mathcal{C}$,
we have a coalgebra object 
\[ A\otimes A =(f,f)_{!}(\mathbf{1}) \]
of
${}_A\bmod_A(\mathcal{C})$,
where $f:\mathbf{1}\to A$ is the unit map of $A$. 
We write 
\[ \Gamma(A)=(A,A\otimes A) \]
for simplicity and 
we call
$A\otimes A$-comodules $\Gamma(A)$-comodules
interchangeably.

Let $f: A\to B$ be a map of algebra objects of $\mathcal{C}$.
We denote by 
$f^*: {\rm LMod}_B(\mathcal{C})\to
      {\rm LMod}_A(\mathcal{C})$
the restriction of scalars functor.
Recall that $f^*$ is a right adjoint to 
the extension of scalars functor 
\[ f_{!}: {\rm LMod}_A(\mathcal{C})\rightarrow
          {\rm LMod}_B(\mathcal{C}),\]
which is given by $f_{!}(M)\simeq B\otimes_A M$.

\begin{theorem}
\label{thm:base-change-comodules}
Let $\Gamma$ be a coalgebra 
object in ${_A}\bmod_A(\mathcal{C})$ and
let $f: A\to B$ be a map of algebra objects
of $\mathcal{C}$.
The map $f$ induces a functor of quasi-categories 
\[ f_!: \lcomod_{(A,\Gamma)}(\mathcal{C})\longrightarrow
   \lcomod_{(B,\Sigma)}(\mathcal{C})\]
which covers the functor 
$f_{!}: \lmod_A(\mathcal{C})\to\lmod_B(\mathcal{C})$
through the forgetful functors,
where $\Sigma=(f,f)_!\Gamma$. 
\end{theorem}

\begin{proof}
This follows from
Proposition~\ref{prop:opposite-lax-tensored-module-functor}.
\qed\end{proof}

Suppose we have a map $f:A\to B$ 
of algebra objects of $\mathcal{C}$.
This induces
an adjunction of functors
\[ f_{\,!}:\lmod_A(\mathcal{C})\rightleftarrows 
   \lmod_B(\mathcal{C}):
   f^*.\]
Taking the opposite quasi-categories, 
we obtain an adjunction of functors
\[ f^{*\,\rm op}: \lmod_B(\mathcal{C})^{\rm op}
        \rightleftarrows \lmod_A(\mathcal{C})^{\rm op}: 
   f_{\,!}^{\,\rm op}.\]
By this adjunction,
we obtain an endomorphism monad 
\[ T\in{\rm Alg}({\rm End}({\rm LMod}_B(\mathcal{C})^{\rm op})), \]
and a quasi-category
\[ {\rm LMod}_T({\rm LMod}_B(\mathcal{C})^{\rm op}). \]
of left $T$-modules in ${\rm LMod}_B(\mathcal{C})^{\rm op}$
(see \cite[\S4.7.4]{Lurie2}).

The following theorem shows that
the definition of a quasi-category
of comodules is consistent with 
the definition in \cite{Torii1}.

\begin{theorem}
\label{thm:comparibility-of-definition-of-comodules}
There is an equivalence of quasi-categories
\[ \lcomod_{(B,B\otimes_A B)}(\mathcal{C})\simeq
   {\rm LMod}_T({\rm LMod}_B(\mathcal{C})^{\rm op})^{\rm op}. \]
\end{theorem}

\begin{proof}
Put $\mathcal{A}=\lmod_A(\mathcal{C})^{\rm op}$ 
and $\mathcal{B}=\lmod_B(\mathcal{C})^{\rm op}$.
We have an adjunction of functors
$F: \mathcal{B}\rightleftarrows \mathcal{A}:G$,
where $F=f^{*\,\rm op}$ and 
$G=f_{\,!}^{\,\rm op}$.
Since ${\rm LMod}_B(\mathcal{C})\simeq
{}_B{\rm BMod}(\mathcal{C})$,
we can regard $B\otimes_A B$
as an algebra object of ${\rm End}(\mathcal{B})$.
By \cite[Prop.~4.7.4.3]{Lurie2},
we can lift $G$ to 
$\overline{G}\in 
{\rm LMod}_{B\otimes_A B}({\rm Fun}(\mathcal{B},\mathcal{A}))$.
We can verify that the composition
\[ B\otimes_A B\stackrel{}{\longrightarrow}
   (B\otimes_A B)\circ G\circ F
   \stackrel{}{\longrightarrow} G\circ F\]
is an equivalence in ${\rm End}(\mathcal{B})$,
where the first map is induced by
the unit map of the adjunction $(F,G)$,
and the second map is induced by
the left $B\otimes_A B$-action on $G$
in ${\rm Fun}(\mathcal{B},\mathcal{A})$.
By \cite[Prop.~4.7.4.3]{Lurie2},
we see that $B\otimes_A B$ is equivalent to 
the endomorphism monad $T$. 
Hence we obtain an equivalence
between $\lmod_T(\lmod_B(\mathcal{C})^{\rm op})$ and 
$\lmod_{B\otimes_A B}({}_B{\rm BMod}(\mathcal{C})^{\rm op})$.
\qed\end{proof}

\subsection{Cotensor products for comodules 
in quasi-categories}

Let $(\mathcal{C},\otimes, \mathbf{1})$ be a stable 
homotopy theory.
In this subsection
we define a (derived) cotensor product
of comodules in $\mathcal{C}$.
In particular,
we define a (derived) functor of taking primitives
of comodules.
We also study a comodule structure on 
cotensor products.

Let $A$ be an algebra object of $\mathcal{C}$. 
Suppose $\Gamma$ is a coalgebra object of 
the quasi-category ${}_A\bmod_A(\mathcal{C})$
of $A$-$A$-bimodules in $\mathcal{C}$,
that is,
$\Gamma$ is an algebra object of 
the opposite monoidal quasi-category
${}_A\bmod_A(\mathcal{C})^{\rm op}$.

For a right $\Gamma$-comodule $M$
and a left $\Gamma$-comodule $N$,
we shall define a cotensor product
\[ M\square_{\Gamma} N. \]

We regard $M$ as an object
in $\rmod_{\Gamma}(\bmod_A(\mathcal{C})^{\rm op})$
and $N$ as an object
in $\lmod_{\Gamma}({}_A\bmod(\mathcal{C})^{\rm op})$.
We can construct a two-sided bar construction
\[ B_{\bullet}(M,\Gamma,N), \]
which is a simplicial object in 
$\mathcal{C}^{\rm op}$.
The simplicial object $B_{\bullet}(M,\Gamma,N)$
has the $n$th term 
given by 
\[ B_n(M,\Gamma,N)\simeq
   M\otimes_A\overbrace{\Gamma\otimes_A\cdots\otimes_A\Gamma}^n 
   \otimes_A N\]
with the usual structure maps.
We regard $B_{\bullet}(M,\Gamma,N)$ as a 
cosimplicial object 
\[ C^{\bullet}(M,\Gamma,N) \]
in $\mathcal{C}$ and 
define the cotensor product $M\square_{\Gamma}N$ 
to be the limit of the cosimplicial 
object $C^{\bullet}(M,\Gamma,N)$:
\[ M\square_{\Gamma}N=\, {\rm lim}_{N(\Delta)}\
   C^{\bullet}(M,\Gamma,N).\]

Now we regard $A$ as a right $A$-module and
suppose that $A$ is a right $\Gamma$-comodule
via $\eta_R: A\to A\otimes_A\Gamma\simeq\Gamma$.
We define a functor
\[ P: \lcomod_{(A,\Gamma)}\longrightarrow
      \mathcal{C}\]
by 
\[ P(N)=A\square_{\Gamma}N.\]
We consider the functor $P$ is a derived 
functor of taking primitives in $N$. 

Suppose we have algebra objects $A,B,C$ of $\mathcal{C}$.
The quasi-category of $B$-$A$-bimodules
${}_B\bmod_A(\mathcal{C})$ in $\mathcal{C}$
is right tensored over 
the monoidal quasi-category
${}_A\bmod_A(\mathcal{C})$
and the quasi-category
of $A$-$C$-bimodules
${}_A\bmod_C(\mathcal{C})$ in $\mathcal{C}$
is left tensored over 
the monoidal quasi-category
${}_A\bmod_A(\mathcal{C})$.
Let $\Gamma$ be a coalgebra object
of ${}_A\bmod_A(\mathcal{C})$.
We can define right $\Gamma$-comodule
objects of ${}_B\bmod_A(\mathcal{C})$
and left $\Gamma$-comodule objects
of ${}_A\bmod_C(\mathcal{C})$
in the same way as 
$\Gamma$-comodule objects of ${}_A\bmod_A(\mathcal{C})$.
Suppose we have a right $\Gamma$-comodule $M$  
of ${}_B\bmod_A(\mathcal{C})$
and a left $\Gamma$-comodule $N$ 
of ${}_A\bmod_C(\mathcal{C})$.
We can form 
the cobar construction
$C^{\bullet}(M,\Gamma,N)$
in ${}_B\bmod_C(\mathcal{C})$.
Hence the cotensor product
$M\square_{\Gamma}N$ is a $B$-$C$-bimodule
\[ M\square_{\Gamma}N\in {}_B\bmod_C(\mathcal{C}).\]

Let $\Sigma$ be a coalgebra object
of ${}_B\bmod_B(\mathcal{C})$.
Now suppose $M$ is a $(\Sigma,\Gamma)$-bicomodule object
of ${}_B\bmod_A(\mathcal{C})$,
that is,
$M$ is a $(\Sigma,\Gamma)$-bimodule object
of ${}_B\bmod_A(\mathcal{C})^{\rm op}$.
In general,
the cotensor product
$M\square_{\Gamma}N$ does not
support a left $\Sigma$-comodule structure.
The following proposition gives
us a sufficient condition for $M\square_{\Gamma}N$
to be a left $\Sigma$-comodule object
of ${}_B\bmod_C(\mathcal{C})$
induced by the left $\Sigma$-comodule
structure on $M$.

\begin{proposition}
\label{prop:cotensor-product-comodule-structure}
Let $M$ be a $(\Sigma,\Gamma)$-bicomodule object
of ${}_B\bmod_A(\mathcal{C})$ and let
$N$ be a left $\Gamma$-comodule object
of ${}_A\bmod_C(\mathcal{C})$.
If the canonical map
\[ \overbrace{\Sigma\otimes_B\cdots\otimes_B\Sigma}^r
   \otimes_B (M\square_{\Gamma}N)\longrightarrow
   (\overbrace{\Sigma\otimes_B\cdots\otimes_B\Sigma}^r\otimes_B
   M)\square_{\Gamma}N\]
is an equivalence in ${}_B\bmod_C(\mathcal{C})$
for all $r>0$,
then the left $\Sigma$-comodule structure on $M$
induces a left $\Sigma$-comodule structure
on $M\square_{\Gamma}N$.
\end{proposition}

In order to prove 
Proposition~\ref{prop:cotensor-product-comodule-structure},
we need the following lemma.
 
\begin{lemma}
\label{lem:general-convergence-colimit-criterion}
Let $\mathcal{M}$ be a monoidal quasi-category,
$A$ an algebra object of $\mathcal{M}$, and
$\mathcal{D}$ a quasi-category left-tensored over $\mathcal{M}$.
Suppose we have a diagram
$X: K\to \lmod_A(\mathcal{D})$,
where $K$ is a simplicial set.
We set $Y=\pi\circ X:K\to\mathcal{D}$,
where $\pi:\lmod_A(\mathcal{D})\to\mathcal{D}$
is the forgetful functor.
We assume that there exists a colimit
${\rm colim}_K^{\mathcal{D}}(A^{\otimes r}\otimes Y)$
in $\mathcal{D}$ for all $r\ge 0$.
If the canonical map 
${\rm colim}_K^{\mathcal{D}} (A^{\otimes r}\otimes Y)
\to A^{\otimes r}\otimes
{\rm colim}_K^{\mathcal{D}}\, Y$ 
is an equivalence for all $r>0$,
then there exists a colimit of $X$ in $\lmod_A(\mathcal{D})$
and the forgetful functor
$\pi:\lmod_A(\mathcal{D})\to\mathcal{D}$
preserves the colimit.
\end{lemma}

\begin{proof}
We use the notation in \cite[\S4.2.2]{Lurie2}.
Let $\mathcal{D}^{\circledast}$ and 
$\mathcal{M}^{\circledast}$
be quasi-categories defined in
\cite[Notation~4.2.2.16]{Lurie2}.
We have maps $\mathcal{D}^{\circledast}\stackrel{q}{\to} 
\mathcal{M}^{\circledast}\stackrel{p}{\to} N(\Delta)^{\rm op}$,
where $p$ and $p\circ q$ are coCartesian fibrations
by \cite[Remark.~4.2.2.24]{Lurie2}.
Furthermore,
$q$ is a categorical fibration 
by \cite[Remark.~4.2.2.18]{Lurie2}
and
a locally coCartesian fibration 
by \cite[Lem.~4.2.2.19]{Lurie2}.
Note that there is an equivalence 
of quasi-categories 
$\mathcal{D}^{\circledast}_{[s]}\simeq
\mathcal{M}^{\circledast}_{[s]}\times\mathcal{D}$
and the restriction $q_{[s]}: \mathcal{D}^{\circledast}_{[s]}\to
\mathcal{M}^{\circledast}_{[s]}$
is the projection
for any $[s]\in N(\Delta)^{\rm op}$.

We have simplicial models of algebra and module objects
in quasi-categories
(see \cite[\S4.1.2 and \S4.2.2]{Lurie2}).
We have a full subcategory
${}^{\Delta}{\rm Alg}(\mathcal{M})$
of the quasi-category ${\rm Fun}_{N(\Delta)^{\rm op}}
(N(\Delta)^{\rm op},\mathcal{M}^{\circledast})$
which is equivalent to 
the quasi-category ${\rm Alg}(\mathcal{M})$
of algebra objects of $\mathcal{M}$
(see \cite[Def.~4.1.2.14 and Prop.~4.1.2.15]{Lurie2}).
We denote by $A': N(\Delta)^{\rm op}\to \mathcal{M}^{\circledast}$
the corresponding simplicial object of $\mathcal{M}^{\circledast}$
to $A\in {\rm Alg}(\mathcal{M})$.

We form a pullback diagram
\[ \xymatrix{
     \mathcal{N} \ar[r]^j \ar[d]_{q'} & 
     \mathcal{D}^{\circledast} \ar[d]^q \\
     N(\Delta)^{\rm op} \ar[r]^{A'} & \mathcal{M}^{\circledast},\\
}\]
where $q'$ is a locally coCartesian fibration and
a categorical fibration.
Note that the fiber $\mathcal{N}_{[n]}$
of $q'$ over $[n]$ is equivalent
to $\mathcal{D}$ for all $[n]\in N(\Delta)^{\rm op}$.
We have a full subcategory
${}^{\Delta}\lmod_{A'}(\mathcal{D})$
of ${\rm Fun}_{N(\Delta)^{\rm op}}
(N(\Delta)^{\rm op},\mathcal{N})$,
which is equivalent to 
$\lmod_A(\mathcal{D})$
(see \cite[Cor.~4.2.2.15]{Lurie2}).
An object $G$ of ${\rm Fun}_{N(\Delta)^{\rm op}}
(N(\Delta)^{\rm op},\mathcal{N})$
belongs to $\lmod_{A'}(\mathcal{D})$
if and only if 
the edge $(j\circ G)(\alpha^{\rm op}):
(j\circ G)([n])\to (j\circ G)([m])$
is $p\circ q$-coCartesian
for any convex map $\alpha: [m]\to [n]$ in
$\Delta$ such that $\alpha(m)=n$.


We denote by 
$f: K\times N(\Delta)^{\rm op}\to\mathcal{N}$
the map corresponding to the diagram
$X\in {\rm Fun}(K,\lmod_A(\mathcal{D}))$.
We let $g: K^{\vartriangleright}\times N(\Delta)^{\rm op}\to 
N(\Delta)^{\rm op}$ be the projection.
We have a commutative diagram
\[ \xymatrix{
    K\times N(\Delta)^{\rm op} \ar[r]^{f} 
    \ar@{^{(}->}[d] &
    \mathcal{N} \ar[d]^{q'}\\
    K^{\vartriangleright}\times N(\Delta)^{\rm op}
    \ar[r]^g&
    N(\Delta)^{\rm op},\\    
}\]
where the left vertical arrow is the inclusion.
We shall show that there is a $q'$-left Kan extension
$\overline{f}:K^{\vartriangleright}\times N(\Delta)^{\rm op}
\to \mathcal{N}$ 
which makes the whole diagram commutative,
and that the adjoint map gives rise to a
colimit diagram 
$K^{\vartriangleright}\to \lmod_{A'}(\mathcal{D})\simeq
\lmod_A(\mathcal{D})$.

Let 
$f_{[n]}:
   K\to \mathcal{N}_{[n]}\simeq\mathcal{D}$
be the restriction of the map $f$
over $[n]\in N(\Delta)^{\rm op}$,
which is equivalent to $Y$.
Since $Y$ has a colimit in $\mathcal{D}$
by the assumption,
we obtain a colimit diagram
$\overline{f}_{[n]}: K^{\vartriangleright}\to 
\mathcal{N}_{[n]}\simeq\mathcal{D}$
that is an extension of $f$.

Let 
$\alpha:[n]\to [m]$ be an edge in $N(\Delta)^{\rm op}$.
Since $q'$ is a locally coCartesian fibration,
$\alpha$ induces a functor
$\alpha_!: \mathcal{N}_{[n]}\to
\mathcal{N}_{[m]}$.
The composition
$\alpha_!\circ f_{[n]}:
K\to \mathcal{N}_{[m]}\simeq \mathcal{D}$
is equivalent to
$A^{\otimes r}\otimes Y$ for some $r\ge 0$.
This implies that 
$\alpha_!\circ\overline{f}_{[n]}$
is a colimit diagram in $\mathcal{N}_{[m]}\simeq\mathcal{D}$
by the assumption
that the canonical map
${\rm colim}^{\mathcal{D}}_K(A^{\otimes r}\otimes Y)
\to A^{\otimes r}\otimes {\rm colim}_K^{\mathcal{D}}\,Y$
is an equivalence.
Hence we see that $i_{[n]}\circ\overline{f}_{[n]}$ 
is a $q'$-colimit diagram
by \cite[Prop.~4.3.1.10]{Lurie1},
where $i_{[n]}:\mathcal{N}_{[n]}\hookrightarrow
\mathcal{N}$ is the inclusion.

By the dual of \cite[Lem.~3.2.2.9(1)]{Lurie2},
there exists a $q'$-left Kan extension
$\overline{f}:K^{\vartriangleright}\times N(\Delta)^{\rm op}
\to \mathcal{N}$ of $f$
such that $q'\circ\overline{f}=g$. 
The restriction of $\overline{f}$
to $K^{\vartriangleright}\times\{[n]\}$
is equivalent to 
$i_{[n]}\circ\overline{f}_{[n]}$
for all $[n]\in N(\Delta)^{\rm op}$.

We consider 
the adjoint map $K^{\vartriangleright}\to 
{\rm Fun}(N(\Delta)^{\rm op},\mathcal{N})$ 
of $f$.
By the dual of \cite[Lem.~3.2.2.9(2)]{Lurie2},
this map is a $(q')^{N(\Delta)^{\rm op}}$-colimit diagram,
where $(q')^{N(\Delta^{\rm op})}:
{\rm Fun}(N(\Delta)^{\rm op},\mathcal{N})\to
{\rm Fun}(N(\Delta)^{\rm op},N(\Delta)^{\rm op})$
is induced by $q'$.
Since $g$ is the projection,
we see that it factors through 
${\rm Fun}_{N(\Delta)^{\rm op}}(N(\Delta)^{\rm op},\mathcal{N})$
and we obtain a map
$\widehat{f}:K^{\vartriangleright}\to
{\rm Fun}_{N(\Delta)^{\rm op}}(N(\Delta)^{\rm op},\mathcal{N})$.
By \cite[Prop.~4.3.1.5(4)]{Lurie1},
we see that $\widehat{f}$
is a colimit diagram.

We shall show that $\widehat{f}$ factors through
$\lmod_{A'}(\mathcal{D})$.
Note that 
the restriction of $\widehat{f}$ to $K$
factors through $\lmod_{A'}(\mathcal{D})$.
Let $F=\widehat{f}(\infty)\in 
{\rm Fun}_{N(\Delta)^{\rm op}}(N(\Delta)^{\rm op},\mathcal{N})$,
where $\infty$ is the cone point of $K^{\vartriangleright}$.
Since $\overline{f}_{[n]}$ 
is a colimit diagram extending $Y$,
we have $F([n])\simeq{\rm colim}^{\mathcal{D}}_KY$
in $\mathcal{N}_{[n]}\simeq\mathcal{D}$
for any $[n]\in N(\Delta)^{\rm op}$.
Let $\alpha: [m]\to [n]$ be a convex map
in $\Delta$ such that $\alpha(m)=n$.
The induced functor
$\alpha_!: \mathcal{N}_{[n]}\to \mathcal{N}_{[m]}$
is identified with the identity functor
of $\mathcal{D}$.
This implies 
that $(j\circ F)(\alpha^{\rm op}): 
(j\circ F)([n])\to (j\circ F)([m])$
is a $p\circ q$-coCartesian edge. 
Hence $\widehat{f}$ factors through
the full subcategory ${}^{\Delta}\lmod_{A'}(\mathcal{D})$
and the map
$\widehat{f}: K^{\vartriangleright}\to 
{}^{\Delta}\lmod_{A'}(\mathcal{D})$ is a colimit diagram.

By the construction of $\widehat{f}$,
the composition
$\pi\circ \widehat{f}: K^{\vartriangleright}\to
\mathcal{D}$ is also a colimit diagram,
where $\pi:{}^{\Delta}\lmod_{A'}(\mathcal{D})\to
\mathcal{D}$ is the forgetful functor.
This completes the proof.
\qed
\end{proof}

\begin{proof}
[Proof of Proposition~\ref{prop:cotensor-product-comodule-structure}]
We shall apply Lemma~\ref{lem:general-convergence-colimit-criterion}.
We have the monoidal quasi-category
${}_B\bmod_B(\mathcal{C})^{\rm op}$,
the algebra object $\Sigma$
of ${}_B\bmod_B(\mathcal{C})^{\rm op}$,
and
the quasi-category ${}_B\bmod_C(\mathcal{C})^{\rm op}$
left tensored over ${}_B\bmod_B(\mathcal{C})^{\rm op}$.
By the bar construction,
we have a simplicial object
$B_{\bullet}(M,\Gamma,N)$
of $\lmod_{\Sigma}({}_B\bmod_C(\mathcal{C})^{\rm op})$.
By the assumption,
the canonical map
\[ {\rm colim}_{N(\Delta)^{\rm op}}
   B_{\bullet}(\Sigma^{\otimes_B r}\otimes_B
   M,\Gamma,N)\to
   \Sigma^{\otimes_B r}\otimes_B
   {\rm colim}_{N(\Delta)^{\rm op}}B_{\bullet}(M,\Gamma,N) \]
is an equivalence in
${}_B\bmod_C(\mathcal{C})^{\rm op}$
for all $r>0$.
By Lemma~\ref{lem:general-convergence-colimit-criterion},
there exists a colimit
of $B_{\bullet}(M,\Gamma,N)$ in
$\lmod_{\Sigma}({}_B\bmod_C(\mathcal{C})^{\rm op})$
and the colimit is created
in ${}_B\bmod_C(\mathcal{C})^{\rm op}$.
Hence we see that
the cosimplicial object
$C^{\bullet}(M,\Gamma,N)$
has a limit in $\lcomod_{\Sigma}({}_B\bmod_C(\mathcal{C}))$
and the underlying object of the limit is
$M\square_{\Gamma}N$
in ${}_B\bmod_C(\mathcal{C})$. 
\qed
\end{proof}

\subsection{Equivalence of quasi-categories of comodules}

Let $(\mathcal{C},\otimes,\mathbf{1})$ be a stable homotopy theory
and let $A$ be an algebra object of $\mathcal{C}$.
In this subsection we study the relationship
between the localization of $\mathcal{C}$ with respect
to $A$ and the quasi-category of $\Gamma(A)$-comodules
in $\mathcal{C}$.

We regard $A$ as a right $A$-module and
the map of right $A$-modules 
$\eta_R: A\simeq S\otimes A \to A\otimes A$ induces 
a right $\Gamma(A)$-comodule structure on $A$.
Since $\rcomod_{\Gamma(A)}(\mathcal{C})$ 
is left tensored over $\mathcal{C}$,
we have $X\otimes A\in\rcomod_{\Gamma(A)}(\mathcal{C})$
for any $X\in\mathcal{C}$.

Recall that 
we have a cosimplicial object
\[ C^{\bullet}(A,A\otimes A,M)\]
in $\mathcal{C}$ by cobar construction
for $M\in \lcomod_{\Gamma(A)}(\mathcal{C})$.
The totalization of $C^{\bullet}(A,A\otimes A,M)$
is $P(M)$.
In the same way as in \cite[Prop.~5.1]{Torii1},
we obtain an adjunction of functors
\[ A\otimes (-):
   \mathcal{C}\rightleftarrows \lcomod_{\Gamma(A)}(\mathcal{C}):
   P.\]

Let $C^{\bullet}: N(\Delta)\to \mathcal{C}$
be a cosimplicial object in $\mathcal{C}$.
We recall the Tot tower associated to $C^{\bullet}$.
For $r\ge 0$,
we denote by $\Delta^{\le r}$
the full subcategory of $\Delta$
spanned by $\{[0],[1],\ldots,[r]\}$.
We denote by
$C^{\bullet}|_{N(\Delta^{\le r})}$
the restriction of $C^{\bullet}$
to $N(\Delta^{\le r})$.
We recall that 
${\rm Tot}_r(C^{\bullet})$ 
is defined to be the limit
of $C^{\bullet}|_{N(\Delta^{\le r})}$
in $\mathcal{C}$. 
The inclusion $\Delta^{\le r}\hookrightarrow\Delta^{\le r+1}$
induces a map
${\rm Tot}_{r+1}(C^{\bullet})\to {\rm Tot}_r(C^{\bullet})$
for $r\ge 0$
and we obtain a tower
$\{{\rm Tot}_r C^{\bullet}\}_{r\ge 0}$.
Note that the limit of the tower is 
equivalent to ${\rm Tot}(C^{\bullet})$:
\[ {\rm Tot}(C^{\bullet})\simeq
   \lim_r{\rm Tot}_r(C^{\bullet}).\]
If there is a coaugmentation $D\to C^{\bullet}$,
then we obtain a map of towers
$c(D)\to \{{\rm Tot}_r(C^{\bullet})\}_{r\ge 0}$,
where $c(D)$ is the constant tower on $D$.

We denote by ${\rm Pro}(\mathcal{C})$
the quasi-category of pro-objects in $\mathcal{C}$
(see \cite[\S3]{Mathew}).
We have a fully faithful embedding
$\mathcal{C}\hookrightarrow {\rm Pro}(\mathcal{C})$.
We say that an object of ${\rm Pro}(\mathcal{C})$
is constant if it is equivalent to
an object in the image of the embedding
$\mathcal{C}\hookrightarrow{\rm Pro}(\mathcal{C})$. 

\begin{lemma}\label{lemma:splitting-tensor-cosimplicial}
For any $M\in \lcomod_{\Gamma(A)}(\mathcal{C})$,
the cosimplicial object
$A\otimes C^{\bullet}(A,A\otimes A,M)$
is split,
and hence 
the tower $\{{\rm Tot}^r(A\otimes C^{\bullet}(A,A\otimes A,M))\}$
associated to the cosimplicial object
$A\otimes C^{\bullet}(A,A\otimes A,M)$
is equivalent to the constant object $M$
in ${\rm Pro}(\mathcal{C})$.
\end{lemma}

\begin{proof}
We have an isomorphism of cosimplicial objects
\[ A\otimes C^{\bullet}(A,A\otimes A,M)\cong
   C^{\bullet}(A\otimes A, A\otimes A, M).\]
The lemma follows from the fact that
$C^{\bullet}(A\otimes A,A\otimes A, M)$
is a split cosimplicial object.
\qed\end{proof}

A full subcategory $\mathcal{I}\subset \mathcal{C}$
is said to be an ideal if $X\otimes Y\in\mathcal{I}$
whenever $X\in \mathcal{C}$ and $Y\in\mathcal{I}$
(cf.~{\cite[Definition~2.16]{Mathew}}).
A full subcategory $\mathcal{D}\subset \mathcal{C}$
is said to be thick if $\mathcal{D}$
is closed under finite limits and colimits and under retracts.
If, furthermore,
$\mathcal{D}$ is an ideal,
we say that $\mathcal{D}$ is a thick tensor ideal
(cf.~{\cite[Definition~3.16]{Mathew}}).
Given a collection of objects in $\mathcal{C}$,
the thick tensor ideal generated by them
is the smallest thick tensor ideal
containing the collection.
Let $A$ be an algebra object of $\mathcal{C}$.
An object $X\in \mathcal{C}$
is said to be $A$-nilpotent if
$X$ belongs to the thick tensor ideal
generated by $A$.

\begin{lemma}\label{lemma:A-nilpotent-pro-constancy}
Let $A$ be an algebra object of $\mathcal{C}$.
If the unit $\mathbf{1}$ is $A$-nilpotent,
then the tower associated to the cosimplicial object
$C^{\bullet}(A,A\otimes A, M)$
is equivalent to the constant object
$P(M)$ in ${\rm Pro}(\mathcal{C})$
for any $M\in \lcomod_{\Gamma(A)}(\mathcal{C})$.
\end{lemma}

\begin{proof}
Let $\mathcal{I}$ be the class
of objects $X$ in $\mathcal{C}$ such that
the tower associated to
the cosimplicial object $X\otimes C^{\bullet}(A,A\otimes A,M)$
is equivalent to a constant object in ${\rm Pro}(\mathcal{C})$.
We see that $\mathcal{I}$ is a thick tensor ideal
of $\mathcal{C}$ and 
contains $A$ by Lemma~\ref{lemma:splitting-tensor-cosimplicial}.
Hence $\mathcal{I}$ contains the unit $\mathbf{1}$
by the assumption.
This implies that the tower associated to
$C^{\bullet}(A,A\otimes A,M)$
is equivalent to the constant object
$\lim_{N(\Delta)} C^{\bullet}(A,A\otimes A,M)\simeq P(M)$
in ${\rm Pro}(\mathcal{C})$.
\qed\end{proof}

For any $X\in\mathcal{C}$ 
and $M\in \lcomod_{\Gamma(A)}(\mathcal{C})$,
we have an equivalence of cosimplicial objects
$X\otimes C^{\bullet}(A,A\otimes A,M)\simeq
C^{\bullet}(X\otimes A,A\otimes A,M)$.
This induces a natural map
\[ X\otimes P(M)\longrightarrow
   (X\otimes A){\square}_{A\otimes A}M.\]

\begin{proposition}
\label{prop:APM=M}
Let $A$ be an algebra object of $\mathcal{C}$.
If the unit $\mathbf{1}$ is $A$-nilpotent,
then the natural map
$X\otimes P(M)\to
   (X\otimes A){\square}_{A\otimes A}M$
is an equivalence for any $X\in\mathcal{C}$
and $M\in {\rm Comod}_{\Gamma(A)}(\mathcal{C})$.
\end{proposition}

\begin{proof}
By Lemma~\ref{lemma:A-nilpotent-pro-constancy},
the tower associated to the cosimplicial object
$C^{\bullet}(A,A\otimes A,M)$
is equivalent to the constant object
$P(M)$ in ${\rm Pro}(\mathcal{C})$.
This implies that
the tower
\[ \{{\rm Tot}^r(X\otimes C^{\bullet}(A,A\otimes A,M))\} \]
associated to the cosimplicial object
$X\otimes C^{\bullet}(A,A\otimes A,M)$
is also equivalent to the constant
object $X\otimes P(M)$
in ${\rm Pro}(\mathcal{C})$.
By the equivalence of cosimplicial objects
$X\otimes C^{\bullet}(A,A\otimes A,M)\simeq
C^{\bullet}(X\otimes A,A\otimes A,M)$,
we obtain the equivalence
$X\otimes P(M)\simeq 
   (X\otimes A){\square}_{A\otimes A}M$.
\qed\end{proof}

\begin{corollary}
\label{cor:primitive-decomposition--comodules}
Let $A$ be an algebra object of $\mathcal{C}$.
If the unit $\mathbf{1}$ is $A$-nilpotent,
then the counit map
\[ A\otimes P(M)\to M\]
is an equivalence for any $M\in \lcomod_{\Gamma(A)}(\mathcal{C})$.
\end{corollary}

\begin{proof}
By Proposition~\ref{prop:APM=M},
we have a natural equivalence
$A\otimes P(M)\simeq (A\otimes A)\square_{A\otimes A}M$.
The corollary follows from the fact that
$(A\otimes A){\square}_{A\otimes A}M\simeq M$.
\qed\end{proof}

We consider the localization of $\mathcal{C}$
with respect to $A$.
A morphism $f: X\to Y$ in $\mathcal{C}$
is said to be an $A$-equivalence
if $A\otimes f$ is an equivalence.
We denote by $L_A\mathcal{C}$
the localization of $\mathcal{C}$
with respect to the class of 
$A$-equivalences.

The following theorem is a slight generalization
of \cite[Prop.~3.21]{Mathew}.

\begin{theorem}
\label{theorem:general-equivalence-comodules}
Let $A$ be an algebra object of $\mathcal{C}$.
If the unit $\mathbf{1}$ is $A$-nilpotent,
then $L_A\mathcal{C}$ is equivalent to
$\lcomod_{\Gamma(A)}(\mathcal{C})$.
We have
an adjoint equivalence
\[ A\otimes (-): L_A\mathcal{C}\rightleftarrows
   \lcomod_{\Gamma(A)}(\mathcal{C}):
   P,\]
and we can identify 
the functor $A\otimes(-):\mathcal{C}\to
\lcomod_{\Gamma(A)}(\mathcal{C})$
with the localization 
$\mathcal{C}\to L_A\mathcal{C}$.
\end{theorem}

\begin{proof}
We have an adjoint pair of functors
$A\otimes (-): \mathcal{C}\rightleftarrows
   \lcomod_{\Gamma(A)}(\mathcal{C}):
   P$.
Clearly,
$A\otimes f$ is an equivalence
in $\lcomod_{\Gamma(A)}(\mathcal{C})$
if and only if $f$ is an $A$-equivalence
for any morphism $f$ in $\mathcal{C}$.
Hence it suffices to show that
the right adjoint $P$ is fully faithful.
The counit map
$\epsilon: A\otimes P(M)\to M$
is an equivalence 
for any $M\in \lcomod_{\Gamma(A)}(\mathcal{C})$
by Corollary~\ref{cor:primitive-decomposition--comodules}.
Hence we see that $P$
is fully faithful.
\qed\end{proof}



\section{Comodules in the quasi-category of spectra}
\label{sec:comodule-spectra}

In this section we study quasi-category
of comodules in spectra.
Using the Bousfield-Kan spectral sequence
and the results in \cite{Hovey-Strickland},
we show that 
the quasi-category of comodules associated
to a Landweber exact $\mathbb{S}$-algebra depends 
only on its height.
We also show that the $E(n)$-local category
is equivalent to the quasi-category
of comodules over the coalgebra $E(n)\otimes E(n)$.
In \cite{Torii1} we considered the model category
of $F_n$-modules in the category of 
discrete symmetric $\mathbb{G}_n$-spectra,
where $\mathbb{G}_n$ is the extended Morava stabilizer group
and $F_n$ is a discrete model of the Morava $E$-theory spectrum
$E_n$.
We show that the category of $F_n$-modules
in the discrete symmetric $\mathbb{G}_n$-spectra
models the $K(n)$-local category.

\subsection{Cotensor product and
its derived functor in algebraic setting} 

In this subsection
we recall some properties of the category of comodules
over a coalgebra in algebraic setting.
We study the derived functor of cotensor product of comodules
and show that the derived functor can be described
by the cobar complex in some situations.
The content in this section is not new.
Our main reference is \cite[Appendix~A1.2]{Ravenel1}.

Let ${\rm Ab}_*$ be the category
of graded abelian groups.
Let $A$ be a monoid object in ${\rm Ab}_*$.
We denote by ${}_A{\rm BMod}_A({\rm Ab}_*)$
the category of $A$-$A$-bimodules in ${\rm Ab}_*$.
The category ${}_A{\rm Bmod}_A({\rm Ab}_*)$
is a monoidal category with
the tensor product $\otimes_A$
and the unit $A$.
We denote by 
${}_A{\rm BMod}_A({\rm Ab}_*)^{\rm op}$
the opposite monoidal category.
A coalgebra in ${}_A{\rm BMod}_A({\rm Ab}_*)$
is defined to be a monoid object 
in ${}_A{\rm BMod}_A({\rm Ab}_*)^{\rm op}$.
In other word,
a coalgebra $\Gamma$
is an $A$-$A$-bimodule
equipped with maps
\[ \begin{array}{rl}
     \psi:& \Gamma\longrightarrow \Gamma\otimes_A\Gamma,\\[2mm]
     \epsilon:&\Gamma\longrightarrow A\\
\end{array}\]
in ${}_A{\rm BMod}_A({\rm Ab}_*)$ satisfying
the coassociativity and counit conditions.

We denote by ${}_A{\rm CoAlg}_A({\rm Ab}_*)$
the category of coalgebras in ${}_A{\rm BMod}_A({\rm Ab}_*)$.
By definition,
we have an equivalence
\[ {}_A{\rm CoAlg}_A({\rm Ab}_*)\simeq
   {\rm Alg}({}_A{\rm BMod}_A({\rm Ab}_*)^{\rm op})^{\rm op}.\]

We denote by
${\rm LMod}_A({\rm Ab}_*)$ 
the category of left $A$-modules
and by ${\rm RMod}_A({\rm Ab}_*)$
the category of right $A$-modules,
respectively.
Let $\Gamma\in {}_A{\rm CoAlg}_A({\rm Ab}_*)$.
A left $\Gamma$-comodule is defined to be a left $A$-module
$M$ equipped with a map
\[ \psi: M\longrightarrow \Gamma\otimes_A M\]
in ${\rm LMod}_A({\rm Ab}_*)$ satisfying
the coassociativity and counit conditions.
A right $\Gamma$-comodule is defined in
the similar fashion.
We denote by
${\rm LComod}_{(A,\Gamma)}({\rm Ab}_*)$
the category of left $\Gamma$-comodules
and by
${\rm RComod}_{(A,\Gamma)}({\rm Ab}_*)$
the category of right $\Gamma$-comodules,
respectively.

The following lemma is obtained 
in the same way as in \cite[Thm.~A1.1.3 and Lem.~A1.2.2]{Ravenel1}.

\begin{lemma}
If $\Gamma$ is flat as a right $A$-module,
then ${\rm LComod}_{(A,\Gamma)}({\rm Ab}_*)$
is an abelian category with enough injectives.
\end{lemma}

In the following of this subsection
we assume that a coalgebra $\Gamma\in {}_A{\rm CoAlg}_A({\rm Ab}_*)$
is flat as a right $A$-module and a left $A$-module.
Hence we can do homological algebra
in ${\rm LComod}_{(A,\Gamma)}({\rm Ab}_*)$.
We abbreviate ${\rm Hom}_{\lcomod_{(A,\Gamma)}({\rm Ab}_*)}(-,-)$
as ${\rm Hom}_{\Gamma}(-,-)$.
For a left $\Gamma$-comodule $M$,
we define
\[ {\rm Ext}^i_{\Gamma}(M,-)\]
to be the $i$th right derived functor
of
\[ {\rm Hom}_{\Gamma}(M,-):
   {\rm LComod}_{(A,\Gamma)}({\rm Ab}_*)\longrightarrow {\rm Ab}_*.\]

For $M\in {\rm RComod}_{(A,\Gamma)}({\rm Ab}_*)$
and $N\in {\rm LComod}_{(A,\Gamma)}({\rm Ab}_*)$,
we denote by
$M\square_{\Gamma} N$
the cotensor product of $M$ and $N$ over $\Gamma$
(see, for example, \cite[Definition~A1.1.4]{Ravenel1}).
We consider the functor
\[ M\square_{\Gamma}(-):
   {\rm LComod}_{(A,\Gamma)}({\rm Ab}_*)\longrightarrow
   {\rm Ab}_*.\]
We define
\[ {\rm Cotor}^i_{\Gamma}(M,-) \]
to be the $i$th right derived functor
of $M\square_{\Gamma}(-)$.
Note that if $M$ is flat as a right $A$-module,
then ${\rm Cotor}^0_{\Gamma}(M,N)\cong M\square_{\Gamma}N$
since $M\square_{\Gamma}(-)$ is left exact in this case.

Let $M$ be a left $\Gamma$-comodule
that is finitely generated and projective as a left $A$-module.
There is a right $\Gamma$-comodule structure
on ${\rm Hom}_A(M,A)$ and 
we have a natural isomorphism
\[ {\rm Hom}_{\Gamma}(M,N)\cong
   {\rm Hom}_A(M,A)\square_{\Gamma}N\]
for any $N\in {\rm LComod}_{(A,\Gamma)}({\rm Ab}_*)$
(cf.~\cite[Lem.~A1.1.6]{Ravenel1}).
This implies that
there is a natural isomorphism
\[ {\rm Ext}^i_{\Gamma}(M,N)\cong
   {\rm Cotor}^i_{\Gamma}({\rm Hom}_A(M,A),N)\]
for any $i\ge 0$.   
In particular,
we have a natural isomorphism
\[ {\rm Ext}^i_{\Gamma}(A,N)\cong
   {\rm Cotor}^i_{\Gamma}(A,N)\]
for any $N\in {\rm LComod}_{(A,\Gamma)}({\rm Ab}_*)$ and $i\ge 0$.

For a right $\Gamma$-comodule $M$ and  
a left $\Gamma$-comodule $N$,
we have a cosimplicial object 
\[ C^{\bullet}(M,\Gamma,N)\]
in ${\rm Ab}_*$ obtained by the cobar construction.
In particular,
we have
\[ C^r(M,\Gamma,N)=M\otimes_A\Gamma^{\otimes_A r}\otimes_AN\]
for $r\ge 0$.
The cobar complex $C^*(M,\Gamma,N)$
is the associated cochain complex.
The normalized cobar complex
$\overline{C}{}^*(M,\Gamma,N)$
is a subcomplex of
$C^*(M,\Gamma,N)$ that is given by
\[ \overline{C}^r(M,\Gamma,N)=
   M\otimes_A\overline{\Gamma}^{\otimes_A r}\otimes_A N,\]
for $r\ge 0$,
where $\overline{\Gamma}=\ker\epsilon$.

We say that a left $\Gamma$-comodule $N$
is relatively injective if
$N$ is a direct summand of
$\Gamma\otimes_A N'$ as a left $\Gamma$-comodule
for some left $A$-module $N'$.
For a left $\Gamma$-comodule $N$,
the map $\psi: N\to \Gamma\otimes_AN$
induces an augmentation
$N\to C^*(\Gamma,\Gamma,N)$.
This gives a resolution
of $N$ in ${\rm LComod}_{(A,\Gamma)}({\rm Ab}_*)$
by relative injectives.
Note that the resolution
is split in ${\rm LMod}_A({\rm Ab}_*)$.
The splitting is given by
\[ \epsilon\otimes 1^{\otimes r}\otimes 1:
   \Gamma\otimes_A\Gamma^{\otimes_A r}\otimes_A N
   \longrightarrow
   A\otimes_A\Gamma^{\otimes_A r}\otimes_A N\cong
   \Gamma\otimes_A\Gamma^{\otimes_A (r-1)}\otimes_AN.\] 
Similarly,
we have a resolution
$N\to \overline{C}{}^*(\Gamma,\Gamma,N)$
of $N$ in ${\rm LComod}_{(A,\Gamma)}({\rm Ab}_*)$
by relative injectives
that is split in ${\rm LMod}_A({\rm Ab}_*)$.
By the proof of \cite[Lemma~A1.2.9]{Ravenel1},
the cobar complex
$C^*(\Gamma,\Gamma,N)$
is cochain homotopy equivalent
to the normalized cobar complex
$\overline{C}{}^*(\Gamma,\Gamma,N)$.
Since $C^*(M,\Gamma,N)\cong
M\square_{\Gamma}C^*(\Gamma,\Gamma,N)$
and
$\overline{C}{}^*(M,\Gamma,N)\cong
M\square_{\Gamma}\overline{C}{}^*(\Gamma,\Gamma,N)$,
this implies that
\[ H^*(C^*(M,\Gamma,N))\cong
   H^*(\overline{C}{}^*(M,\Gamma,N))\]
for any $M\in {\rm RComod}_{(A,\Gamma)}({\rm Ab}_*)$
and $N\in {\rm LComod}_{(A,\Gamma)}({\rm Ab}_*)$.

The following proposition
is obtained in the same way as in \cite[Cor.~A1.2.12]{Ravenel1}.

\begin{proposition}
If $M$ is flat as a right $A$-module,
then
\[ H^*(C^*(M,\Gamma,N))\cong
   {\rm Cotor}_{\Gamma}^*(M,N).\]
In particular, we have
\[ H^*(C^*(A,\Gamma,N))\cong
   {\rm Ext}_{\Gamma}^*(A,N).\]
\end{proposition}

\subsection{Bousfield-Kan spectral sequences}
\label{sec:Bousfield-Kan-SS}

In this subsection 
we work in the quasi-category
of spectra $\sp$ and study
the Bousfield-Kan spectral sequence abutting to
the homotopy groups of cotensor products of comodules.

Note that $\sp$ is a presentable stable symmetric monoidal 
category in which the tensor product commutes 
with all colimits separately 
in each variable.
We use $\otimes$ for the tensor product in $\sp$
instead of $\wedge$.
We denote by $\mathbb{S}$ the sphere spectrum
that is the monoidal unit.

We would like to compute the homotopy groups
of cotensor products of comodules. 
Since a cotensor product of comodules is a limit of  
a cosimplicial object,
we have the Bousfield-Kan spectral sequence
abutting to the homotopy groups of
the cotensor product.

First,
we recall the Bousfield-Kan spectral sequence
associated to a cosimplicial object in $\sp$.
Let $X^{\bullet}:N(\Delta)\to\sp$
be a cosimplicial object in $\sp$.
Since the quasi-category $\sp$ of spectra
is the underlying quasi-category
of the combinatorial simplicial model category
$\Sigma\sp$ of symmetric spectra,
we can take a cosimplicial object
$Y^{\bullet}: \Delta\to \Sigma\sp^{\circ}$
such that $N(Y^{\bullet})\simeq X^{\bullet}$ 
by \cite[Prop.~4.2.4.4.]{Lurie1},
where $\Sigma\sp^{\circ}$
is the simplicial full subcategory 
of $\Sigma\sp$ consisting of
objects that are both fibrant and cofibrant.
Then the limit ${\rm lim}_{N(\Delta)}\,X^{\bullet}$
in $\sp$ is represented by the homotopy limit
${\rm holim}_{\Delta}\, Y^{\bullet}$.

We recall that 
${\rm Tot}_r(X^{\bullet})$ 
is defined to be the limit
of $X^{\bullet}|_{N(\Delta^{\le r})}$
in $\sp$ for $r\ge 0$,
where  
$\Delta^{\le r}$
is the full subcategory of $\Delta$
spanned by $\{[0],[1],\ldots,[r]\}$.
The inclusion $\Delta^{\le r}\hookrightarrow
\Delta^{\le r+1}$ induces
a map ${\rm Tot}_{r+1}(X^{\bullet})\to
{\rm Tot}_r(X^{\bullet})$ for $r\ge 0$.
We have a tower
$\{{\rm Tot}_rX^{\bullet}\}_{r\ge 0}$
and the limit of the tower is 
equivalent to ${\rm Tot}(X^{\bullet})$:
\[ {\rm Tot}(X^{\bullet})\simeq
   \lim_r{\rm Tot}_r(X^{\bullet}).\]

Let $F_r(X^{\bullet})$ be the fiber of the map
${\rm Tot}_r(X^{\bullet})\to {\rm Tot}_{r-1}(X^{\bullet})$
for $r\ge 0$,
where ${\rm Tot}_{-1}(X^{\bullet})=0$.
Associated to the tower
$\{{\rm Tot}_r(Y^{\bullet})\}_{r\ge 0}$,
by applying the homotopy groups,
we obtain the Bousfield-Kan spectral sequence
\[ E_1^{s,t}\cong \pi_{t-s}F_s(X^{\bullet})
   \Longrightarrow \pi_{t-s}{\rm Tot}(X^{\bullet})\]
(see \cite[Ch.~IX, \S4]{Bousfield-Kan}).
We can identify the $E_2$-page
of the spectral sequence with 
the cohomotopy groups
of the cosimplicial graded abelian group
$\pi_*(X^{\bullet})$:
\[ E_2^{s,t}\cong \pi^s\pi_t(X^{\bullet}) \]
(see \cite[Ch.~X, \S7]{Bousfield-Kan}).

Next,
we construct a spectral sequence
that computes the homotopy groups
of cotensor products of comodules.
Let $A$ be an algebra object of $\sp$
and $\Gamma$ a coalgebra object 
of ${}_A\bmod_A(\sp)$.
Recall that the cotensor product 
$M\square_{\Gamma}N$
is defined to be the limit of
the cosimplicial object 
$C^{\bullet}(M,\Gamma,N)$ 
for a right $\Gamma$-comodule $M$
and a left $\Gamma$-comodule $N$.
Hence we obtain the Bousfield-Kan spectral sequence
abutting to the homotopy groups
of the cotensor product $M\square_{\Gamma}N$:
\[ E_2^{s,t}\Longrightarrow
   \pi_{t-s}(M\square_{\Gamma}N),\]
where the $E_2$-page is given by
\[ E_2^{s,t}\cong \pi^s\pi_t\,C^{\bullet}(M,\Gamma,N).\]

For a spectrum $X\in {\rm Sp}$,
we write $X_*$ for the homotopy groups $\pi_*X$
for simplicity. 
Now we suppose that
$\Gamma_*$ is flat as a left $A_*$-module
and a right $A_*$-module.
Since $C^r(M,\Gamma,N)\simeq
M\otimes_A\overbrace{\Gamma\otimes_A\cdots\otimes_A\Gamma}^r\otimes_AN$
for all $r\ge 0$,
we see that
\[ \pi_*C^{\bullet}(M,\Gamma,N)\cong
    C^{\bullet}(M_*,\Gamma_*,N_*)\]
if $M_*$ is a flat right $A_*$-module
or $N_*$ is a flat left $A_*$-module.

\begin{proposition}
If $M_*$ is a flat right $A_*$-module
or $N_*$ is a flat left $A_*$-module,
then we have the Bousfield-Kan spectral sequence
abutting to the homotopy groups
of the cotensor product
$M\square_{\Gamma}N$\mbox{\rm :}
\[ E_2^{s,t}\Longrightarrow
   \pi_{t-s}(M\square_{\Gamma}N).\]
The $E_2$-page of the spectral sequence is given by
\[ E_2^{s,t}\cong\,{\rm Cotor}^s_{\Gamma_*}(M_*,N_*)_t.\]
\end{proposition}

Now we regard $A$ as a right $A$-module 
and suppose $A$ is a right $\Gamma$-comodule
via $\eta_R$.
In this case $A_*$ is a right $\Gamma_*$-comodule
via $\eta_{R*}$.
Then we can regard $A_*$ as a left $\Gamma_*$-comodule
by using the isomorphism
$A_*\cong {\rm Hom}_{A_*}(A_*,A_*)$
of left $A_*$-modules,
where ${\rm Hom}_{A_*}(A_*,A_*)$
is the graded abelian group of
graded homomorphisms of right $A_*$-modules.
Hence we can form
${\rm Ext}_{\Gamma_*}^*(A_*,N_*)$
for any left $\Gamma$-comodule $N$.

We consider the Bousfield-Kan spectral sequence
associated to $C^{\bullet}(A,\Gamma,N)$.
Note that the limit of $C^{\bullet}(A,\Gamma,N)$
is $P(N)=A\square_{\Gamma}N$.

\begin{corollary}
We assume that the right $A$-module $A$
is a right $\Gamma$-comodule via $\eta_R$.
Then we have the Bousfield-Kan spectral sequence
abutting to the homotopy groups
of $P(N)$\mbox{\rm :}
\[ E_2^{s,t}\Longrightarrow \pi_{t-s}P(N),\]
where the $E_2$-page is given by
\[ E_2^{s,t}\cong{\rm Ext}_{\Gamma_*}^s(A_*,N_*)_t.\]
\end{corollary}

In the following of this subsection
we study the relationship between
Bousfield-Kan spectral sequences
and Adams spectral sequences.

For an $\mathbb{S}$-algebra $A$,
we have the coalgebra $A\otimes A$ in 
${}_A{\rm BMod}_A(\sp)$.
We write $\Gamma(A)=(A,A\otimes A)$ for simplicity
and we call $A\otimes A$-comodules 
$\Gamma(A)$-comodules interchangeably.
We can regard $A$ as a left $\Gamma(A)$-comodule
via $\eta_L :A\simeq A\otimes \mathbb{S}
\stackrel{{\rm id}_A\otimes u}{\longrightarrow} A\otimes A$ and
as a right $\Gamma(A)$-comodule via
$\eta_R: A\simeq \mathbb{S}\otimes A
\stackrel{u\otimes {\rm id}_A}{\longrightarrow} 
A\otimes A$,
where $u: \mathbb{S}\to A$ is the unit map.

By Theorem~\ref{thm:base-change-comodules},
we have a left $\Gamma(A)$-comodule $A\otimes X$
for any $X\in\sp$.
We consider the cobar construction
\[ C^{\bullet}(A,A\otimes A,A\otimes X), \]
where $X\in\sp$.
Note that we have a coaugmentation
$X\to C^{\bullet}(A,A\otimes A,A\otimes X)$,
which is given by
$X\simeq \mathbb{S}\otimes X
\stackrel{u\otimes {\rm id}_X}{\longrightarrow} 
A\otimes X\simeq
C^0(A,A\otimes A,A\otimes X)$.
This induces a map
\[ X\to P(A\otimes X)
   =\, {\rm lim}_{N(\Delta)}\,
   C^{\bullet}(A,A\otimes A,A\otimes X).\]

We have an equivalence
$C^{\bullet}(A,A\otimes A,A\otimes X)\simeq
C^{\bullet}(A,A\otimes A,A)\otimes X$.
We see that the cobar construction
$C^{\bullet}(A,A\otimes A,A)$ is the Amitsur complex
in $\sp$ given by
\[ C^r(A,A\otimes A,A)\simeq
   \overbrace{A\otimes\cdots\otimes A}^{r+1}\]
for any $r\ge 0$
with the usual structure maps.
The Bousfield-Kan spectral sequence of
the cobar construction
$C^{\bullet}(A,A\otimes A,A\otimes X)$
is related to
the $A$-Adams spectral sequence of $X$.
Although this may be well-known to experts,
we briefly review this relation  
for the reader's convenience
(see, for example, \cite[\S2.1]{MNJ}). 

The coaugmented cosimplicial object
$X\to C^{\bullet}(A,A\otimes A,A\otimes X)$ induces 
a tower 
\[ \{{\rm Tot}_rC^{\bullet}(A,A\otimes A,A\otimes X)\}_{r\ge 0},\]
and a map of towers
$c(X)\to \{{\rm Tot}_rC^{\bullet}(A,A\otimes A,A\otimes X)\}_{r\ge 0}$
for any $X\in \sp$.
This tower is related to the $A$-Adams tower of $X$.

Let $\overline{A}$ be the fiber of
the unit map $u:\mathbb{S}\to A$.
We have a canonical map
$\overline{A}\to \mathbb{S}$.
For $r\ge 0$,
we set 
\[ T_r(A,X)= \overbrace{\overline{A}\otimes\cdots
             \otimes\overline{A}}^r\otimes X,\]
where we understand $\overline{A}^{\otimes 0}=\mathbb{S}$.
Using the canonical map $\overline{A}\to \mathbb{S}$,
we define a map $T_{r+1}(A,X)\to T_r(A,X)$ for $r\ge 0$
by
\[ T_{r+1}(A,X)\simeq \overline{A}\otimes\overline{A}^{\otimes r}
\otimes X\to
S\otimes \overline{A}^{\otimes r}\otimes X 
\simeq T_r(A,X). \]
With these maps,
we obtain a tower
$\{T_r(A,X)\}_{r\ge 0}$
and a map $\{T_r(A,X)\}_{r\ge 0}\to c(X)$
of towers.

Let $G_r(A,X)$ be the cofiber of
the map $T_{r+1}(A,X)\to T_r(A,X)$
for $r\ge 0$.
Associate to the tower
$\{T_r(A,X)\}_{r\ge 0}$,
by applying the homotopy groups,
we obtain the $A$-Adams spectral sequence of $X$.
The $E_1$-page of the spectral sequence
is given by
\[ E_1^{s,t}\cong\pi_{t-s}G_s(A,X)\]
(see, for example, \cite[Ch.~2.2]{Ravenel1}).

We set $C^{\bullet}=C^{\bullet}(A,A\otimes A,A\otimes X)$. 
By \cite[Prop.~2.14]{MNJ},
the cofiber of the map $T_{r+1}(A,X)\to X$ is equivalent
to ${\rm Tot}_r(C^{\bullet})$ for all $r\ge 0$.
Hence we obtain a natural cofiber sequence of towers
\[ \{T_{r+1}(A,X)\}_{r\ge 0}\to
   c(X)\to
   \{{\rm Tot}_r(C^{\bullet})\}_{r\ge 0}.\]
In particular,
we see that $G_r(A,X)$ is equivalent to 
the fiber $F_r(C^{\bullet})$
of the map ${\rm Tot}_r(C^{\bullet})\to
{\rm Tot}_{r-1}(C^{\bullet})$.
Comparing the spectral sequences,
we see that 
the $A$-Adams spectral sequence of $X$
coincides with the Bousfield-Kan spectral sequence
associated to the cobar construction
$C^{\bullet}(A,A\otimes A,A\otimes X)$. 

We recall that the map
$X\to P(A\otimes X)$ is an $A$-nilpotent
completion in ${\rm Ho}(\sp)$ in the sense of 
Bousfield~\cite{Bousfield},
where ${\rm Ho}(\sp)$ is the stable homotopy category
of spectra.

Let $R$ be a ring spectrum in ${\rm Ho}(\sp)$.
A spectrum $W$ is said to be $R$-nilpotent
if $W$ lies in the thick ideal of ${\rm Ho}(\sp)$
generated by $R$.
An $R$-nilpotent resolution of a spectrum $Z$ is a 
tower $\{W_r\}_{r\ge 0}$ equipped with a map of towers
$c(Z)\to \{W_r\}_{r\ge 0}$ in ${\rm Ho}(\sp)$
such that $W_r$ is $R$-nilpotent for all $r\ge 0$
and the map 
\[ {\rm colim}_r\, {\rm Hom}_{{\rm Ho}(\sp)}(W_r,N)
   \to {\rm Hom}_{{\rm Ho}(\sp)}(Z,N) \]
is an isomorphism for any $R$-nilpotent spectrum $N$.
An $R$-nilpotent completion of $Z$ is defined to be
the map $Z\to {\rm holim}_r W_r$ for an $R$-nilpotent
resolution $\{W_r\}_{r\ge 0}$ of $Z$. 

We shall show that the tower
$\{{\rm Tot}_r(C^{\bullet})\}_{r\ge 0}$
is an $A$-nilpotent resolution of $X$,
where $C^{\bullet}=C^{\bullet}(A,A\otimes A,A\otimes X)$.
For any $r\ge 0$,
the fiber $F_r(C^{\bullet})$ of the map
${\rm Tot}_r(C^{\bullet})\to {\rm Tot}_{r-1}(C^{\bullet})$
is equivalent to $G_r(A,X)$.
Since 
$G_r(A,X)\simeq A\otimes \overline{A}^{\otimes r}\otimes X$
is a left $A$-module,
$G_r(A,X)$ is $A$-nilpotent for all $r\ge 0$.
By induction on $r$ and the fact that
${\rm Tot}_0(C^{\bullet})=A\otimes X$,
we see that ${\rm Tot}_r(C^{\bullet})$
is $A$-nilpotent for all $r\ge 0$.

Recall that the fiber of the map $X\to {\rm Tot}_r(C^{\bullet})$
is equivalent to $T_{r+1}(A,X)$ for all $r\ge 0$.
The map $\overline{A}\to \mathbb{S}$ 
is null in ${\rm Ho}(\sp)$
after tensoring with $A$
since the unit map
$\mathbb{S}\to A$ has a left inverse
after tensoring with $A$.
Hence we see that the map $T_{r+1}(A,X)\to T_r(A,X)$ is 
null in ${\rm Ho}(\sp)$ after tensoring with $A$.
This implies that the map
${\rm colim}_r\, {\rm Hom}_{\rm Ho(\sp)}
 ({\rm Tot}_r(C^{\bullet}),A\otimes Y)
\to {\rm Hom}_{{\rm Ho}(\sp)}(X,A\otimes Y)$
is an isomorphism for any spectrum $Y$.
Since the class of $A$-nilpotent spectra coincides
with the thick subcategory generated 
by the class $\{A\otimes Z|\ Z\in\sp\}$,
we see that the map
${\rm colim}_r\, {\rm Hom}_{\rm Ho(\sp)}
 ({\rm Tot}_r(C^{\bullet}),N)
\to {\rm Hom}_{{\rm Ho}(\sp)}(X,N)$
is an isomorphism for any $A$-nilpotent spectrum $N$.

Therefore,
the tower $\{{\rm Tot}_r(C^{\bullet})\}_{r\ge 0}$ 
is an $A$-nilpotent resolution of $X$. 
Since $P(A\otimes X)\simeq {\rm lim}_r\,{\rm Tot}_r(C^{\bullet})$,
we see that the map
$X\to P(A\otimes X)$
is an $A$-nilpotent completion in ${\rm Ho}(\sp)$.

\subsection{Complex oriented spectra}

In this subsection we study
quasi-categories of comodules over coalgebras
associated to Landweber exact $\mathbb{S}$-algebras.
We show that the quasi-category of comodules
over the coalgebra associated to a Landweber exact 
$\mathbb{S}$-algebra depends only on the height
of the underlying $MU_*$-algebra.

Let $MU$ be the complex cobordism spectrum.
The coefficient ring of $MU$
is a polynomial ring over the ring $\mathbb{Z}$ of integers
with infinite many variables
\[ \pi_*MU=\mathbb{Z}[x_1,x_2,\ldots] \]
with degree $|x_i|=2i$ for $i\ge 1$.
We assume that the Chern numbers of $x_{p^n-1}$ are all divisible by
$p$ for all positive integers $n$ and all prime numbers $p$.
In this case
the ideals $I_{p,n}=(p,x_{p-1},\ldots,x_{p^{n-1}-1})$ 
are invariant and independent of the choice of generators.
We set $I_{p,0}=(0)$ and $I_{p,\infty}=\subrel{n\ge 0}{\cup}\,I_{p,n}$.
The ideals $I_{p,n}$ for $0\le n\le \infty$ and all primes $p$
are the only invariant prime ideals in $MU_*$
(see \cite{Landweber}).

For a graded commutative $MU_*$-algebra $R_*$,
we say that $R_*$ is Landweber exact if
$p,x_{p-1},\ldots, x_{p^n-1},\ldots$ is a regular sequence
in $R_*$ for all prime numbers $p$.  

If $E^*(-)$ is a complex oriented cohomology theory
represented by a spectrum $E$,
then there is a ring spectrum map $f: MU\to E$
in the stable homotopy category ${\rm Ho}(\sp)$ 
of spectra.
We say $E$ is Landweber exact if
$E_*$ is a graded commutative ring
and Landweber exact via the graded ring homomorphism
$f_*: MU_*\to E_*$.

We consider an $\mathbb{S}$-algebra that is Landweber exact.

\begin{definition}\rm
We say that $A$ is Landweber exact $\mathbb{S}$-algebra if
$A$ is an $\mathbb{S}$-algebra spectrum
equipped with a map $f: MU\to A$ of ring spectra
in ${\rm Ho}(\sp)$ 
such that 
$A_*$ is a graded commutative ring and
Landweber exact via the graded ring homomorphism
$f_*: MU_*\to A_*$.
\end{definition}

Let $p$ be a prime number
and let $\mathbb{S}_{(p)}$ be the localization of 
the sphere spectrum $\mathbb{S}$ at $p$.
We can consider a Landweber exact $\mathbb{S}_{(p)}$-algebra 
in the same way.
If $A$ is a Landweber exact $\mathbb{S}$-algebra,
then the localization $A_{(p)}$ 
is a Landweber exact $\mathbb{S}_{(p)}$-algebra
at any prime number $p$.

\begin{example}\rm
For any prime number $p$ and 
any positive integer $n$,
the Johnson-Wilson spectrum $E(n)$ at $p$
is a complex oriented Landweber exact spectrum.
By \cite[Proposition~4.1]{Baker-Jeanneret},
$E(n)$  
admits an $MU_{(p)}$-algebra spectrum structure.
Hence, in particular,
$E(n)$ is a Landweber exact $\mathbb{S}_{(p)}$-algebra.
\end{example}


\begin{definition}\rm
For a Landweber exact graded commutative ring
$A_*$, we denote by ${\rm ht}_p\,A_*$
the height of $A_*$ at a prime $p$
in the sense of 
\cite[Definition~7.2]{Hovey-Strickland},
that is,
the largest number $n$ such that
$A_*/I_{p,n}$ is nonzero, or $\infty$
if $A_*/I_{p,n}$ is nonzero for all $n$.
For a Landweber exact $\mathbb{S}$-algebra $A$,
we denote by ${\rm ht}_p\,A$
the height of $A_*$ at $p$.
\end{definition}

For Landweber exact 
$\mathbb{S}$-algebras $E$ and $F$,
we have an isomorphism
\[ F_*(E)\cong F_*\otimes_{MU_*}MU_*(MU)\otimes_{MU_*}E_*.\]
By abuse of notation,
for graded commutative Landweber exact $MU_{*}$-algebras
$A_*$ and $B_*$,
we set $B_*(A)=B_*\otimes_{MU_*}MU_*(MU)\otimes_{MU_*}A_*$.
We denote by $\Gamma(A_*)$
the pair $(A_*,A_*(A))$,
which forms a graded Hopf algebroid 
(see \cite[Appendix~A.1]{Ravenel1}).
We can consider
the categories of graded $\Gamma(A_*)$-comodules
$\lcomod_{\Gamma(A_*)}(\ab)$
which is an abelian category 
since $\Gamma(A_*)$ is a flat Hopf algebroid.

The canonical map $MU_*(MU)\to A_*(A)$ induces 
a map of graded Hopf algebroids
$\Phi(A): \Gamma(MU_*)\to\Gamma(A_*)$.
We consider the functor
\[ \Phi(A)_*:
   \lcomod_{\Gamma(MU_*)}(\ab)\longrightarrow
   \lcomod_{\Gamma(A_*)}(\ab)\]
given by $\Phi(A)_*(M_*)=A_*\otimes_{MU_*}M_*$
for $M_*\in\lcomod_{\Gamma(MU_*)}(\ab)$.
The functor $\Phi(A)_*$
has the right adjoint 
$\Phi(A)^*: \lcomod_{\Gamma(A_*)}(\ab)\to
   \lcomod_{\Gamma(MU_*)}(\ab)$ 
given by
\[ \Phi(A)^*(N_*)=MU_*(A)\square_{A_*(A)}N_* \]
for $N_*\in\lcomod_{\Gamma(A_*)}(\ab)$
(see \cite[Lem.~2.4 and Remark after its proof]{Hovey-Strickland}).
Let $\mathcal{T}_{A_*}$ be
the class  of all graded $\Gamma(MU_*)$-comodules
$M_*$ such that 
$A_*\otimes_{MU_*}M_*$ is trivial. 
By \cite[Thm.~2.5]{Hovey-Strickland},
the adjoint pair $(\Phi(A)_*,\Phi(A)^*)$ 
induces an adjoint equivalence
of categories between
$\lcomod_{\Gamma(A_*)}(\ab)$ and
the localization of 
$\lcomod_{\Gamma(MU_*)}$
with respect to $\mathcal{T}_{A_*}$.

Let $A_*$ and $B_*$ be graded commutative
Landweber exact $MU_*$-algebras.
We recall that $B_*(A)=B_*\otimes_{MU_*}MU_*(MU)\otimes_{MU_*}A_*$.
Note that $B_*(A)$ is a graded left $\Gamma(B_*)$-comodule
and a graded right $\Gamma(A_*)$-comodule.
We can define a functor
\[ G_{B_*,A_*}: \lcomod_{\Gamma(A_*)}(\ab)
           \longrightarrow
           \lcomod_{\Gamma(B_*)}(\ab)\]
which assigns to a graded left $\Gamma(A_*)$-comodule 
$M_*$ the graded left $\Gamma(B_*)$-comodule 
$G_{B_*,A_*}(M_*)$ given by
\[ G_{B_*,A_*}(M_*)=B_*(A)\square_{A_*(A)}M_*.\]

\begin{lemma}
\label{lem:GBA-equivalence}
If\, ${\rm ht}_pA_*=\,{\rm ht}_pB_*$ for all primes $p$,
then the functor $G_{B_*,A_*}$ gives an equivalence
of categories between
$\lcomod_{\Gamma(A_*)}(\ab)$
and $\lcomod_{\Gamma(B_*)}(\ab)$.
\end{lemma}

\begin{proof}
Note that the functor $G_{B_*,A_*}$
is the composition  
$\Phi(B)_*\Phi(A)^*$.
The functor $\Phi(A)^*$
induces an equivalence 
of categories from 
$\lcomod_{\Gamma(A_*)}(\ab)$
to the localization of 
$\lcomod_{\Gamma(MU_*)}(\ab)$
with respect to $\mathcal{T}_{A_*}$
and $\Phi(B)_*$
induces an equivalence of categories
from the localization of 
$\lcomod_{\Gamma(MU_*)}(\ab)$
with respect to $\mathcal{T}_{B_*}$
to $\lcomod_{\Gamma(B_*)}(\ab)$.
By \cite[Thm.~7.3]{Hovey-Strickland},
the assumption that
$A_*$ and $B_*$ have the same heights for all $p$
implies that
$\mathcal{T}_{A_*}=\mathcal{T}_{B_*}$.
Hence we see that $G_{B_*,A_*}$
gives an equivalence of categories. 
\qed\end{proof}

\begin{lemma}
\label{lem:landweber-commutes-cotensor}
Let $A, B, C$ be Landweber exact $\mathbb{S}$-algebras.
We assume that ${\rm ht}_pA={\rm ht}_pB={\rm ht}_pC$ 
for all primes $p$. 
Then, for any $\Gamma(A)$-comodule $M$,
the canonical map
\[ C\otimes ((B\otimes A)\square_{A\otimes A}M)
   \longrightarrow
   (C\otimes B\otimes A)\subrel{A\otimes A}{\square}M\]
is an equivalence.
\end{lemma}

\begin{proof}
Let $R$ be a Landweber exact 
$\mathbb{S}$-algebra
which has the same height at all $p$ as $A$.
First, we consider the homotopy groups
of the cotensor product
$(R\otimes A)\square_{A\otimes A}M$.
The Bousfield-Kan spectral sequence
abutting to the homotopy groups
of $(R\otimes A)\square_{A\otimes A}M$
has the $E_2$-page given by
\[ E_2^{s,t}\cong{\rm Cotor}^s_{A_*(A)}(R_*(A),M_*)_t.\]

We have 
\[ H^s(C(A_*(A),A_*(A),M_*))=0\]
for $s>0$.
Note that the cobar complex
$C^*(A_*(A),A_*(A),M_*)$ is a cochain complex
in the abelian category $\lcomod_{\Gamma(A_*)}(\ab)$.
Applying the functor $G_{R_*,A_*}$
to $C^*(C(A_*(A),A_*(A),M_*))$,
we obtain
\[ H^s(C(R_*(A),A_*(A),M_*))=0\]
for $s>0$ by Lemma~\ref{lem:GBA-equivalence}.
Hence the Bousfield-Kan spectral sequence
abutting
to the homotopy groups
of $(R\otimes A)\square_{A\otimes A}M$
collapses from the $E_2$-page
and we obtain an isomorphism
\[ \pi_*((R\otimes A)\square_{A\otimes A}M)\cong
   R_*(A)\square_{A_*(A)}M_*.\]

In particular,
since $C\otimes B$ is Landweber exact
through the map
$MU\to B\to C\otimes B$ in ${\rm Ho}(\sp)$,
we obtain an isomorphism
\[ \pi_*((C\otimes B\otimes A)\square_{A\otimes A}M)\cong
   (C\otimes B)_*(A)\square_{A_*(A)}M_*.\]  
Since we have an isomorphism
\[ (C\otimes B)_*(A)\cong
   C_*(B)\otimes_{B_*}B_*(A),\]
we obtain an isomorphism
\[ \pi_*((C\otimes B\otimes A)\square_{A\otimes A}M))\cong
   C_*(B)\otimes_{B_*}(B_*(A)\square_{A_*(A)}M_*).\]

On the other hand,
we have isomorphisms
\[ \pi_*(C\otimes ((B\otimes A)\square_{A\otimes A}M))
   \cong
   C_*(B)\otimes_{B_*}\pi_*((B\otimes A)\square_{A\otimes A}M).\] 
and
\[ \pi_*((B\otimes A)\square_{A\otimes A}M)
   \cong
   B_*(A)\square_{A_*(A)}M_*.\]
Hence we see that the canonical map
$C\otimes (B\otimes A)\square_{A\otimes A}M\to
(C\otimes B\otimes A)\square_{A\otimes A}M$
induces an isomorphism of homotopy groups. 
This completes the proof.
\qed\end{proof}

\begin{corollary}
\label{cor:comodule-extension-Landweber-exact}
Let $A$ and $B$ be Landweber exact
$\mathbb{S}$-algebras.
We assume that ${\rm ht}_pA={\rm ht}_pB$ for all primes $p$.
For any left $\Gamma(A)$-comodule $M$,
the left $\Gamma(B)$-comodule structure
on $B$ induces a left $\Gamma(B)$-comodule structure on
$(B\otimes A)\square_{A\otimes A}M$.
\end{corollary}

\begin{proof}
For $r>0$,
$B^{\otimes r}$ is a Landweber exact $\mathbb{S}$-algebra
and has the same height at all $p$ as $B$.
By Lemma~\ref{lem:landweber-commutes-cotensor},
the canonical map
$B^{\otimes r}\otimes ((B\otimes A)\square_{A\otimes A}M)
\to
(B^{\otimes r}\otimes B\otimes A)\square_{A\otimes A}M$
is an equivalence for all $r>0$.
Applying Lemma~\ref{lem:general-convergence-colimit-criterion}
for the simplicial object
$B_{\bullet}(B\otimes A,A\otimes A,M)$
in $\lmod_{B\otimes B}(\lmod_B(\sp)^{\rm op})$,
we see that the cosimplicial object
$C^{\bullet}(B\otimes A,A\otimes A,M)$
has a limit in $\lcomod_{\Gamma(B)}(\sp)$
and the forgetful functor
$\lcomod_{\Gamma(B)}(\sp)\to \lmod_B(\sp)$
preserves the limit.
\qed\end{proof}

Using Corollary~\ref{cor:comodule-extension-Landweber-exact},
we can define a functor 
\[ F_{B,A}: {\rm Comod}_{\Gamma(A)}({\rm Sp})\longrightarrow
   {\rm Comod}_{\Gamma(B)}({\rm Sp})\]
by assigning to 
$M\in \lcomod_{\Gamma(A)}({\rm Sp})$
the $\Gamma(B)$-comodule $F_{B,A}(M)$ given by 
\[ F_{B,A}(M)=(B\otimes A)\subrel{A\otimes A}{\square}M.\]

\begin{theorem}
\label{thm:equiv-comodule-at-p}
Let $A$ and $B$ be Landweber exact $\mathbb{S}$-algebras.
We assume that $A$ and $B$ have the same height at all $p$.
Then the functor $F_{B,A}$ gives an equivalence of quasi-categories
\[ \lcomod_{\Gamma(A)}({\rm Sp})\simeq
     \lcomod_{\Gamma(B)}({\rm Sp}).\]
\end{theorem}

\begin{proof}
For any left $\Gamma(A)$-comodule $M$,
we have 
\[ \pi_*F_{B,A}(M)\cong B_*(A)\square_{A_*(A)}M_*.\]
Since the functor
$G_{B_*,A_*}= B_*(A)\square_{A_*(A)}(-)$
gives an equivalence of categories
by Lemma~\ref{lem:GBA-equivalence},
we see that
the functor $F_{B,A}$ gives an equivalence 
of quasi-categories between
$\lcomod_{\Gamma(A)}(\sp)$
and $\lcomod_{\Gamma(B)}(\sp)$.
\qed\end{proof}

\begin{proposition}
\label{prop:compatibility-extension-AandB}
Let $A$ and $B$ be Landweber exact $\mathbb{S}$-algebras.
We assume that ${\rm ht}_pA={\rm ht}_pB$ for all primes $p$.
Then the following diagram is commutative
\[ \xymatrix{
    & \sp \ar[dl]_{A\otimes (-)} \ar[dr]^{B\otimes(-)} & \\
    \lcomod_{\Gamma(A)} \ar[rr]^{F_{B,A}}
    && \lcomod_{\Gamma(B)}. \\
    }\]
\end{proposition}

\begin{proof}
Since $B\otimes \mathbb{S}\simeq 
B\otimes_{\mathbb{S}}\mathbb{S}\otimes \mathbb{S}$
is an extended right $\Gamma(\mathbb{S})$-comodule,
we have a natural equivalence
\[ B\otimes X\simeq 
   (B\otimes \mathbb{S})
   \square_{\mathbb{S}\otimes \mathbb{S}}(\mathbb{S}\otimes X) \]
for any $X\in\sp$.
The unit map
$\mathbb{S}\to A$ induces
a natural map
\[ (B\otimes \mathbb{S})
    \square_{\mathbb{S}\otimes \mathbb{S}}(\mathbb{S}\otimes X)
   \longrightarrow
   (B\otimes A)\square_{A\otimes A}(A\otimes X)\]
and hence we obtain a natural map
\[ f: B\otimes X
   \longrightarrow
   (B\otimes A)\square_{A\otimes A}(A\otimes X).\]
Note that $f$ is a map
of left $\Gamma(B)$-comodule.
Since 
\[ \pi_*((B\otimes A)\square_{A\otimes A}(A\otimes X))\cong
   B_*(A)\square_{A_*(A)}A_*(X),\]
we see that $f$ induces an isomorphism
of homotopy groups. 
This completes the proof.
\qed\end{proof}

\subsection{The $E(n)$-local category}

In this subsection 
we study the quasi-category of $E(n)$-local spectra,
where $E(n)$ is the $n$th Johnson-Wilson spectrum
at a prime $p$.
We show that the quasi-category of $E(n)$-local spectra
is equivalent to the quasi-category of comodules
over the coalgebra $A\otimes A$
for any Landweber exact $\mathbb{S}_{(p)}$-algebra
of height $n$ at $p$.

In this subsection we fix a non-negative integer $n$ 
and a prime number $p$.
Let $L_n\sp$ be the Bousfield localization
of the quasi-category of spectra
with respect to the $n$th Johnson-Wilson spectrum $E(n)$
at $p$.
We denote by $L_n:\sp\to\sp$ the associated localization functor.
The quasi-category $L_n\sp$
is a stable homotopy theory
with the tensor product in $\sp$
and the unit $L_n\mathbb{S}$,
where $L_n\mathbb{S}$ is 
the $E(n)$-localization of the sphere spectrum $\mathbb{S}$.

We recall that $E(n)$ is a Landweber exact 
$\mathbb{S}_{(p)}$-algebra
with height ${\rm ht}_p\,E(n)=n$.
For simplicity,
we set
$\Gamma(n)=(E(n),E(n)\otimes E(n))$.
The functor
\[ E(n)\otimes (-):\sp\to 
   \lcomod_{\Gamma(n)}(\sp) \]
factors through $L_n\sp$ and
we obtain a functor
\[ E(n)\otimes (-):
   L_n\sp\longrightarrow
   \lcomod_{\Gamma(n)}(\sp).\]

Since any $N\in \lmod_{E(n)}(\sp)$
is $E(n)$-local,
we see that 
$P(M)$ lies in $L_n\sp$
for any $M\in\lcomod_{\Gamma(n)}(\sp)$,
and we obtain an adjunction of functors
\[ E(n)\otimes (-): L_n\sp
   \rightleftarrows
   \lcomod_{\Gamma(n)}(\sp): P.\]      

\begin{proposition}
\label{prop:E(n)-local-equivalence-comodule}
The pair of functors 
\[ E(n)\otimes (-): L_n\sp\rightleftarrows
   \lcomod_{\Gamma(n)}(\sp):
   P \]
is an adjoint equivalence.
\end{proposition}

\begin{proof}
By \cite[Theorem~5.3]{Hovey-Sadofsky},
the unit $L_n\mathbb{S}$ is $E(n)$-nilpotent
(see, also, \cite[Chapter~8]{Ravenel2}).
By Theorem~\ref{theorem:general-equivalence-comodules},
we obtain the proposition.
\qed\end{proof}

Let $A$ be a Landweber exact $\mathbb{S}_{(p)}$-algebra
with ${\rm ht}_p\,A=n$.
Since $A$ is Bousfield equivalent 
to $E(n)$ by \cite[Corollary~1.12]{Hovey},
we have a canonical equivalence
of stable homotopy theories
\[ L_A\sp\simeq L_n\sp, \]
where $L_A\sp$
is the Bousfield localization of $\sp$
with respect to $A$.
In the same way as $E(n)$,
we have an adjunction of functors
\[ A\otimes (-):
   L_A\sp\rightleftarrows
   \lcomod_{\Gamma(A)}(\sp): P.\]

Recall that we have the functor
\[ F_{A,E(n)}:\lcomod_{\Gamma(n)}(\sp)\to\lcomod_{\Gamma(A)}(\sp) \]
given by
\[ F_{A,E(n)}(M)=(A\otimes E(n)){\square}_{E(n)\otimes E(n)}M\]
for $M\in \lcomod_{\Gamma(n)}(\sp)$.
By Proposition~\ref{prop:compatibility-extension-AandB},
we see that there is a commutative diagram
of quasi-categories
\[ \begin{array}{ccc}
    L_n\sp &\stackrel{E(n)\otimes(-)}
    {\hbox to 15mm{\rightarrowfill}} &
    \lcomod_{\Gamma(n)}(\sp)\\[1mm]
    \bigg\downarrow & & 
    \phantom{\mbox{$\scriptstyle F_{A,E(n)}$}}
    \bigg\downarrow
    \mbox{$\scriptstyle F_{A,E(n)}$}\\[1mm]
    L_A\sp & \stackrel{A\otimes (-)}
    {\hbox to 15mm{\rightarrowfill}} &
    \lcomod_{\Gamma(A)}(\sp),\\
   \end{array}\] 
where the left vertical arrow is the canonical equivalence.

\begin{theorem}
\label{thm:Ln-equivalent-comodules}
If $A$ is a Landweber exact $\mathbb{S}_{(p)}$-algebra
of height $n$ at $p$,
then the pair of functors
\[ A\otimes (-): L_A{\rm Sp}\rightleftarrows
{\rm Comod}_{A\otimes A}({\rm Sp}):P \]
is an adjoint equivalence.
\end{theorem}

\begin{proof}
The theorem follows
from Theorem~\ref{thm:equiv-comodule-at-p}
and Proposition~\ref{prop:E(n)-local-equivalence-comodule}.
\qed\end{proof}

\subsection{Connective cases}

In this subsection 
we consider the quasi-category of comodules
over $A\otimes A$ for a connective $\mathbb{S}$-algebra $A$.
We show that the quasi-category of connective $A$-local spectra
is equivalent to the quasi-category
of connective comodules over $A\otimes A$ under
some conditions.

We say that a spectrum $X$ is connective
if $\pi_iX=0$ for all $i<0$.
We denote by $\sp^{\ge 0}$
the full subcategory of $\sp$
consisting of connective objects.
In this subsection 
we let $A$ be a connective $\mathbb{S}$-algebra
and assume that $A_*(A)$
is flat as a left and right $A_*$-module.

We consider the condition that
the multiplication map induces an isomorphism
$\pi_0(A\otimes A)\stackrel{\cong}{\to}\pi_0A$.
Note that there is an isomorphism $\pi_0(A\otimes A)\cong
\pi_0A\otimes\pi_0A$.

Let $R$ be a (possibly non-commutative) ring with identity $1$.
The core $cR$ of $R$ is defined to be the subring
\[ cR=\{r\in R|\ 
   \mbox{\rm $r\otimes 1=1\otimes r$ in $R\otimes R$} 
   \}\]
(see \cite[6.4]{Bousfield}).
The core $cR$ is a commutative ring
and the multiplication map 
$cR\otimes cR\to cR$ is an isomorphism.
We see that, if the multiplication map
induces an isomorphism $\pi_0(A\otimes A)\stackrel{\cong}{\to}
\pi_0A$, then $c\pi_0A=\pi_0A$ and, in particular,
$\pi_0A$ is a commutative ring.

\begin{lemma}
\label{lemma:connective-primitive-connective}
Let $M$ be a connective left $\Gamma(A)$-comodule in $\sp$.
If the multiplication map
$A\otimes A\to A$ induces an isomorphism
$\pi_0(A\otimes A)\stackrel{\cong}{\to}
   \pi_0A$,
then $P(M)$ is connective.
\end{lemma}

\begin{proof}
We have the Bousfield-Kan spectral sequence
\[ E_2^{s,t}\cong\pi^s\pi_tC^{\bullet}(A,A\otimes A,M)
   \Longrightarrow \pi_{t-s}P(M).\]
This is an upper half plane spectral sequence.
Since $A_*(A)$
is flat as a left and right $A_*$-module,
there is an isomorphism
\[ \pi_*C^r(A,A\otimes A,M)\cong
   C^r(A_*,A_*(A),M_*)\]
for any $r\ge 0$.
Hence we obtain an isomorphism
\[ E_2^{s,*}\cong
   H^s(\overline{C}{}^*(A_*,A_*(A),M_*)).\]
Let $\overline{\Gamma}$
be the kernel of the map $A_*(A)\to A_*$
induced by the multiplication.
By the assumption
that $\pi_0(A\otimes A)\stackrel{\cong}{\to}\pi_0(A)$,
$\overline{\Gamma}_t=0$ for $t\le 0$.
This implies that
\[ \overline{C}{}^s(A_*,A_*(A),M_*)_t=0\]
for $t<s$.
Hence $E_2^{s,t}=0$ for $t<s$.
The lemma follows from
\cite[Lemma~X.7.3]{Bousfield-Kan}.
\qed\end{proof}

We recall that $L_A\sp$
is the Bousfield localization of $\sp$
with respect to $A$
and $L_A:\sp\to \sp$ is the associated
localization functor and that 
we have the adjoint pair of functors
\[ A\otimes (-): L_A\sp\rightleftarrows 
   \lcomod_{\Gamma(A)}(\sp): P.\] 
We denote by $L_A{\rm Sp}^{\ge 0}$
the full subcategory of $L_A{\rm Sp}$
consisting of connective objects.
We also denote by 
$\lcomod_{\Gamma(A)}(\sp)^{\ge 0}$
the full subcategories of $\lcomod_{\Gamma(A)}(\sp)$
consisting of connective objects.

By Lemma~\ref{lemma:connective-primitive-connective},
we see that the functor 
$P: \lcomod_{\Gamma(A)}(\sp)\to L_A\sp$ restricted to
the full subcategory $\lcomod_{\Gamma(A)}(\sp)^{\ge 0}$
factors through $L_A\sp^{\ge 0}$
when the multiplication map
induces an isomorphism
$\pi_0(A\otimes A)\stackrel{\cong}{\to} \pi_0 A$.
Hence we obtain the following corollary.

\begin{corollary}
\label{cor:connective-adjunction}
If the multiplication map
induces an isomorphism
$\pi_0(A\otimes A)\stackrel{\cong}{\to}\pi_0A$,
then there is an adjunction of functors
\[ A\otimes (-):
   L_A\sp^{\ge 0}\rightleftarrows
   \lcomod_{\Gamma(A)}(\sp)^{\ge 0}:
   P.\]
\end{corollary}

We would like to show that
the pair of functors $(A\otimes (-),P)$
is an adjoint equivalence
under some conditions.
In order to show that $A\otimes (-)$
is fully faithful,
we have to show that the unit map
$X\to P(A\otimes X)$ is an equivalence
for any $X\in L_A\sp^{\ge 0}$.
We recall that the map $X\to P(A\otimes X)$
is an $A$-nilpotent completion in ${\rm Ho}(\sp)$
(see \S\ref{sec:Bousfield-Kan-SS}). 
The relation between the $A$-localization
and the $A$-nilpotent completion
was studied by Bousfield~\cite{Bousfield}. 

\begin{lemma}
\label{lem:connective-adjoint-fully-faithful}
We assume that $\pi_0A$ is isomorphic to
$\mathbb{Z}/n\mathbb{Z}$ 
for some $n\ge 2$
or $\mathbb{Z}[J^{-1}]$ for some set $J$ of primes.
Then the functor $A\otimes (-):
L_A\sp^{\ge 0}\to \lcomod_{\Gamma(A)}(\sp)^{\ge 0}$
is fully faithful.
\end{lemma}

\begin{proof}
Note that the multiplication map
induces an isomorphism
$\pi_0(A\otimes A)\stackrel{\cong}{\to}\pi_0(A)$
under the assumption
and hence we have the adjunction of functors
$A\otimes (-):L_A\sp^{\ge 0}\rightleftarrows 
\lcomod_{\Gamma(A)}(\sp)^{\ge 0}:P$ 
by Corollary~\ref{cor:connective-adjunction}.

We have to show that 
the unit map $X\to P(A\otimes X)$
of the adjunction $(A\otimes (-), P)$
is an equivalence for any $X\in L_A\sp^{\ge 0}$.
By \cite[Thm.~6.5 and 6.6]{Bousfield},
the $A$-localization $Y\to L_AY$ 
is equivalent to the $A$-nilpotent completion
$Y\to P(A\otimes Y)$
for any connective spectrum $Y$
under the assumption.
This implies that the map $X\to P(A\otimes X)$
is an equivalence for any $X\in L_A\sp^{\ge 0}$.
\qed\end{proof}

In order to show that the left adjoint $A\otimes (-)$
is essentially surjective,
we need the following lemma.
 
\begin{lemma}\label{lem:key-essential-surjective-connective}
We assume that the multiplication map induces an isomorphism
$\pi_0(A\otimes A)\stackrel{\cong}{\to}
   \pi_0A$.
Let $M$ be a connective left $\Gamma(A)$-comodule.
If $P(M)\simeq 0$,
then $M\simeq 0$.
\end{lemma}

\begin{proof}
We shall show that $P(M)\not\simeq 0$
if $M\not\simeq 0$.
Suppose that
$\pi_iM=0$ for $i<n$ and
$\pi_nM\neq 0$.
Let $E_r^{s,t}$ be the Bousfield-Kan
spectral sequence abutting to
$\pi_*P(M)$.
We have
$E_2^{s,*}\cong H^s(\overline{C}^*(A_*,A_*(A),M_*))$.
By the assumption,
$E_2^{s,t}=0$ for $t-s<n$
and $E_2^{0,n}\cong \pi_nM$.
In particular, we have
$E_{\infty}^{0,n}\neq 0$ and 
$\subrel{r}{\rm lim}^1 E_r^{s,s+n}=0$
for all $s\ge 0$.
By \cite[Lemma~IX.5.4]{Bousfield-Kan},
$E_{\infty}^{0,n}\cong
{\rm Im}(\pi_nP(M)\to \pi_nM)$.
Hence we obtain
$\pi_nP(M)\neq 0$.
\qed\end{proof}

\begin{theorem}
\label{thm:connective-case-adjoint-equivalence}
Let $A$ be a connective $\mathbb{S}$-algebra
such that $A_*(A)$ is flat as a left and right $A_*$-module.
We assume that $\pi_0A$ is isomorphic to
$\mathbb{Z}/n$ for some $n\ge 2$
or $\mathbb{Z}[J^{-1}]$ for some set $J$ of primes.
Then there is an adjoint equivalence
of quasi-categories
\[ A\otimes (-):
     L_A\sp^{\ge 0}\rightleftarrows
     \lcomod_{\Gamma(A)}(\sp)^{\ge 0}:
     P.\]
\end{theorem}

\begin{proof}
By Corollary~\ref{cor:connective-adjunction},
we have the adjunction of functors
$A\otimes (-):
     L_A\sp^{\ge 0}\rightleftarrows
     \lcomod_{\Gamma(A)}(\sp)^{\ge 0}:
     P$.
By Lemma~\ref{lem:connective-adjoint-fully-faithful},
the left adjoint $A\otimes (-)$
is fully faithful.
Hence it suffices to show that 
$A\otimes (-)$ is essentially surjective.

Let $M\in {\rm Comod}_{\Gamma(A)}(\sp)^{\ge 0}$.
By the counit of the adjunction,
we have a map
$\varepsilon: A\otimes P(M)\to M$.
Let $N$ be the cofiber of $\varepsilon$.
Since the unit map $P(M)\to P(A\otimes P(M))$
is an equivalence,
we see that $P(N)\simeq 0$.
By Lemma~\ref{lem:key-essential-surjective-connective},
we obtain $N\simeq 0$ and 
hence $A\otimes P(M)\simeq M$.
This shows that the left adjoint 
$A\otimes (-)$ is essentially surjective.
\qed\end{proof}

In the following of this subsection
we shall consider some examples.

First, we consider the complex cobordism spectrum $MU$.
The spectrum
$MU$ admits a commutative $\mathbb{S}$-algebra structure
by \cite{May}. 
Since $\pi_*MU=\mathbb{Z}[x_1,x_2,\ldots]$
with $|x_i|=2i$ for $i>0$,
$MU$ is a connective commutative $\mathbb{S}$-algebra,
the multiplication map induces an
isomorphism $\pi_0(MU\otimes MU)\stackrel{\cong}{\to}\pi_0MU$,
and $c\pi_0MU\cong\mathbb{Z}$.
Note that the localization $L_{MU}X$ coincides with $X$
for a connective spectrum $X$ by \cite[Thm.~3.1]{Bousfield}
and hence $L_{MU}\sp^{\ge 0}$ is equivalent to
$\sp^{\ge 0}$.
By Theorem~\ref{thm:connective-case-adjoint-equivalence},
we obtain the following corollary.

\begin{corollary}
[{cf.~\cite[6.1.2]{Hess}}]
There is an adjoint equivalence
  \[ MU\otimes (-):
     \sp^{\ge 0}\rightleftarrows
     \lcomod_{\Gamma(MU)}(\sp)^{\ge 0}:
     P.\]
\end{corollary}

Next, we consider
the Brown-Peterson spectrum $BP$ at a prime $p$.
The spectrum $BP$ admits an $\mathbb{S}$-algebra structure by
\cite[\S2]{Lazarev}.
The coefficient ring of $BP$ is a polynomial
ring over the ring $\mathbb{Z}_{(p)}$ 
of integers localized at $p$ with
infinite many variables
\[ \pi_*BP = \mathbb{Z}_{(p)}[v_1,v_2,\ldots],\]
with degree $|v_i|=2(p^i-1)$
for $i\ge 1$.
Hence $BP$ is a connective $\mathbb{S}$-algebra,
the multiplication map induces an
isomorphism $\pi_0(BP\otimes BP)\stackrel{\cong}{\to}\pi_0BP$,
and $c\pi_0BP\cong\mathbb{Z}_{(p)}$.
Note that the localization $L_{BP}X$ coincides with 
the $p$-localization $X_{(p)}$
for a connective spectrum $X$ by \cite[Thm.~3.1]{Bousfield}
and hence $L_{BP}\sp^{\ge 0}$ is equivalent to
the full subcategory $\sp_{(p)}^{\ge 0}$
of $p$-local spectra in ${\rm Sp}^{\ge 0}$.
By Theorem~\ref{thm:connective-case-adjoint-equivalence},
we obtain the following corollary.

\begin{corollary}
There is an adjoint equivalence
  \[ BP\otimes (-):
     \sp_{(p)}^{\ge 0}\rightleftarrows
     \lcomod_{\Gamma(BP)}(\sp)^{\ge 0}: P.\]
\end{corollary}

Finally,
we consider the mod $p$
Eilenberg-Mac~Lane spectrum $H\mathbb{F}_p$ 
for a prime $p$.   
We know that
$H\mathbb{F}_p$ is a connective commutative $\mathbb{S}$-algebra.
Since $\pi_0H\mathbb{F}_p\cong\mathbb{F}_p$,
we see that the multiplication induces 
an isomorphism $\pi_0(H\mathbb{F}_p\otimes H\mathbb{F}_p)
\stackrel{\cong}{\to}\pi_0H\mathbb{F}_p$ and
$c\pi_0H\mathbb{F}_p\cong\mathbb{F}_p$. 
If $X$ is a connective spectrum,
then $L_{H\mathbb{F}_p}X$ is equivalent
to the $p$-completion of $X$ by \cite[Thm.~3.1]{Bousfield},
and hence 
$L_{H\mathbb{F}_p}\isp^{\ge 0}$ is 
equivalent to the full subcategory $(\isp_p^{\wedge})^{\ge 0}$
of $p$-complete spectra in $\isp^{\ge 0}$.
By Theorem~\ref{thm:connective-case-adjoint-equivalence},
we obtain the following corollary.

\begin{corollary}
[{cf.~\cite[6.1.1]{Hess}}]
There is an adjoint equivalence
  \[ H\mathbb{F}_p\otimes (-):
     (\sp_{p}^{\wedge})^{\ge 0}\rightleftarrows
     \lcomod_{\Gamma(H\mathbb{F}_p)}(\isp)^{\ge 0}: P.\]
\end{corollary}

\subsection{A model of the $K(n)$-local category}

Let $K(n)$ be the $n$th Morava $K$-theory spectrum at a prime $p$
and $\mathbb{G}_n$ the $n$th Morava stabilizer group. 
In this subsection
we show that the category of module objects
over $F_n$ in the $K(n)$-local 
discrete symmetric $\mathbb{G}_n$-spectra
models the $K(n)$-local category,
where $F_n$ is a discrete model of the
$n$th Morava $E$-theory spectrum. 

The $K(n)$-local category is the Bousfield localization
of the stable homotopy category 
of spectra with respect to $K(n)$.
It is known that the $K(n)$-local categories
for various $n$ and $p$
are fundamental building blocks of the stable homotopy 
category of spectra.
Thus it is important to understand 
the $K(n)$-local category.

Let $E_n$ be the $n$th Morava $E$-theory spectrum at $p$.
The Morava $E$-theory spectrum $E_n$
is a commutative ring spectrum
in the stable homotopy category of spectra
and $\mathbb{G}_n$ is identified with the group of
multiplicative automorphisms of $E_n$.
By Goerss-Hopkins \cite{Goerss-Hopkins},
it was shown that the commutative ring spectrum structure
on $E_n$ can be lifted to a unique $E_{\infty}$-ring
spectrum structure up to homotopy.
Furthermore,
it was shown that
$\mathbb{G}_n$ acts on $E_n$
in the category of $E_{\infty}$-ring spectra.
There is a $K(n)$-local $E_n$-based Adams spectral sequence 
abutting to the homotopy groups of the $K(n)$-local sphere
whose $E_2$-page is the continuous cohomology groups
of $\mathbb{G}_n$ with coefficients 
in the homotopy groups of $E_n$.
This suggests that
the $K(n)$-local sphere may be 
the $\mathbb{G}_n$-homotopy fixed points of $E_n$.
Motivated by this observation,
Devinatz-Hopkins \cite{Devinatz-Hopkins} 
constructed a $K(n)$-local $E_{\infty}$-ring spectrum
$E_n^{dhU}$ for any open subgroup $U$ of $\mathbb{G}_n$,
which has expected properties of
the homotopy fixed points spectrum.  

Davis \cite{Davis}
constructed a discrete $\mathbb{G}_n$-spectrum 
$F_n$ which is defined by
\[ F_n=\, {\rm colim}_{U}\, E_n^{dhU},\]
where $U$ ranges over the open subgroups of $\mathbb{G}_n$.
The spectrum $F_n$ is a discrete model of $E_n$
and actually we can recover $E_n$ from $F_n$
by the $K(n)$-localization as
\[ L_{K(n)}F_n\simeq E_n.\]
Furthermore,
Behrens-Davis \cite{Behrens-Davis}
upgraded the discrete $\mathbb{G}_n$-spectrum $F_n$
to a commutative monoid object
in the category $\ssp(\mathbb{G}_n)$
of discrete symmetric $\mathbb{G}_n$-spectra,
and showed that the unit map
$L_{K(n)}\mathbb{S}\to F_n$ is a consistent
$K(n)$-local $\mathbb{G}_n$-Galois extension.

We can give a model structure on 
the category $\ssp(\mathbb{G}_n)$
of discrete symmetric $\mathbb{G}_n$-spectra
and consider the left Bousfield localization 
$\ssp(\mathbb{G}_n)_{K(n)}$
with respect to $K(n)$
(see \cite{Behrens-Davis}).
The category $\ssp(\mathbb{G}_n)_{K(n)}$ is a left proper,
combinatorial, $\ssp$-model category. 

The unit map $\mathbb{S}\to F_n$ induces   
a symmetric monoidal $\ssp$-Quillen
adjunction  
\[ {\rm Ex}: \ssp_{K(n)}
             \rightleftarrows 
             \lmod_{F_n}(\ssp(\mathbb{G}_n)_{K(n)}):{\rm Re},\]
where $\ssp_{K(n)}$ is the left Bousfield localization
of the category $\ssp$ of symmetric spectra
with respect to $K(n)$.
In \cite{Torii1}
we showed that
the total left derived functor
$\mathbb{L}{\rm Ex}$ of ${\rm Ex}$
is fully faithful
as an $\Ho(\isp)$-enriched functor.
In this subsection
we shall show that the adjunction
is actually a Quillen equivalence and 
hence we can regard $\lmod_{F_n}(\ssp(\mathbb{G}_n)_{K(n)})$
as a model of the $K(n)$-local category.

We denote by $\isp_{K(n)}$ the underlying quasi-category
of the simplicial model category $\ssp_{K(n)}$.
The quasi-category $\isp_{K(n)}$ is a stable homotopy theory
with the tensor product $L_{K(n)}(-\otimes -)$
and the unit $L_{K(n)}\mathbb{S}$.
Since we can regard $E_n$ as an algebra
object of $\isp_{K(n)}$,
we can consider the coalgebra
$L_{K(n)}(E_n\otimes E_n)$ in 
${}_{E_n}\bmod_{E_n}(\isp_{K(n)})$ and
the quasi-category of left
$\Gamma(E_n)$-comodules 
\[ \lcomod_{\Gamma(E_n)}(\isp_{K(n)}),\]
where $\Gamma(E_n)=(E_n,L_{K(n)}(E_n\otimes E_n))$.

\begin{proposition}
We have an equivalence of quasi-categories
\[ L_{K(n)}(E_n\otimes(-)):
   \isp_{K(n)}\stackrel{\simeq}{\longrightarrow}
   \lcomod_{\Gamma(E_n)}(\isp_{K(n)}).\] 
\end{proposition}

\begin{proof}
We shall apply 
Theorem~\ref{theorem:general-equivalence-comodules}
for the stable homotopy theory $\isp_{K(n)}$
and the algebra object $E_n$.
The unit object $L_{K(n)}\mathbb{S}$
is $E_n$-nilpotent in $\isp_{K(n)}$
by \cite[Prop.~A.3]{Devinatz-Hopkins}.
Note that a map $f: X\to Y$ in $\isp_{K(n)}$
is an equivalence if and only if
$L_{K(n)}(E_n\otimes f)$ is an equivalence
since $K(n)\otimes E_n$ is a wedge of copies of $K(n)$.
Hence $L_{E_n}(\isp_{K(n)})\simeq \isp_{K(n)}$
and the proposition follows from
Theorem~\ref{theorem:general-equivalence-comodules}.
\qed
\end{proof}


We have an adjunction of quasi-categories
\begin{equation}\label{adjunction:ispkn-modenispkn}
\isp_{K(n)}\rightleftarrows \lmod_{E_n}(\isp_{K(n)}),
\end{equation}
where the left adjoint is given by smashing with $E_n$
in $\sp_{K(n)}$
and the right adjoint is the forgetful functor.
By Theorem~\ref{thm:comparibility-of-definition-of-comodules},
we have an equivalence of quasi-categories
\[ \lcomod_{\Gamma(E_n)}(\isp_{K(n)})\simeq
   \lcomod_{\Theta}(\lmod_{E_n}(\isp_{K(n)})^{\rm op})^{\rm op},\] 
where $\Theta$ is the comonad associated to
adjunction~(\ref{adjunction:ispkn-modenispkn}).
Hence we see that the forgetful functor
$\lmod_{E_n}(\isp_{K(n)})\to \sp_{K(n)}$
exhibits the quasi-category $\sp_{K(n)}$ as comonadic
over $\lmod_{E_n}(\isp_{K(n)})$
(see also \cite[Prop.~10.10]{Mathew}).

Let $\isp(\mathbb{G}_n)_{K(n)}$
be the underlying quasi-categories
of $\ssp(\mathbb{G}_n)_{K(n)}$.
We can regard $F_n$ as an algebra
object of $\isp(\mathbb{G}_n)_{K(n)}$,
and form a quasi-category 
\[ \lmod_{F_n}(\isp(\mathbb{G}_n)_{K(n)}) \]
of left modules objects over $F_n$
in $\isp(\mathbb{G}_n)_{K(n)}$.
Note that 
$\lmod_{F_n}(\isp(\mathbb{G}_n)_{K(n)})$
is equivalent to the underlying quasi-category
of the symmetric monoidal $\ssp$-model category
$\lmod_{F_n}(\ssp(\mathbb{G}_n)_{K(n)})$.
The adjunction $({\rm Ex},{\rm Re})$
of the $\ssp$-model categories
induces an adjunction of the underlying quasi-categories
\[ {\mathcal Ex}:\isp_{K(n)}\rightleftarrows
    \lmod_{F_n}(\isp(\mathbb{G}_n)_{K(n)}):{\mathcal Re}.\]

Let $U: \ssp(\mathbb{G}_n)_{K(n)}\to\ssp_{K(n)}$
be the forgetful functor.
We can regard $UF_n$ as a commutative monoid
object in $\ssp_{K(n)}$,
and the unit map $\mathbb{S}\to UF_n$ induces
a symmetric monoidal $\ssp$-Quillen adjunction 
\[ 
\ssp_{K(n)}\rightleftarrows\lmod_{UF_n}(\ssp_{K(n)}),
\]
where the left adjoint is given by smashing
with $UF_n$ and the right adjoint is given by
the forgetful functor.
This induces an adjunction
of the underlying quasi-categories
\begin{equation}\label{adjunction-sspkn-modufnsspkn} 
\isp_{K(n)}\rightleftarrows\lmod_{UF_n}(\isp_{K(n)}).
\end{equation}
Hence we can consider the quasi-category
of comodules
\[ \comod_{(UF_n,\Theta)}(\isp_{K(n)}) \]
associated to the adjunction and 
a map of quasi-categories
\[ {\rm Coex}:\isp_{K(n)}\longrightarrow
   \comod_{(UF_n,\Theta)}(\isp_{K(n)}).\] 

In \cite{Torii1}
we showed that
there is an equivalence of quasi-categories
\[ \lmod_{F_n}(\isp(\mathbb{G}_n)_{K(n)})\simeq
   \comod_{(UF_n,\Theta)}(\isp_{K(n)}),\]
and there is an equivalence of functors
\[ {\mathcal Ex}\simeq {\rm Coex}\]
under this equivalence.

Since the canonical map $UF_n\to E_n$ of commutative
algebras is a weak equivalence in $\ssp_{K(n)}$,
we have a Quillen equivalence  
\[ \lmod_{UF_n}(\ssp_{K(n)})\stackrel{\simeq}{\longrightarrow} 
   \lmod_{E_n}(\ssp_{K(n)}),\]
and hence we obtain an equivalence of
the underlying quasi-categories
\[ \lmod_{UF_n}(\isp_{K(n)})\stackrel{\simeq}{\longrightarrow} 
   \lmod_{E_n}(\isp_{K(n)}).\]
Under this equivalence,
we can identify two adjunctions 
(\ref{adjunction:ispkn-modenispkn})
and (\ref{adjunction-sspkn-modufnsspkn}),
and hence the forgetful functor
$\lmod_{UF_n}(\isp_{K(n)})\to \isp_{K(n)}$
exhibits $\isp_{K(n)}$ as comonadic
over $\lmod_{UF_n}(\isp_{K(n)})$,
that is,
the functor ${\rm Coex}$ is an equivalence
of quasi-categories.

The adjunction $({\mathcal Ex},{\mathcal Re})$
of quasi-categories
induces an adjunction of the homotopy categories
\[ {\rm Ho}(\isp_{K(n)})\rightleftarrows
   {\rm Ho}(\lmod_{F_n}(\isp(\mathbb{G}_n)_{K(n)})),\]
which is identified with the derived adjunction
of the Quillen adjunction
$({\rm Ex},{\rm Re})$.
Since the functor ${\rm Coex}$
is equivalent to ${\mathcal Ex}$ and
is an equivalence of quasi-categories,
the total left derived functor
$\mathbb{L}{\rm Ex}$ is an equivalence of categories.
Hence we obtain the following theorem.

\begin{theorem}
\label{thm:model-K(n)-local-category}
The adjunction
\[ {\rm Ex}: \ssp_{K(n)}
             \rightleftarrows 
             {\rm Mod}_{F_n}(\ssp(\mathbb{G}_n)_{K(n)}):{\rm Re} \]
is a Quillen equivalence
and hence the category
${\rm Mod}_{F_n}(\ssp(\mathbb{G}_n)_{K(n)})$
models the $K(n)$-local category.
\end{theorem}

\section{Proof of 
Proposition~\ref{prop:existence-final-object-sections}}
\label{sec:proof-theorem-construction-RF}

In this section we prove 
Proposition~\ref{prop:existence-final-object-sections}
stated in \S\ref{subsec:opposite-coCartesian-fibrations},
which is technical but important for
constructing a canonical map
between opposite coCartesian fibrations.
First, we give some basic examples of inner anodyne maps
and study opposite marked anodyne maps.
In \S\ref{subsec:ODeltan}
we introduce a marked simplicial set
$\widetilde{\mathcal{O}}(\Delta^n)^+$
in which the underlying simplicial set is 
$\widetilde{\mathcal{O}}(\Delta^n)$
and study inclusions of subcomplexes
of the marked simplicial sets
$\widetilde{\mathcal{O}}(\Delta^n)^+$ and 
$(\widetilde{\mathcal{O}}(\Delta^n)^+\times(\Delta^{\{0\}})^{\flat})
\cup
(\widetilde{\mathcal{O}}(\Delta^n)^{\flat}\times(\Delta^1)^{\flat})$.
In \S\ref{subsec:proof-Prop-1}
we give a proof of 
Proposition~\ref{prop:existence-final-object-sections}.

\subsection{Examples of inner anodyne maps}
\label{subsec:inner-anodyne-maps}

In this subsection 
we give some basic examples of inner anodyne maps.

A map of simplicial sets is said to be inner anodyne
if it has the left lifting property with respect
to all inner fibrations.
The class of inner anodyne maps 
is the smallest weakly saturated class of morphisms
generated by
all horn inclusions
$\Lambda^n_i\hookrightarrow \Delta^n$
for $0<i<n$.

For a sequence $i_1,\ldots,i_k$ of integers
such that $0\le i_1<\ldots<i_k\le n$,
we denote by
$\Lambda^n(i_1,\ldots,i_k)$ the subcomplex
$\bigcup_{i\neq i_1,\ldots,i_k}d_i\Delta^{n-1}$
of $\Delta^n$.

\begin{lemma}
\label{lemma:fundamental-extension-lemmaI}
The inclusion 
$\Lambda^n(i_1,\ldots,i_k)\hookrightarrow \Delta^n$
is an inner anodyne map
for $k>0$ and $0<i_1<\ldots<i_k<n$.
\end{lemma}

\begin{proof}
We shall prove the lemma by induction on $k$.
When $k=1$, the inclusion
is the map $\Lambda^n_i\hookrightarrow\Delta^n$
for $0<i=i_1<n$ and hence
it is an inner anodyne map.
Suppose the lemma holds for $k-1$
and we shall prove the lemma for $k$.
The subcomplex 
$\Lambda^n(i_1,\ldots,i_k)\cap d_{i_k}\Delta^{n-1}$
of $d_{i_k}\Delta^{n-1}$
is isomorphic to the subcomplex
$\Lambda^{n-1}(i_1,\ldots,i_{k-1})$ of $\Delta^{n-1}$.
By the hypothesis of induction,
the inclusion
$\Lambda^{n-1}(i_1,\ldots,i_{k-1})\hookrightarrow\Delta^{n-1}$
is an inner anodyne map.
By the cobase change of the inclusion 
$\Lambda^{n-1}(i_1,\ldots,i_{k-1})\hookrightarrow\Delta^{n-1}$
along the map
$\Lambda^{n-1}(i_1,\ldots,i_{k-1})\cong
\Lambda^n(i_1,\ldots,i_k)\cap d_{i_k}\Delta^{n-1}
\hookrightarrow\Lambda^n(i_1,\ldots,i_k)$,
we see that the inclusion
$\Lambda^n(i_1,\ldots,i_k)\hookrightarrow
\Lambda^n(i_1,\ldots,i_{k-1})$
is an inner anodyne map.  
By the hypothesis of induction,
the inclusion
$\Lambda^n(i_1,\ldots,i_{k-1})\hookrightarrow \Delta^n$
is an inner anodyne map.
Hence the composition
$\Lambda^n(i_1,\ldots,i_k)\hookrightarrow
\Lambda^n(i_1,\ldots,i_{k-1})\hookrightarrow \Delta^n$
is also an inner anodyne map.
\qed\end{proof}

\begin{lemma}
\label{lemma:fundamental-extension-lemmaII}
The inclusion 
$\Lambda^n(0,i_1,\ldots,i_k)\hookrightarrow\Delta^n$
is an inner anodyne map 
for $k>0$ and $1<i_1<\ldots<i_k<n$.
\end{lemma}

\begin{proof}
The subcomplex $\Lambda^n(0,i_1,\ldots,i_k)\cap d_0\Delta^{n-1}$
of $d_0\Delta^{n-1}$
is isomorphic to the subcomplex
$\Lambda^{n-1}(i_1-1,\ldots,i_k-1)$ of $\Delta^{n-1}$.
Since $0<i_1-1<\cdots<i_k-1<n-1$,
the inclusion $\Lambda^{n-1}(i_1-1,\ldots,i_k-1)
\hookrightarrow \Delta^{n-1}$
is an inner anodyne map
by Lemma~\ref{lemma:fundamental-extension-lemmaI}.
By the cobase change of 
$\Lambda^{n-1}(i_1-1,\ldots,i_k-1)
\hookrightarrow \Delta^{n-1}$
along the map
$\Lambda^{n-1}(i_1-1,\ldots,i_k-1)\cong
\Lambda^n(0,i_1,\ldots,i_k)\cap d_0\Delta^{n-1}
\hookrightarrow \Lambda^n(0,i_1,\ldots,i_k)$,
we see that 
the inclusion
$\Lambda^n(0,i_1,\ldots,i_k)\hookrightarrow
\Lambda^n(i_1,\ldots,i_k)$
is an inner anodyne map.
Since
the inclusion
$\Lambda^n(i_1,\ldots,i_k)\hookrightarrow
\Delta^n$ is an inner anodyne map
by Lemma~\ref{lemma:fundamental-extension-lemmaI},
the composition
$\Lambda^n(0,i_1,\ldots,i_k)\hookrightarrow
\Lambda^n(i_1,\ldots,i_k)\hookrightarrow
\Delta^n$ is also an inner anodyne map.
\qed\end{proof}

\subsection{Opposite marked anodyne maps}
\label{subsec:opposite-marked-anodyne}

In this subsection we study 
opposite marked anodyne maps.

A marked simplicial set is a
pair $(K,\mathcal{E})$,
where $K$ is a simplicial set
and $\mathcal{E}$ is a set of edges of $K$
that contains all degenerate edges.
A map of marked simplicial sets 
$(K,\mathcal{E})\to (L,\mathcal{E}')$
is a map of simplicial set $f:K\to L$
such that $f(\mathcal{E})\subset \mathcal{E}'$.
We denote by ${\rm sSet}^{+}$
the category of marked simplicial sets.

For a simplicial set $K$,
we denote by $K^{\flat}$
the marked simplicial set $(K,s_0(K_0))$,
where $s_0(K_0)$ is the set of all degenerate edges of $K$,
and by $K^{\sharp}$
the marked simplicial set $(K,K_1)$,
where $K_1$ is the set of all edges of $K$.

For a marked simplicial set $(K,\mathcal{E})$,
we have the opposite marked simplicial set
$(K,\mathcal{E})^{\rm op}=(K^{\rm op},\mathcal{E}^{\rm op})$,
where $K^{\rm op}$ is the opposite simplicial set of $K$
and $\mathcal{E}^{\rm op}$ is the corresponding
set of edges of $K^{\rm op}$.

We say that a map of marked simplicial sets
$K\to L$ is an opposite marked anodyne map
if the opposite $K^{\rm op}\to L^{\rm op}$
is a marked anodyne map
defined in \cite[Def.~3.1.1.1]{Lurie1}.
The class of opposite marked anodyne maps
in ${\rm sSet}^+$
is the smallest weakly saturated class of morphisms
with the following properties:
\begin{enumerate}
\item
For each $0<i<n$,
the inclusion $(\Lambda^n_i)^{\flat}\hookrightarrow
(\Delta^n)^{\flat}$ is opposite marked anodyne.
\item
For every $n>0$,
the inclusion
\[ (\Lambda^n_0,(\Lambda^n_0)_1\cap\mathcal{E})\hookrightarrow
   (\Delta^n,\mathcal{E}) \]
is opposite marked anodyne, where $\mathcal{E}$
denotes the set of all degenerate edges of $\Delta^n$
together with the initial edge $\Delta^{\{0,1\}}$.
\item
The inclusion
\[ (\Lambda^2_1)^{\sharp}\coprod_{(\Lambda^2_1)^{\flat}}
   (\Delta^2)^{\flat}\hookrightarrow
   (\Delta^2)^{\sharp} \]
is opposite marked anodyne. 
\item
For every Kan complex $K$,
the map
$K^{\flat}\to K^{\sharp}$ is opposite marked anodyne.  
\end{enumerate}

\begin{lemma}
\label{lem:marked-anodyne-simplex}
The inclusion 
$(\Lambda^n_0)^{\sharp}\hookrightarrow
(\Delta^n)^{\sharp}$
is opposite marked anodyne for $n>0$.
\end{lemma}

\begin{proof}
When $n=1$, the lemma holds by 
property 2 of the class of opposite marked anodyne maps.
We consider the case $n=2$. 
By \cite[Cor.~3.1.1.7]{Lurie1},
the inclusion
$(\Lambda^2_0)^{\sharp}\coprod_{(\Lambda^2_0)^{\flat}}(\Delta^2)^{\flat}
\hookrightarrow (\Delta^2)^{\sharp}$
is opposite marked anodyne.
The inclusion $(\Lambda^2_0,(\Lambda^2_0)_1\cap\mathcal{E})
\hookrightarrow
(\Delta^2,\mathcal{E})$ is opposite marked anodyne
by property 2 of the class of opposite marked anodyne maps,
where $\mathcal{E}$ is the set of edges of $\Delta^2$
consisting of all degenerate edges together with $\Delta^{\{0,1\}}$.
Taking the pushout of
$(\Lambda^2_0,(\Lambda^2_0)_1\cap\mathcal{E})\hookrightarrow
(\Delta^2,\mathcal{E})$
along the map $(\Lambda^2_0,(\Lambda^2_0)_1\cap\mathcal{E})
\to (\Lambda^2_0)^{\sharp}$,
we see that
the inclusion $(\Lambda^2_0)^{\sharp}\hookrightarrow
(\Lambda^2_0)^{\sharp}\coprod_{(\Lambda^2_0)^{\flat}}(\Delta^2)^{\flat}$
is opposite marked anodyne.
Hence the composition
$(\Lambda^2_0)^{\sharp}\hookrightarrow
(\Lambda^2_0)^{\sharp}\coprod_{(\Lambda^2_0)^{\flat}}(\Delta^2)^{\flat}
\hookrightarrow (\Delta^2)^{\sharp}$
is also opposite marked anodyne.

Now we consider the case $n\ge 3$.
The inclusion $(\Lambda^n_0,(\Lambda^n_0)_1\cap\mathcal{E})
\hookrightarrow
(\Delta^n,\mathcal{E})$ is opposite marked anodyne
by property 2 of the class of opposite marked anodyne maps,
where $\mathcal{E}$ is the set of edges of $\Delta^n$
consisting of all degenerate edges together with $\Delta^{\{0,1\}}$.
Taking the pushout of
$(\Lambda^n_0,(\Lambda^n_0)_1\cap\mathcal{E})
\hookrightarrow
(\Delta^n,\mathcal{E})$
along the map
$(\Lambda^n_0,(\Lambda^n_0)_1\cap\mathcal{E})\to
(\Lambda^n_0)^{\sharp}$,
we see that 
the inclusion
$(\Lambda^n_0)^{\sharp}\hookrightarrow
(\Delta^n)^{\sharp}$ is 
opposite marked anodyne.
\qed\end{proof}

\begin{lemma}
\label{lemma:extension_the_first_step}
Let $K=(\Delta^n\times
\partial\Delta^{1})\cup(\Lambda^n_0\times\Delta^1)$
be the subcomplex of $\Delta^n\times \Delta^1$
for $n\ge 1$.
Let $\mathcal{E}$ be the set of edges of 
$\Delta^n\times\Delta^1$
consisting of all degenerate edges
together with $\Delta^{\{0,1\}}\times\Delta^{\{0\}}$.
The inclusion
$(K, K_1\cap \mathcal{E})
\hookrightarrow
(\Delta^n\times\Delta^1,\mathcal{E})$
is an opposite marked anodyne map.
\if
Let $r: Z\to T$ be an inner fibration.
Suppose we have a commutative diagram
\[ \begin{array}{ccc}
     K & \stackrel{f}{\longrightarrow} & Z\\[1mm]
     \mbox{$\scriptstyle i$}\big\downarrow
     \phantom{\mbox{$\scriptstyle i$}} & &
     \phantom{\mbox{$\scriptstyle r$}}
     \big\downarrow\mbox{$\scriptstyle r$}\\[1mm]
     \Delta^n\times\Delta^1 &
     \stackrel{g}{\longrightarrow} & T\\
   \end{array}\]
in ${\rm sSet}$,
where $i$ is the inclusion.
If $f(\Delta^{\{0,1\}}\times\Delta^{\{0\}})$
is an $r$-coCartesian edge,
then there exists a map
$h:\Delta^n\times\Delta^1\to Z$ such that
$ih=f$ and $rh=g$.
\fi
\end{lemma}

\begin{proof}
Put $L(i)=(\Delta^{\{0,\ldots,i\}}\times\Delta^{\{0\}})
\star(\Delta^{\{i,\ldots,n\}}\times\Delta^{\{1\}})$
for $0\le i\le n$.
We set $\overline{L}(i)= K\cup (\cup_{j=0}^i L(j))$
for $0\le i\le n$.
Note that $\overline{L}(n)=\Delta^n\times\Delta^1$. 

First, we show that
the inclusion $K\hookrightarrow \overline{L}(n-1)$
is inner anodyne.
Since $L(0)\cap K$ is isomorphic to $\Lambda^{n+1}_1$
in $L(0)\cong \Delta^{n+1}$,
we see that the inclusion $K\hookrightarrow \overline{L}(0)$
is inner anodyne.
For $0<i<n$,
since $L(i)\cap \overline{L}(i-1)$
is isomorphic to $\Lambda^{n+1}(0,i+1)$
in $L(i)\cong\Delta^{n+1}$,
we see that the inclusion
$\overline{L}(i-1)\hookrightarrow
\overline{L}(i)$ is inner anodyne
by Lemma~\ref{lemma:fundamental-extension-lemmaII}.
Hence the composition
$K\hookrightarrow \overline{L}(0)\hookrightarrow
\cdots \hookrightarrow\overline{L}(n-1)$
is also inner anodyne.

Since the class of inner anodyne maps
is stable under the opposite,
the inclusion $K^{\rm op}\hookrightarrow
\overline{L}(n-1)^{\rm op}$ is also inner anodyne.
By \cite[Remark~3.1.1.4]{Lurie1},
we see that 
$K^{\flat}\hookrightarrow 
\overline{L}(n-1)^{\flat}$
is an opposite marked anodyne map.
This implies that
the inclusion
$(K,K_1\cap \mathcal{E})\hookrightarrow
 (\overline{L}(n-1),\overline{L}(n-1)_1\cap\mathcal{E})$
is opposite marked anodyne.

Now we consider the inclusion
$\overline{L}(n-1)\hookrightarrow\overline{L}(n)$.
We see that
$L(n)\cap\overline{L}(n-1)$ is isomorphic to $\Lambda^{n+1}_0$
in $L(n)\cong \Delta^{n+1}$.
We can identify
$(L(n),L(n)_1\cap\mathcal{E})$
with $(\Delta^{n+1},\mathcal{E}')$,
where $\mathcal{E}'$ is the set of edges of $\Delta^{n+1}$
consisting of all degenerate edges together with
$\Delta^{\{0,1\}}$.
Since the map 
$(\Lambda^{n+1}_0,(\Lambda^{n+1}_0)_1\cap\mathcal{E}')
\hookrightarrow (\Delta^{n+1},\mathcal{E}')$
is opposite marked anodyne,
we see that 
$(\overline{L}(n-1),\overline{L}(n-1)_1\cap \mathcal{E})
\to 
(\overline{L}(n),\mathcal{E})$
is opposite marked anodyne.

Therefore,
the composition
$(K,K_1\cap \mathcal{E})\hookrightarrow
(\overline{L}(n-1),\overline{L}(n-1)_1\cap\mathcal{E})
\hookrightarrow
(\overline{L}(n),\mathcal{E})$
is also an opposite marked anodyne map.
This completes the proof.
\qed\end{proof}

\subsection{The marked simplicial set 
$\widetilde{\mathcal{O}}(\Delta^n)^+$}
\label{subsec:ODeltan}

In this subsection
we introduce a marked simplicial set
$\widetilde{\mathcal{O}}(\Delta^n)^+$
in which the underlying simplicial set is 
$\widetilde{\mathcal{O}}(\Delta^n)$.
We study inclusions of subcomplexes of
the marked simplicial sets
$\widetilde{\mathcal{O}}(\Delta^n)^+$ and
$(\widetilde{\mathcal{O}}(\Delta^n)^+\times (\Delta^{\{0\}})^{\flat})\cup
(\widetilde{\mathcal{O}}(\Delta^n)^{\flat}\times (\Delta^1)^{\flat})$.

Let $\widetilde{\mathcal{E}}$
be the set of edges of $\widetilde{\mathcal{O}}(\Delta^n)$
consisting of all non-degenerate edges together
with edges $ij\to ik$ for $0\le i\le j\le k\le n$.
We regard the pair $(\widetilde{\mathcal{O}}(\Delta^n),
\widetilde{\mathcal{E}})$ as a marked simplicial set.
For a subcomplex $K$ of $\widetilde{\mathcal{O}}(\Delta^n)$,
we set $\widetilde{\mathcal{E}}_K=\widetilde{\mathcal{E}}\cap K_1$
and denote by $K^+$
the marked simplicial set $(K,\widetilde{\mathcal{E}}_K)$.

For $n>0$,
we let $M_n$ be the subcomplex of 
$\widetilde{\mathcal{O}}(\Delta^n)$
that contains all non-degenerate $k$-simplexes 
for $0\le k\le n$
except for the $n$-simplex corresponding
to $nn\to \cdots \to 0n$.

\begin{lemma}
\label{lem:opposite-makred-anodyne-partial-M}
The inclusion
$\widetilde{\mathcal{O}}(\partial\Delta^n)^+
\hookrightarrow M_n^+$
is an opposite marked anodyne map for all $n>0$.
\end{lemma}

\begin{proof}
First, we consider the case $n=1$.
We let $B_0$ be the $1$-simplex corresponding 
to $00\to 01$.
The subcomplex $B_0\cap \widetilde{\mathcal{O}}(\partial\Delta^1)$
is isomorphic to $\Lambda^1_0$ in $B_0\cong\Delta^1$.
The inclusion $(\Lambda^1_0)^{\sharp}\hookrightarrow
(\Delta^1)^{\sharp}$ is opposite marked anodyne
by Lemma~\ref{lem:marked-anodyne-simplex}.
Taking the pushout of $(\Lambda^1_0)^{\sharp}\hookrightarrow
(\Delta^1)^{\sharp}$ along
the map $(\Lambda^1_0)^{\sharp}\cong
B_0^+\cap \widetilde{\mathcal{O}}(\partial\Delta^1)^+
\to \widetilde{\mathcal{O}}(\partial\Delta^1)^+$,
we see that 
the inclusion
$\widetilde{\mathcal{O}}(\partial\Delta^1)^+\hookrightarrow 
M_1^+$
is opposite marked anodyne.

Next, we consider the case $n=2$.
Let $B_0$ be the $2$-simplex in 
$\widetilde{\mathcal{O}}(\Delta^2)$ corresponding
to $00\to 01\to 02$.
The subcomplex $B_0^+\cap 
\widetilde{\mathcal{O}}(\partial\Delta^2)^+$
is isomorphic to $(\Lambda^2_0)^{\sharp}$ in 
$B_0^+\cong (\Delta^2)^{\sharp}$.
By Lemma~\ref{lem:marked-anodyne-simplex},
the inclusion $(\Lambda^2_0)^{\sharp}\hookrightarrow
(\Delta^2)^{\sharp}$ is opposite marked anodyne. 
Taking the pushout
of $(\Lambda^2_0)^{\sharp}\hookrightarrow
(\Delta^2)^{\sharp}$ along
the map
$(\Lambda^2_0)^{\sharp}\cong
B_0^+\cap\widetilde{\mathcal{O}}(\partial\Delta^2)^+
\to \widetilde{\mathcal{O}}(\partial\Delta^2)^+$,
we obtain an opposite marked anodyne map
$\widetilde{\mathcal{O}}(\partial\Delta^2)^+\hookrightarrow
\widetilde{\mathcal{O}}(\partial\Delta^2)^+\cup B_0^+$.
Let $B_1(0)$ be the $2$-simplex in 
$\widetilde{\mathcal{O}}(\Delta^2)$ corresponding
to $11\to 01\to 02$.
The subcomplex $B_1(0)\cap 
(\widetilde{\mathcal{O}}(\partial\Delta^2)\cup B_0)$
is isomorphic to $\Lambda^2_1$ in $\Delta^2$.
The inclusion $(\Lambda^2_1)^{\flat}\hookrightarrow
(\Delta^2)^{\flat}$ is opposite marked anodyne.
Taking the pushout of 
$(\Lambda^2_1)^{\flat}\hookrightarrow
(\Delta^2)^{\flat}$
along the map
$(\Lambda^2_1)^{\flat}\to 
\widetilde{\mathcal{O}}(\partial\Delta^2)^+\cup B_0^+$,
we obtain an opposite marked anodyne map
$\widetilde{\mathcal{O}}(\partial\Delta^2)^+\cup B_0^+
\to \widetilde{\mathcal{O}}(\partial\Delta^2)^+\cup
B_0^+\cup B_1(0)^+$. 
Let $B_1(1)$ be the $2$-simplex in 
$\widetilde{\mathcal{O}}(\Delta^2)$ corresponding
to $11\to 12\to 02$.
The subcomplex $B_1(1)\cap 
(\widetilde{\mathcal{O}}(\partial\Delta^2)\cup B_0\cup
B_1(0))$
is isomorphic to $\Lambda^2_0$ in $\Delta^2$.
The inclusion
$(\Lambda^2_0,(\Lambda^2_0)_1\cap\mathcal{E}')\hookrightarrow
(\Delta^2,\mathcal{E}')$ is opposite marked anodyne,
where $\mathcal{E}'$ is the set of edges of $\Delta^2$
consisting of all degenerate edges together
with $\Delta^{\{0,1\}}$. 
Taking the pushout of 
$(\Lambda^2_0,(\Lambda^2_0)_1\cap\mathcal{E}')\hookrightarrow
(\Delta^2,\mathcal{E}')$
along the map
$(\Lambda^2_0,(\Lambda^2_0)_1\cap\mathcal{E}')
\cong B_1(1)^+\cap
(\widetilde{\mathcal{O}}(\partial\Delta^2)^+\cup B_0^+\cup B_1(0)^+)
\to \widetilde{\mathcal{O}}(\partial\Delta^2)^+\cup B_0^+\cup
B_1(0)^+$,
we see that
the inclusion
$\widetilde{\mathcal{O}}(\partial\Delta^2)^+\cup B_0^+\cup
B_1(0)^+\to
\widetilde{\mathcal{O}}(\partial\Delta^2)^+\cup B_0^+\cup
B_1(0)^+\cup B_1(1)^+$ is opposite marked anodyne.
Hence the composition
$\widetilde{\mathcal{O}}(\partial\Delta^2)^+\to
\widetilde{\mathcal{O}}(\partial\Delta^2)^+\cup B_0^+\to\cdots
\to 
\widetilde{\mathcal{O}}(\partial\Delta^2)^+\cup 
B_0^+\cup B_1(0)^+\cup B_1(1)^+=M_2^+$
is also opposite marked anodyne.

Now we assume $n\ge 3$.
In this case we note that all edges of 
$\widetilde{\mathcal{O}}(\Delta^n)$
is included in $\widetilde{\mathcal{O}}(\partial\Delta^n)$.
Let $B_0^+$ be the $n$-simplex in 
$\widetilde{\mathcal{O}}(\Delta^n)^+$ corresponding
to $00\to 01\to\cdots\to 0n$.
The subcomplex $B_0^+\cap 
\widetilde{\mathcal{O}}(\partial\Delta^n)^+$
of $B_0^+$ is isomorphic to
$(\Lambda^n_0)^{\sharp}$ in $B_0^+\cong (\Delta^n)^{\sharp}$.
By Lemma~\ref{lem:marked-anodyne-simplex},
the inclusion 
$(\Lambda^n_0)^{\sharp}\hookrightarrow
(\Delta^n)^{\sharp}$ is opposite marked anodyne.
Taking the pushout of
$(\Lambda^n_0)^{\sharp}\hookrightarrow
(\Delta^n)^{\sharp}$
along the map
$(\Lambda^n_0)^{\sharp}\cong
B_0^+\cap 
\widetilde{\mathcal{O}}(\partial\Delta^n)^+\to
\widetilde{\mathcal{O}}(\partial\Delta^n)^+$,
we see that
the inclusion 
$\widetilde{\mathcal{O}}(\partial\Delta^n)^+
\hookrightarrow
\widetilde{\mathcal{O}}(\partial\Delta^n)^+\cup B_0^+$
is opposite marked anodyne.

For $0\le i<n$,
we let $\mathcal{L}(i)$ be the set of all paths from $ii$ to $0n$ 
in diagram~(\ref{eq:twisted-arroes-diagram}).
To a path $l\in \mathcal{L}(i)$,
we assign a sequence of integers
$J(l)=(j_i,j_{i-1},\ldots,j_1)$
with $i\le j_i\le j_{i-1}\le\cdots\le j_1\le n$
such that $l$ is depicted as 
\[ \begin{array}{ccccccccccccccc}
   & & & & & &          &   &      &   &0j_1&\to&\cdots&\to&0n  \\ 
   & & & & & &          &   &      &   &\uparrow & & & &    \\ 
   & & & &          &   & & &\cdots&\to&1j_1& & & & \\
   & & & &          &   & &\cdots& & & & & & &  \\
   & & & &(i-1)j_i  &\to&\cdots& & & & & & & & \\
   & & & &\uparrow& & & & & & & & & & \\
   ii&\to&\cdots&\to&ij_i& & & & & & & & & & \\
\end{array}\]    
We give $\{J(l)|\, l\in \mathcal{L}(i)\}$
the lexicographic order,
and write $l<l'$ if $J(l)< J(l')$.
This gives rise to a total order on $\mathcal{L}(i)$.
For example,
the path $ii\to \cdots\to 0i\to\cdots\to 0n$
is the smallest 
and the path $ii\to \cdots\to in\to\cdots\to 0n$
is the largest.
For $l\in \mathcal{L}(i)$,
we denote by $B(l)$ the $n$-simplex
in $\widetilde{\mathcal{O}}(\Delta^n)$
corresponding to $l$.
Note that $\mathcal{L}(0)$ consists of a unique
element $l_0$ and that 
$B(l_0)=B_0$.
We set $B_i=\cup_{l\in\mathcal{L}(i)}B(l)$
and $\overline{B}_i=\widetilde{\mathcal{O}}(\partial\Delta^n)\cup
\bigcup_{j=0}^iB_j$.
We shall show that the inclusion 
$\overline{B}_{i-1}^+\hookrightarrow
\overline{B}_i^+$
is opposite marked anodyne
for $0<i<n$.

For $0<i<n$ and $l\in \mathcal{L}(i)$,
we set 
$\overline{B}(l)=\overline{B}_{i-1}\cup
\bigcup_{l'\le l}B(l')$ and
$\overline{B}(l)^{\circ}=
\overline{B}_{i-1}\cup
\bigcup_{l'<l}B(l')$.
It suffices to show that the inclusion
$\overline{B}(l)^{\circ +}\hookrightarrow
\overline{B}(l)^+$ 
is opposite marked anodyne
for all $l\in \mathcal{L}(i)$.

Let $l_i$ be the path $ii\to \cdots\to 0i\to \cdots\to 0n$
for $0<i<n$.
The subcomplex 
$B(l_i)\cap \overline{B}_{i-1}$ of $B(l_i)$
is isomorphic to
$\Lambda^n_i$ of $\Delta^n$. 
The inclusion
$(\Lambda^n_i)^{\flat}\hookrightarrow
(\Delta^n)^{\flat}$ is opposite marked anodyne.
Taking the pushout of
$(\Lambda^n_i)^{\flat}\hookrightarrow (\Delta^n)^{\flat}$
along the map
$(\Lambda^n_i)^{\flat}\cong
B(l_i)^{\flat}\cap \overline{B}_{i-1}^+\to \overline{B}_{i-1}^+$,
we see that the inclusion
$\overline{B}_{i-1}^+\hookrightarrow
\overline{B}_{i-1}^+\cup B(l_i)^+$ 
is opposite marked anodyne.

Let $l_i'$ be the path 
$ii\to \cdots\to in\to\cdots\to 0n$.
We take $l\in\mathcal{L}(i)$ 
such that $l_i<l<l_i'$.
Let $\{\alpha_1,\ldots,\alpha_k\}\
(0<\alpha_1<\ldots<\alpha_k<n)$ be the set of integers
such that 
the sub-path $l(\alpha_t-1)\to l(\alpha_t)\to l(\alpha_t+1)$
of $l$ is depicted as
\[ \begin{array}{ccc}
     a,b & \to & a,b+1\\[1mm]
     \big\uparrow & & \\[1mm]
     a+1,b & & \\
   \end{array}\]
for $t=1,\ldots,k$.
We consider the subcomplex  $B(l)\cap
\overline{B}(l)^{\circ}$ of $B(l)$.
There are two cases. 
(1)\,If the first edge of $l$ is $ii\to (i-1)i$,
then the subcomplex  $B(l)\cap
\overline{B}(l)^{\circ}$ of $B(l)$
is isomorphic to the subcomplex
$\Lambda^n(\alpha_1,\ldots,\alpha_k)$
of $\Delta^n$.
Since the inclusion 
$\Lambda^n(\alpha_1,\ldots,\alpha_k)\hookrightarrow
\Delta^n$ is inner anodyne 
by Lemma~\ref{lemma:fundamental-extension-lemmaI},
the inclusion 
$(\Lambda^n(\alpha_1,\ldots,\alpha_k))^{\flat}\hookrightarrow
(\Delta^n)^{\flat}$ is opposite marked anodyne. 
Taking the pushout of
$(\Lambda^n(\alpha_1,\ldots,\alpha_k))^{\flat}\hookrightarrow
(\Delta^n)^{\flat}$
along the map
$(\Lambda^n(\alpha_1,\ldots,\alpha_k))^{\flat}
\cong B(l)^{\flat}\cap \overline{B}(l)^{\circ +}\to
\overline{B}(l)^{\circ +}$,
we see that the inclusion
$\overline{B}(l)^{\circ +}\hookrightarrow
\overline{B}(l)^+$ is opposite marked anodyne.
(2)\,If the first edge of $l$ is $ii\to i(i+1)$,
then the subcomplex  $B(l)\cap
\overline{B}(l)^{\circ}$ of $B(l)$
is isomorphic to the subcomplex
$\Lambda^n(0,\alpha_1,\ldots,\alpha_k)$
of $\Delta^n$.
Note that $\alpha_1>1$ 
in this case.
Since the inclusion 
$\Lambda^n(0,\alpha_1,\ldots,\alpha_k)\hookrightarrow
\Delta^n$ is inner anodyne 
by Lemma~\ref{lemma:fundamental-extension-lemmaII},
the inclusion 
$(\Lambda^n(0,\alpha_1,\ldots,\alpha_k))^{\flat}\hookrightarrow
(\Delta^n)^{\flat}$ is opposite marked anodyne. 
Taking the pushout of
$(\Lambda^n(0,\alpha_1,\ldots,\alpha_k))^{\flat}\hookrightarrow
(\Delta^n)^{\flat}$
along the map
$(\Lambda^n(0,\alpha_1,\ldots,\alpha_k))^{\flat}
\cong B(l)^{\flat}\cap \overline{B}(l)^{\circ +}\to
\overline{B}(l)^{\circ +}$,
we see that the inclusion
$\overline{B}(l)^{\circ +}\hookrightarrow
\overline{B}(l)^+$ is opposite marked anodyne.

Finally,
we shall show that
$\overline{B}(l_i')^{\circ +}\hookrightarrow
\overline{B}(l_i')^+$ is opposite marked anodyne for $0<i<n$.
The subcomplex 
$B(l_i')\cap \overline{B}(l_i')^{\circ}$
is isomorphic to 
the subcomplex 
$\Lambda^n_0$
of $\Delta^n$.
Note that $ii\to i(i+1)$ is a marked edge,
which corresponds to $\Delta^{\{0,1\}}$
under the isomorphism
$\Lambda^n_0\cong B(l_i')\cap \overline{B}(l_i')^{\circ}$.
The inclusion 
$(\Lambda^n_0,(\Lambda^n_0)\cap\mathcal{E}')
\hookrightarrow
(\Delta^n,\mathcal{E}')$
is opposite marked anodyne,
where $\mathcal{E}'$ is the set of edges of $\Delta^n$
consisting of all degenerate edges together
with $\Delta^{\{0,1\}}$.
Taking the pushout of
$(\Lambda^n_0,(\Lambda^n_0)\cap\mathcal{E}')
\hookrightarrow
(\Delta^n,\mathcal{E}')$ along
the map
$(\Lambda^n_0,(\Lambda^n_0)\cap\mathcal{E}')
\to B(l_i')^+\cap \overline{B}(l_i')^{\circ +} \to
\overline{B}(l_i')^{\circ +}$,
we see that the inclusion
$\overline{B}(l_i')^{\circ +}\hookrightarrow
\overline{B}(l_i')^+$ is opposite marked anodyne.
This completes the proof.
\qed\end{proof}

For $n>0$, we let
\[ \begin{array}{rcl}
     \widetilde{A}&=&(\widetilde{\mathcal{O}}(\Delta^n)\times 
     \Delta^{\{0\}})\cup 
     (\widetilde{\mathcal{O}}(\partial\Delta^n)
     \times \Delta^1),\\[2mm] 
     \widetilde{B}&=&
     \widetilde{A}\cup (M_n\times\Delta^{\{1\}}),\\[2mm]
     \widetilde{C}&=&\widetilde{A}\cup (M_n
     \times\Delta^1)\\
   \end{array}\]
be the subcomplexes
of 
$\widetilde{\mathcal{O}}(\Delta^n)\times\Delta^1$.
We denote by
$(\widetilde{\mathcal{O}}(\Delta^n)\times\Delta^1)^+$
the marked simplicial set
$(\widetilde{\mathcal{O}}(\Delta^n)^+\times
(\Delta^{\{0\}})^{\flat})\cup
(\widetilde{\mathcal{O}}(\Delta^n)^{\flat}\times
(\Delta^1)^{\flat})$.
For a subcomplex $K$ 
of $\widetilde{\mathcal{O}}(\Delta^n)\times\Delta^1$,
we denote by $K^+$
the subcomplex of the marked simplicial set
$(\widetilde{\mathcal{O}}(\Delta^n)\times\Delta^1)^+$
in which the underlying simplicial set is $K$.

\begin{lemma}
\label{lem:opp-mark-anodyne-B-C}
The inclusion
$\widetilde{B}^+\hookrightarrow\widetilde{C}^+$
is an opposite marked anodyne map. 
\end{lemma}

\begin{proof}
We use the notation in the proof of
Lemma~\ref{lem:opposite-makred-anodyne-partial-M}.
Recall that $B_0$ is the $n$-simplex in 
$\widetilde{\mathcal{O}}(\Delta^n)$ corresponding
to $00\to 01\to\cdots\to 0n$.
Since the subcomplex $B_0\cap 
\widetilde{\mathcal{O}}(\partial\Delta^n)$
of $B_0$ is isomorphic to
$\Lambda^n_0$ in $\Delta^n$,
the subcomplex
$\widetilde{B}\cap (B_0\times\Delta^1)$
of $B_0\times\Delta^1$
is isomorphic to 
the subcomplex 
$(\Delta^n\times\partial\Delta^1)\cup 
(\Lambda^n_0\times\Delta^1)$ 
of $\Delta^n\times\Delta^1$.
Since $00\to 01$ is a marked edge of
$\widetilde{\mathcal{O}}(\Delta^n)^+$,
we see that
the inclusion
$\widetilde{B}^+\hookrightarrow
\widetilde{B}^+\cup (B_0\times\Delta^1)^+$
is opposite marked anodyne  
by using Lemma~\ref{lemma:extension_the_first_step}.

We set 
$C_i=\widetilde{B}\cup (\overline{B}_i\times \Delta^1)$.
We shall show that
the inclusion $C_{i-1}^+\hookrightarrow C_i^+$
is opposite marked anodyne
for $0<i<n$.
For this purpose,
it suffices to show that
the inclusion
$C_{i-1}^+\cup (\overline{B}(l)^{\circ}\times\Delta^1)^+
\hookrightarrow
C_{i-1}^+\cup (\overline{B}(l)\times\Delta^1)^+$
is opposite marked anodyne
for all $l\in\mathcal{L}(i)$.

Recall that $l_i$ is the path $ii\to \cdots\to 0i\to \cdots\to 0n$
and that the subcomplex 
$B(l_i)\cap \overline{B}_{i-1}$ of $B(l_i)$
is isomorphic to
$\Lambda^n_i$ of $\Delta^n$.
This implies that
$(B(l_i)\times \Delta^1)\cap C_{i-1}$
is isomorphic to
$(\Lambda^n_i\times \Delta^1)\cup
(\Delta^n\times\partial\Delta^1)$.
The inclusion
$(\Lambda^n_i\times \Delta^1)\cup
(\Delta^n\times\partial\Delta^1)
\hookrightarrow (\Delta^n\times \Delta^1)$
is inner anodyne for $0<i<n$
by \cite[Cor.~2.3.2.4]{Lurie1}.
This implies that
$(\Lambda^n_i\times \Delta^1)^{\flat}\cup
(\Delta^n\times\partial\Delta^1)^{\flat}
\hookrightarrow (\Delta^n\times \Delta^1)^{\flat}$
is opposite marked anodyne. 
Hence we see that
the inclusion $C_{i-1}^+\hookrightarrow
C_{i-1}^+\cup (B(l_i)\times\Delta^1)^+$
is opposite marked anodyne.

We take $l\in\mathcal{L}(i)$ 
such that $l_i<l<l_i'$,
where $l_i'$ is the path 
$ii\to \cdots\to in\to\cdots\to 0n$.
Let $\{\alpha_1,\ldots,\alpha_k\}\
(0<\alpha_1<\ldots<\alpha_k<n)$ be the set of integers
such that 
the sub-path $l(\alpha_t-1)\to l(\alpha_t)\to l(\alpha_t+1)$
is
$a+1,b\to a,b\to a,b+1$
for some $a,b$.

Recall that 
the subcomplex  $B(l)\cap
\overline{B}(l)^{\circ}$ of $B(l)$
is isomorphic to the subcomplex
$\Lambda^n(\alpha_1,\ldots,\alpha_k)$
of $\Delta^n$,
if the first edge of $l$ is $ii\to (i-1)i$.
This implies that
$(B(l)\times \Delta^1)\cap
(C_{i-1}\cup (\overline{B}(l)^{\circ }\times \Delta^1))$
is isomorphic to
$(\Lambda^n(\alpha_1,\ldots,\alpha_k)\times\Delta^1)\cup
(\Delta^n\times\partial\Delta^1)$.
In this case 
the inclusion
$(\Lambda^n(\alpha_1,\ldots,\alpha_k)\times\Delta^1)\cup
(\Delta^n\times\partial\Delta^1)\hookrightarrow
\Delta^n\times\Delta^1$ is 
inner anodyne by Lemma~\ref{lemma:fundamental-extension-lemmaI} 
and \cite[Cor.~2.3.2.4]{Lurie1}.
This implies that
$(\Lambda^n(\alpha_1,\ldots,\alpha_k)\times\Delta^1)^{\flat}\cup
(\Delta^n\times\partial\Delta^1)^{\flat}\hookrightarrow
(\Delta^n\times\Delta^1)^{\flat}$ is 
opposite marked anodyne.
Hence we see that 
$C_{i-1}^+\cup 
(\overline{B}(l)^{\circ}\times\Delta^1)^+
\hookrightarrow C_{i-1}^+\cup 
(\overline{B}(l)\times\Delta^1)^+$ 
is opposite marked anodyne
in this case. 

If the first edge of $l$ is $ii\to i(i+1)$,
then the subcomplex  $B(l)\cap
\overline{B}(l)^{\circ}$ of $B(l)$
is isomorphic to the subcomplex
$\Lambda^n(0,\alpha_1,\ldots,\alpha_k)$
of $\Delta^n$,
where $\alpha_1>1$.
This implies that
$(B(l)\times\Delta^1)\cap
(C_{i-1}\cup (\overline{B}(l)^{\circ }\times(\Delta^1)))$
is isomorphic to
$(\Lambda^n(0,\alpha_1,\ldots,\alpha_k)\times\Delta^1)\cup
(\Delta^n\times \partial\Delta^1)$.
In this case 
the inclusion
$(\Lambda^n(0,\alpha_1,\ldots,\alpha_k)\times\Delta^1)\cup
(\Delta^n\times\partial\Delta^1)\hookrightarrow
\Delta^n\times\Delta^1$ is 
inner anodyne by Lemma~\ref{lemma:fundamental-extension-lemmaII}
and \cite[Cor.~2.3.2.4]{Lurie1}.
This implies that
$(\Lambda^n(0,\alpha_1,\ldots,\alpha_k)\times\Delta^1)^{\flat}\cup
(\Delta^n\times\partial\Delta^1)^{\flat}\hookrightarrow
(\Delta^n\times\Delta^1)^{\flat}$ is 
opposite marked anodyne.
Hence we see that 
$C_{i-1}^+\cup (\overline{B}(l)^{\circ }\times\Delta^1)^+
\hookrightarrow C_{i-1}^+\cup 
(\overline{B}(l)\times\Delta^1)^+$ 
is also opposite marked anodyne
in this case. 

Finally,
we shall show that 
$C_{i-1}^+\cup (\overline{B}(l_i')^{\circ}\times\Delta^1)^+
\hookrightarrow 
C_{i-1}^+\cup (\overline{B}(l_i')\times\Delta^1)^+$
is opposite marked anodyne.
Since the subcomplex 
$B(l_i')\cap \overline{B}(l_i')^{\circ}$
is isomorphic to the subcomplex 
$\Lambda^n_0$
of $\Delta^n$,
the subcomplex $(B(l_i')\times\Delta^1)\cap
(C_{i-1}\cup (\overline{B}(l_i')^{\circ }\times \Delta^1))$
is isomorphic to 
$(\Lambda^n_0\times\Delta^1)\cup
(\Delta^n\times\partial\Delta^1)$.
Since $ii\to i(i+1)$ is a marked edge of 
$\widetilde{\mathcal{O}}(\Delta^n)$,
we see that 
$C_{i-1}^+\cup (\overline{B}(l_i')^{\circ}\times\Delta^1)^+
\hookrightarrow 
C_{i-1}^+\cup (\overline{B}(l_i')\times\Delta^1)^+$
is opposite marked anodyne
by using Lemma~\ref{lemma:extension_the_first_step}.
This completes the proof.
\qed\end{proof}

\subsection{Proof of 
Proposition~\ref{prop:existence-final-object-sections}}
\label{subsec:proof-Prop-1}

In this subsection we give a proof of 
Proposition~\ref{prop:existence-final-object-sections}.
For this purpose,
we show that the map $\pi_X:\mathcal{R}\to RX$
has right lifting property with respect to
the maps $\partial\Delta^n\hookrightarrow\Delta^n$
for $n>0$
if the final vertex $\Delta^{\{n\}}$ goes to
an object of $\mathcal{R}^0$.

Let $p: X\to S$ and $q: Y\to S$ be coCartesian fibrations
over a quasi-category $S$.
Suppose we have a map $G: Y\to X$ over $S$
such that
$G_s$ admits a left adjoint $F_s$ for all $s\in S$.

We recall that
\[  \mathcal{R}= RX\times_{H({\rm Fun}(\Delta^{\{0\}},X))}
       H({\rm Fun}^S(\Delta^1,X))
   \times_{H({\rm Fun}(\Delta^{\{1\}},X))}RY.\]
We have the projection map
$\pi_X: \mathcal{R}\to RX$.

We identify objects of $RX$with objects of $X$.
For $x\in X$ with $s=p(x)$,
we have an object $(x,u_x,F_s(x))$
of $\mathcal{R}$ over $s$,
where $u_x: x\to G_sF_s(x)$
is the unit map of the adjunction
$(F_s,G_s)$ at $x$. 

The following is a key lemma.

\begin{lemma}
\label{lem:key-lemma-lifiting-R}
Suppose we have a commutative diagram
\[ \xymatrix{
    \partial\Delta^n \ar[r]^{f} \ar@{^{(}->}[d]&
    \mathcal{R} \ar[d]^{\pi_X}\\
    \Delta^n \ar@{.>}[ur]\ar[r]^g&  RX\\
}\]
for $n>0$,
where the left vertical arrow is the inclusion.
We put $x=g(\Delta^{\{n\}})$ and $s=p(x)$.
If $f(\Delta^{\{n\}})=(x,u_x,F_s(x))$,
then there exists a dotted arrow $\Delta^n\to \mathcal{R}$
making the whole diagram commutative.
\end{lemma}

\begin{proof}
[Proof of Proposition~\ref{prop:existence-final-object-sections}]
By Lemma~\ref{lem:key-lemma-lifiting-R},
the map $\pi_X^0:\mathcal{R}^0\to RX$
has the right lifting property with respect
to the maps $\partial\Delta^n\hookrightarrow \Delta^n$
for all $n\ge 0$. 
Hence
$\pi_X^0$ is a trivial Kan fibration.
\qed\end{proof}

In order to prove Lemma~\ref{lem:key-lemma-lifiting-R},
we consider the following situation.

Let $h$ be a map $\widetilde{\mathcal{O}}(\Delta^n)\to S$
for $n>0$ such that $h(ii)\to \cdots \to h(0i)$
is a totally degenerate simplex in $S$ for 
all $0\le i\le n$.
We set $\overline{h}=h\pi$,
where $\pi:\widetilde{\mathcal{O}}(\Delta^n)\times\Delta^1
\to\widetilde{\mathcal{O}}(\Delta^n)$ is the projection.

Let $X^{\natural}$ be the marked simplicial set
in which the simplicial set is $X$ and the set of
marked edges consists of all $p$-coCartesian edges. 
Suppose that we have
an $n$-simplex in $RX$ that is represented by
$g:\widetilde{\mathcal{O}}(\Delta^n)\to X$ 
covering $h$.
Note that we can regard $g$ as a map
of marked simplicial sets 
$\widetilde{\mathcal{O}}(\Delta^n)^+\to X^{\natural}$.

Furthermore,
we suppose that we have a map
$\partial\Delta^n\to \mathcal{R}$ that 
is represented by a triple of maps $(g',k,f)$,
where $g': \widetilde{\mathcal{O}}(\partial\Delta^n)\to X$,
$k:\widetilde{\mathcal{O}}(\partial\Delta^n)\to 
{\rm Fun}^S(\Delta^1,X)$,
and 
$f:\widetilde{\mathcal{O}}(\partial\Delta^n)\to Y$.
We assume that $g'$ is the restriction of $g$.
Then the maps $g'$, $k$, and $f$ cover $h$, respectively. 

Let $Y^{\natural}$ be the marked simplicial set
defined in the same way as $X^{\natural}$.
We can regard $f$ as a map of marked simplicial sets
$\widetilde{\mathcal{O}}(\partial\Delta^n)^+\to Y^{\natural}$.
There is an extension $\widetilde{f}$ of $f$
to $M_n^+$ covering $h$
by Lemma~\ref{lem:opposite-makred-anodyne-partial-M}. 

We recall that
$\widetilde{A}$, $\widetilde{B}$, and
$\widetilde{C}$ are subcomplexes of 
$\widetilde{\mathcal{O}}(\Delta^n)\times\Delta^1$ given by
$\widetilde{A}=(\widetilde{\mathcal{O}}(\Delta^n)\times 
     \Delta^{\{0\}})\cup 
     (\widetilde{\mathcal{O}}(\partial\Delta^n)
     \times \Delta^1)$,
$\widetilde{B}=\widetilde{A}\cup (M_n\times\Delta^{\{1\}})$,
$\widetilde{C}=\widetilde{A}\cup (M_n\times\Delta^1)$.
Using the maps $g$, $k$, and $G(\widetilde{f})$, 
we obtain a map of marked simplicial sets
$\widetilde{B}^+\to X^{\natural}$
over $\overline{h}$.
Furthermore,
by Lemma~\ref{lem:opp-mark-anodyne-B-C},
we can extend this map to a map of marked simplicial sets
$w: \widetilde{C}^+\to X^{\natural}$
covering $\overline{h}$.

Let $D$ be the $n$-simplex of 
$\widetilde{\mathcal{O}}(\Delta^n)$ corresponding
to $nn\to \cdots\to 0n$.
By restricting $w$
to $(D\times\Delta^{\{0\}})\cup
(\partial D\times\Delta^1)$,
we obtain a map $v: (D\times\Delta^{\{0\}})\cup
(\partial D\times\Delta^1)\to X_s$,
where $s=h(nn)$.
We denote by $g_D$ the restriction of $g$
to $D$ and by $\widetilde{f}_{\partial D}$ 
the restriction of $\widetilde{f}$
to $\partial D$.
Note that the restriction of $v$
to $D\times\Delta^{\{0\}}$ is identified with $g_D$
and that the restriction of $v$
to $\partial D\times\Delta^{\{1\}}$ is 
$G_s(\widetilde{f}_{\partial D})$.

We would like to have 
maps 
$\widetilde{f}_D: D\to Y_s$
and $\overline{v}:D\times\Delta^1\to X_s$
such that 
$\widetilde{f}_D$ is an extension of 
$\widetilde{f}_{\partial D}$,
$\overline{v}$ is an extension of $v$,
and the restriction of $\overline{v}$ to 
$D\times\Delta^{\{1\}}$
is $G(\widetilde{f}_D)$.
Hence,
in order to prove Lemma~\ref{lem:key-lemma-lifiting-R},
it suffices to prove the following lemma.

\begin{lemma}
Let $L:\mathcal{C}\rightleftarrows\mathcal{D}: R$ be
an adjunction of quasi-categories.
Suppose we have maps
$f: (\Delta^n\times\Delta^{\{0\}})\cup
(\partial\Delta^n\times\Delta^1)
\to \mathcal{C}$ and
$g: \partial\Delta^n\to \mathcal{D}$ for $n>0$
such that 
$Rg= f|_{\partial\Delta^n\times \Delta^{\{1\}}}$.
We put $c=f(\Delta^{\{0\}}\times \Delta^{\{0\}})$
and $d=g(\Delta^{\{0\}})$.
If $g(d)=L(c)$ and
$f(\Delta^{\{0\}}\times \Delta^1)$ is the unit
map $c\to RL(c)$
of the adjunction $(L,R)$ at
$c$, 
then
there exist maps
$F: \Delta^n\times \Delta^1\to \mathcal{C}$
and 
$G: \Delta^n\to \mathcal{D}$
such that
$F$ is an extension of $f$,
$G$ is an extension of $g$,
and
$RG=F|_{\Delta^n\times\Delta^{\{1\}}}$. 
\end{lemma}

\begin{proof}
Let $\pi: \mathcal{M}\to\Delta^1$ be
a map associated to the adjunction $(L,R)$,   
which is a coCartesian fibration and 
a Cartesian fibration.
We may assume that the fibers $\mathcal{M}_{\{0\}}$
and $\mathcal{M}_{\{1\}}$ over $\{0\}$ and $\{1\}$
are isomorphic to $\mathcal{C}$
and $\mathcal{D}$, respectively. 
We regard $f$ as 
a map $(\Delta^n\times\Delta^{\{0\}})\cup
(\partial\Delta^n\times\Delta^1)\to \mathcal{M}_{\{0\}}$
and $g$ as a map
$\partial\Delta^n\to \mathcal{M}_{\{1\}}$.

Since 
$\mathcal{M}\to\Delta^1$ is a Cartesian fibration,
we can extend the map $g$ to 
a map $h:\partial\Delta^n\times\Delta^1\to \mathcal{M}$
such that $h|_{\partial\Delta^n\times\Delta^{\{0\}}}=Rg$,
$h|_{\partial\Delta^n\times\Delta^{\{1\}}}=g$,
and $h(\Delta^{\{i\}}\times\Delta^1)$
is a $\pi$-Cartesian edge over $\Delta^1$ for all 
$i=0,1,\ldots,n$.
By the assumption that 
$Rg=f|_{\partial\Delta^n\times\Delta^{\{1\}}}$,
we obtain a map
$k:\partial\Delta^n\times\Lambda^2_1\to \mathcal{M}$
such that 
$k|_{\partial\Delta^n\times\Delta^{\{0,1\}}}=
f|_{\partial\Delta^n\times\Delta^{\{0,1\}}}$
and 
$k|_{\partial\Delta^n\times\Delta^{\{1,2\}}}= h$.

By the assumptions that
$g(d)=L(c)$ and
$f(\Delta^{\{0\}}\times \Delta^1)$ is the unit
map $c\to RL(c)$,
we have a map $l: \Delta^{\{0\}}\times\Delta^2\to\mathcal{M}$
such that $l|_{\Delta^{\{0\}}\times \Delta^{\{0,1\}}}=
f|_{\Delta^{\{0\}}\times\Delta^1}$,
$l|_{\Delta^{\{0\}}\times \Delta^{\{1,2\}}}=k|_{\Delta^{\{0\}}\times \Delta^{\{1,2\}}}$,
and
$l(\Delta^{\{0\}}\times \Delta^{\{0,2\}})$
is $\pi$-coCartesian.

Hence we obtain a map $k\cup l:
(\partial\Delta^n\times\Lambda^2_1)\cup
(\Delta^{\{0\}}\times\Delta^2)\to\mathcal{M}$.
Let $\sigma:\Delta^n\times\Delta^2\to\Delta^1$
be the projection
$\Delta^n\times\Delta^2\to\Delta^2$ followed
by $s^0:\Delta^2\to\Delta^1$,
where $s^0(\{0\})=s^0(\{1\})=\{0\}$
and $s^0(\{2\})=\{1\}$.
We shall show that $k\cup l$
extends to a map on $\Delta^n\times\Delta^2$
covering $\sigma$.
 
Since $\Lambda^2_1\hookrightarrow\Delta^2$ is inner anodyne,
$(\partial\Delta^n\times \Lambda^2_1)\cup
(\Delta^{\{0\}}\times \Delta^2) 
\to\partial\Delta^n\times\Delta^2$ is also inner anodyne
by \cite[Cor.~2.3.2.4]{Lurie1}.
Hence there is an extension 
$m: \partial\Delta^n\times\Delta^2 \to\mathcal{M}$
of $k\cup l:
(\partial\Delta^n\times \Lambda^2_1)\cup
(\Delta^{\{0\}}\times \Delta^2) \to\mathcal{M}$
covering $\sigma$.

We have the map
$f'|_{\Delta^n\times\Delta^{\{0\}}}\cup 
m|_{\partial\Delta^n\times  \Delta^{\{0,2\}}}:
(\Delta^n\times\Delta^{\{0\}})\cup
(\partial\Delta^n\times\Delta^{\{0,2\}})\to
\mathcal{M}$.
Since $m(\Delta^{\{0\}}\times \Delta^{\{0,2\}})$
is a $\pi$-coCartesian edge over $\Delta^1$,
there is an extension 
$p(0,2): \Delta^n\times\Delta^{\{0,2\}}\to \mathcal{M}$ 
of $f'|_{\Delta^n\times\Delta^{\{0\}}}\cup 
m|_{\partial\Delta^n\times  \Delta^{0,2}}$
covering $\sigma$
by \cite[Prop.~2.4.1.8]{Lurie1}.

We have the map
$p(0,2)|_{\Delta^n\times\Delta^{\{2\}}}\cup
m|_{\partial\Delta^n\times\Delta^{\{1,2\}}}:
(\Delta^n\times\Delta^{\{2\}})\cup
(\partial\Delta^n\times\Delta^{\{1,2\}})\to
\mathcal{M}$.
Since $m(\Delta^{\{n\}}\times\Delta^{\{1,2\}})$
is a $\pi$-Cartesian edge over $\Delta^1$,
there is an extension 
$p(1,2): \Delta^n\times\Delta^{\{1,2\}}\to\mathcal{M}$ 
of 
$P(0,2)|_{\Delta^n\times\Delta^{\{2\}}}\cup
m|_{\partial\Delta^n\times\Delta^{\{1,2\}}}$
covering $\sigma$
by the dual of \cite[Prop.~2.4.1.8]{Lurie1}.

Hence we obtain a map
$q=m\cup p(1,2)\cup p(0,2):
(\partial\Delta^n\times\Delta^2)\cup
(\Delta^n\times\Lambda^2_2)\to\mathcal{M}$
covering $\sigma$.
We note that $q(\Delta^{\{i\}}\times \Delta^{\{1,2\}})$
is $\pi$-Cartesian for all $i=0,1,\ldots,n$.

Let $\mathcal{E}$ be the set of edges of $\Delta^2$
consisting of all degenerate edges together with
$\Delta^{\{1,2\}}$.
We denote by $(\Delta^2)^+$
the marked simplicial set $(\Delta^2,\mathcal{E})$
and by $(\Lambda^2_2)^+$ the marked simplicial set
$(\Lambda^2_2,\mathcal{E}\cap (\Lambda^2_2)_1)$.
The map of marked simplicial sets
$(\Lambda^2_2)^+\to (\Delta^2)^+$ is marked anodyne 
by \cite[Def.~3.1.1.1]{Lurie1}.
This implies that
$(\Delta^n)^{\flat}\times 
(\Lambda^2_2)^+\cup 
(\partial\Delta^n)^{\flat}\times (\Delta^2)^+
\to
(\Delta^n)^{\flat}\times (\Delta^2)^+$
is also marked anodyne
by \cite[Prop.~3.1.2.3]{Lurie1}.

Let $(\Delta^1)^{\sharp}$ be the marked simplicial set
$\Delta^1$ equipped with the set of all edges,
and let $\mathcal{M}^{\natural}$ be the marked simplicial set
$\mathcal{M}$ equipped with the set of all $\pi$-Cartesian edges.
Since $q(\Delta^{\{i\}}\times\Delta^{\{1,2\}})$
is a $\pi$-Cartesian edge for all $i=0,1,\ldots,n$,
we have a map of marked simplicial sets
$q: (\Delta^n)^{\flat}\times (\Lambda^2_2)^+
    \cup (\partial\Delta^n)^{\flat}\times (\Delta^2)^+
    \to\mathcal{M}^{\natural}$.
We consider the following commutative diagram
of marked simplicial sets
\[ \xymatrix{
    (\Delta^n)^{\flat}\times (\Lambda^2_2)^+
    \cup (\partial\Delta^n)^{\flat}\times (\Delta^2)^+
    \ar[r]\ar@{^{(}->}[d] &
    \mathcal{M}^{\natural} \ar[d]^{\pi}\\
    (\Delta^n)^{\flat}\times (\Delta^2)^+ 
    \ar[r]^{\sigma} \ar@{.>}[ur]^r&
    (\Delta^1)^{\sharp},\\
}\]
where the upper horizontal arrow is $q$.
Since the left vertical arrow is marked anodyne,
there is a dotted arrow $r$ which makes the whole diagram
commutative by \cite[Prop.~3.1.1.6]{Lurie1}.  
The proof is completed by
setting $F=r|_{\Delta^n\times\Delta^{\{0,1\}}}$ and
$G=r|_{\Delta^n\times\Delta^{\{2\}}}$.
\qed
\end{proof}



\end{document}